\documentclass[12pt,a4paper,pdftex,draft]{amsart}
\usepackage[margin=2.6cm]{geometry}
\usepackage[T1]{fontenc}
\usepackage[utf8]{inputenc}
\usepackage{lmodern,amssymb}
\usepackage{tikz-cd}
\numberwithin{equation}{section}

\newtheorem{thm}{Theorem}[section]
\newtheorem{prop}[thm]{Proposition}
\newtheorem{lem}[thm]{Lemma}
\newtheorem{cor}[thm]{Corollary}

\theoremstyle{definition}
\newtheorem{defn}[thm]{Definition}
\newtheorem{assump}[thm]{Assumption}
\theoremstyle{remark}
\newtheorem{rem}[thm]{Remark}

\DeclareMathOperator{\grk}{grk}
\DeclareMathOperator{\Hom}{Hom}
\DeclareMathOperator{\End}{End}
\DeclareMathOperator{\Ima}{Im}
\DeclareMathOperator{\supp}{supp}

\DeclareMathOperator{\ch}{ch}
\DeclareMathOperator{\Aut}{Aut}
\DeclareMathOperator{\Ker}{Ker}
\DeclareMathOperator{\Frac}{Frac}
\DeclareMathOperator{\Stab}{Stab}
\DeclareMathOperator{\Proj}{Proj}
\DeclareMathOperator{\Rep}{Rep}
\DeclareMathOperator{\Mod}{Mod}

\newcommand*{\Coeff}{\mathbb{K}}
\newcommand*{\Z}{\mathbb{Z}}
\newcommand{\Sbimod}{\mathcal{S}\mathrm{Bimod}}
\newcommand*{\R}{\mathbb{R}}
\newcommand*{\aff}{\mathrm{aff}}
\newcommand*{\property}[1]{{\normalfont(#1)}}
\DeclareMathOperator{\id}{id}
\newcommand*{\AJS}{\mathrm{AJS}}
\DeclareMathOperator{\rank}{rank}
\newcommand*{\integer}{\mathrm{int}}

\title{A Hecke action on $G_1T$-modules}
\author{Noriyuki Abe}
\address{Graduate School of Mathematical Sciences, the University of Tokyo, 3-8-1 Komaba, Meguro-ku, Tokyo 153-8914, Japan.}
\email{abenori@ms.u-tokyo.ac.jp}

\begin{document}
\begin{abstract}
We construct an action of the Hecke category on the principal block $\mathrm{Rep}_0(G_1T)$ of $G_1T$-modules where $G$ is a connected reductive group over an algebraically closed field of characteristic $p > 0$, $T$ a maximal torus of $G$ and $G_1$ the Frobenius kernel of $G$.
To define it, we define a new category with a Hecke action which is equivalent to the combinatorial category defined by Andersen-Jantzen-Soergel.
\end{abstract}

\maketitle

\section{Introduction}
Let $G$ be a connected reductive group over an algebraically closed field $\Coeff$ of characteristic $p > 0$.
One of the most important problems in representation theory is to describe the characters of irreducible representations.
In the case of algebraic representations of $G$, Lusztig gave a conjectural formula on the characters of irreducible representations of $G$ in terms of Kazhdan-Lusztig polynomials of the affine Weyl group for $p > h$ where $h$ is the Coxeter number. 
Thanks to the works of Kazhdan-Lusztig~\cite{MR1186962,MR1239506,MR1239507}, Kashiwara-Tanisaki~\cite{MR1317626,MR1408544} and Andersen-Jantzen-Soergel~\cite{MR1272539}, this is proved for $p$ large enough.
An explicit bound on $p$ is known by Fiebig~\cite{MR2999126}.

However, as Williamson~\cite{MR3671935} showed, Lustzig's conjecture is not true for many $p$.
Therefore we need a new approach for such $p$.
Riche-Williamson~\cite{MR3805034} gave it and we explain their approach.
Assume that $p > h$.
Let $\Rep_0(G)$ be the principal block of the category of algebraic representations of $G$.
For each affine simple reflection $s$, we have the wall-crossing functor $\theta_s\colon \Rep_0(G)\to \Rep_0(G)$.
The Grothendieck group of $\Rep_0(G)$ is isomorphic to the anti-spherical quotient of the group algebra of the affine Weyl group.
Here the structure of a representation of the affine Weyl group is given by $[M](s + 1) = [\theta_s(M)]$ for $M\in \Rep_0(G)$.
Riche-Williamson~\cite{MR3805034} conjectured the existence of the categorification of this anti-spherical quotient.
More precisely, they conjectured that there is an action of $\mathcal{D}$ on $\Rep_0(G)$ where $\mathcal{D}$ is the diagrammatic Hecke category defined by Elias-Williamson~\cite{MR3555156}.
Assuming this conjecture, they proved that the anti-spherical quotient of $\mathcal{D}$ is a graded version of the category of tilting modules in $\Rep_0(G)$.
In particular, their result gives a character formula for indecomposable tilting modules in terms of $p$-Kazhdan-Lusztig polynomials.
Recently this character formula was proved by Achar-Makisumi-Riche-Williamson~\cite{MR3868004} when $p > h$ and for any $p$ by Riche-Williamson~\cite{arXiv:2003.08522}.
We note that if $p\ge 2h - 2$ then a character formula for indecomposable tilting modules implies a character formula for irreducible modules.
We also remark that Sobaje~\cite{MR4092982} gave an algorithm to calculate the character of irreducible modules by the character of indecomposable tilting modules.

Achar-Makisumi-Riche-Williamson also proved a big part of the conjecture, but not a full statement.
In the case of $G = \mathrm{GL}_n$, the original conjecture is proved by Riche-Williamson~\cite{MR3805034}.
Recently, the conjecture is proved by Bezrukavnikov-Riche~\cite{arXiv:2009.10587}.
In this paper, we consider the $G_1T$-version of this conjecture where $T\subset G$ is a maximal torus and $G_1$ is the Frobenius kernel of $G$.
Namely, we define an action of the category $\mathcal{D}$ on the principal block of $G_1T$-modules.

Next, we state our main theorem.
We remark that we have an object $B_s\in \mathcal{D}$ for any affine simple reflection $s$ (see the next subsection for the precise definition).
Let $\Rep_0(G_1T)$ be the principal block of $G_1T$-modules.
\begin{thm}[Theorem~\ref{thm:Hecke action on Representations}]\label{thm:Main Theorem}
The category $\mathcal{D}$ acts on $\Rep_0(G_1T)$ where $B_s\in \mathcal{D}$ acts as the wall-crossing functor for any affine simple reflection $s$.
\end{thm}
Kaneda (private communication) proved this theorem for $\mathrm{GL}_n$ using the arguments of Riche-Williamson~\cite{MR3805034}. 

Let $X^\vee$ be the cocharacter group of $T$ and set $X^\vee_{\Coeff} = X^\vee\otimes_{\Z}\Coeff$.
Put $S = S(X^\vee_{\Coeff})$.
This is a graded algebra via $\deg(X^{\vee}_{\Coeff}) = 2$.
Andersen-Jantzen-Soergel defined a combinatorial category $\mathcal{K}_{\AJS,P}$.
This category is an $S$-linear category with a grading.
We define a category $\Coeff\otimes_{S}\mathcal{K}_{\AJS}^{\mathrm{f}}$ with the same objects as $\mathcal{K}_{\AJS}$, however the space of morphisms is defined as $\Hom_{\Coeff\otimes_{S}\mathcal{K}_{\AJS}}(M,N) = \Coeff\otimes_{S}\bigoplus_{i\in\Z}\Hom_{\mathcal{K}_{\AJS}}(M,N(i))$ where $N(i)$ denotes the grading shift.
Let $\Proj(\Rep_0(G_1T))$ be the category of projective objects in $\Rep_0(G_1T)$.
Andersen-Jantzen-Soergel constructed a functor $\mathcal{V}\colon \Proj(\Rep_0(G_1T))\to \Coeff\otimes_{S}\mathcal{K}_{\AJS}^{\mathrm{f}}$ and proved that it is fully-faithful.
They also determined the essential image of $\mathcal{V}$ and using this functor they proved Lusztig's conjecture for large $p$.

In order to obtain an action of $\mathcal{D}$ on $\Rep_0(G_1T)$, it is sufficient to define an action on $\Proj(\Rep_0(G_1T))$ (see \ref{subsec:Representation Theory}).
Therefore, by the results of Andersen-Jantzen-Soergel, it is sufficient to construct the action of $\mathcal{D}$ on the essential image of $\mathcal{V}$.
The main obstructions to do it are the following.
\begin{enumerate}
\item Elias-Williamson defined $\mathcal{D}$ via generators and relations.
Since the relations are very complicated, it is hard to check that the action is well-defined.
\item The category $\mathcal{K}_{\AJS}$ contains only `local' information. Hence, it is difficult to construct the action directly.
\end{enumerate}

\subsection{The category $\Sbimod$}
We use the category $\Sbimod$~\cite{arXiv:1901.02336_accepted} instead of the category $\mathcal{D}$.
The category $\Sbimod$ is equivalent to the category $\mathcal{D}$.
We recall the definition of $\Sbimod$.
Let $W_{\aff}$ be the affine Weyl group attached to $G$.
An object we consider is a graded $S$-bimodule with a decomposition $M\otimes_{S}\Frac(S) = \bigoplus_{x\in W_{\aff}}M_x^{\Frac(S)}$ such that $mf = \overline{x}(f)m$ for $f\in S$ and $m\in M_x^{\Frac(S)}$.
Here $\overline{x}$ is the image of $x$ in the finite Weyl group.
For such objects $M$ and $N$, we have the tensor product $M\otimes N = M\otimes_{S}N$ with the decomposition $(M\otimes N)\otimes_{S}\Frac(S) = \bigoplus_{x\in W_{\aff}}(M\otimes N)_x^{\Frac(S)}$ where $(M\otimes N)_x^{\Frac(S)} = \bigoplus_{yz = x}M_{y}^{\Frac(S)}\otimes_{\Frac(S)} N_{z}^{\Frac(S)}$.
A homomorphism $M\to N$ is a degree zero $S$-bimodule homomorphism which sends $M_x^{\Frac(S)}$ to $N_x^{\Frac(S)}$ for any $x\in W_{\aff}$.

Let $X$ be the character group of $T$.
An alcove is a connected component of $X\otimes_{\Z}\R\setminus \bigcup_t H_t$ where $t$ runs through the affine reflections in $W_{\aff}$ and $H_t$ is the fixed hyperplane of $t$.
We fix an alcove $A_0$ and let $S_{\aff}$ be the reflections with respect to the walls of $A_0$.
Then $(W_{\aff},S_{\aff})$ is a Coxeter system.
For each $s\in S_{\aff}$, put $S^s = \{f\in S\mid s(f) = f\}$.
Then the $S$-bimodule $S\otimes_{S^s}S(1)$ has the unique decomposition as described above such that $(S\otimes_{S^s}S(1))_w^{\Frac(S)}\ne 0$ only when $w = e,s$.
We denote this object by $B_s$.
Now $\Sbimod$ consists of the objects $M$ which is a direct summand of a direct sum of objects of a form $B_{s_1}\otimes\dotsm\otimes B_{s_l}(n)$ where $s_1,\dots,s_l\in S_{\aff}$ and $n\in\Z$.
It is proved in \cite{arXiv:1901.02336_accepted} that the category $\Sbimod$ is equivalent to the diagrammatic Hecke category defined by Elias-Williamson.
As showed in \cite{MR3555156,arXiv:1901.02336_accepted}, this gives a categorification of the Hecke algebra of affine Weyl group, namely the split Grothendieck group of $\Sbimod$ is isomorphic to the Hecke algebra.

\subsection{Another combinatorial category}
We also give another realization of the category of Andersen-Jantzen-Soergel $\mathcal{K}_{\AJS}$~\cite{MR1272539}.
As in \cite{MR591724}, we use the combinatorics of alcoves to define the category.
Let $\mathcal{A}$ be the set of alcoves.
We fix a positive system $\Delta^+$ of the root system $\Delta$ of $G$.
Then this defines an order on $\mathcal{A}$ \cite{MR591724}.
Recall that we have fixed $A_0\in \mathcal{A}$.
The action of $W_{\aff}$ on $X\otimes_{\Z}\R$ induces the action of $W_{\aff}$ on $\mathcal{A}$.
The map $w\mapsto w(A_0)$ gives a bijection $W_{\aff}\to \mathcal{A}$.

Set $S^{\emptyset} = S[(\alpha^\vee)^{-1}\mid \alpha\in\Delta]$.
We define the category $\widetilde{\mathcal{K}}'$ as follows: 
An object of $\widetilde{\mathcal{K}}'$ is a graded $S$-bimodule $M$ with a decomposition $S^\emptyset\otimes_{S}M = \bigoplus_{A\in \mathcal{A}}M_A^{\emptyset}$ such that $mf = \overline{x}(f)m$ for  $m\in M_{A}^{\emptyset}$, $f\in S^{\emptyset}$, $x\in W_{\aff}$ such that $A = x(A_0)$ and $\overline{x}$ the image of $x$ in the finite Weyl group.
A morphism $f\colon M\to N$ is a degree zero $S$-bimodule homomorphism such that $f(M_A^{\emptyset})\subset \bigoplus_{A' \ge A}N_{A'}^{\emptyset}$.
We will also define some subcategories of $\widetilde{\mathcal{K}}'$.
Especially the category denoted by $\widetilde{\mathcal{K}}_P$ plays an important role in our construction.
Since it is technical, we do not say anything about its definitions in the introduction.
We only note that for each $A\in \mathcal{A}$ the module $M_{\{A\}} = (M\cap \bigoplus_{A'\ge A}M_{A'}^{\emptyset})/(M\cap \bigoplus_{A' > A}M_{A'}^{\emptyset})$ is graded free for $M\in\widetilde{\mathcal{K}}_P$.

We define an action $B\in \Sbimod$ on $\widetilde{\mathcal{K}}'$ as follows.
Note that we have a submodule $B^{\emptyset}_{x}\subset S^{\emptyset}\otimes_{S}B$ such that $B^{\emptyset}_x\otimes_{S^{\emptyset}}\Frac(S) = B_x^{\Frac(S)}$.
Let $M\in \widetilde{\mathcal{K}}'$.
Then we define $M*B$ by $M*B = M\otimes_{S}B$ as a graded $S$-bimodule and $(M*B)_{w(A_0)}^{\emptyset} = \bigoplus_{x\in W_{\aff}}M_{wx^{-1}(A_0)}^{\emptyset}\otimes_{S^{\emptyset}}B_{x}^{\emptyset}$ for $w\in W_{\aff}$.
We can prove the action preserves the subcategory $\widetilde{\mathcal{K}}_{P}$ (Proposition~\ref{prop:K_Delta is stable under the Hecke actions}).
Therefore the split Grothendieck group $[\widetilde{\mathcal{K}}_{P}]$ of $\widetilde{\mathcal{K}}_{P}$ has a structure of $[\Sbimod]$-module defined by $[M][B] = [M*B]$.
Hence $[\widetilde{\mathcal{K}}_{P}]$ is a module of the Hecke algebra.
This category satisfies the following.

\begin{thm}[Theorem~\ref{thm:indecomposables in widetilde(K)}, \ref{thm:categorification for tilde(K)}]\label{thm:theorem in introduction, categorification for tilde(K)}
We have the following.
\begin{enumerate}
\item For each $A\in \mathcal{A}$ we have an indecomposable module $Q(A)\in \widetilde{\mathcal{K}}_{P}$ such that $Q(A)_{\{A\}}\simeq S$ and $Q(A)_{\{A'\}}\ne 0$ implies $A'\ge A$.
\item Any object in $\widetilde{\mathcal{K}}_{P}$ is isomorphic to a direct sum of $Q(A)(n)$ for $A\in \mathcal{A}$ and $n\in \Z$.
\item The split Grothendieck group $[\widetilde{\mathcal{K}}_P]$ is isomorphic to a certain submodule $\mathcal{P}^0$ of the periodic Hecke module. (The submodule was introduced in \cite{MR591724}.)
\end{enumerate}
\end{thm}

\subsection{A relation with a work of Fiebig-Lanini}
Fiebig-Lanini~\cite{arXiv:1504.01699} had a similar work (earlier than this work) and defined a certain category.
Logically, results in this paper does not depend on their work.
However, in the proof in this paper, we borrow many ideas from their work.
Moreover, in subsection~\ref{subsec:A relation with a work of Fiebig-Lanini}, we prove that our category $\widetilde{\mathcal{K}}_P$ is equivalent to the category of Fiebig-Lanini.
The author thinks it is possible to establish the theory on top of the theory of Fiebig-Lanini, but the existence of a Hecke action does not soon follow from their theory.

\subsection{Relations with representation theory}
The category $\widetilde{\mathcal{K}}_{P}$ is not the category we really need.
We modify this category as follows.
Objects in $\mathcal{K}_{P}$ are the same as those in $\widetilde{\mathcal{K}}_{P}$ and the space of homomorphisms is defined by
\[
\Hom_{\mathcal{K}_P}(M,N) = \Hom_{\widetilde{\mathcal{K}}_{P}}(M,N)/\{\varphi\colon M\to N\mid \varphi(M_{A}^{\emptyset})\subset \bigoplus_{A' > A}N_{A'}^{\emptyset}\}.
\]
We can prove that the action of $B\in \Sbimod$ on $\mathcal{K}_{P}$ is well-defined.
\begin{thm}[Proposition~\ref{prop:indecomposable is indecomposable}, Theorem~\ref{thm:categorification for K}]
We have the following.
\begin{enumerate}
\item The object $Q(A)$ is also indecomposable as an object of $\mathcal{K}_{P}$.
\item We have $[\mathcal{K}_{P}]\simeq [\widetilde{\mathcal{K}}_P]$.
Hence $[\mathcal{K}_{P}]$ is also isomorphic to $\mathcal{P}^0$.
\end{enumerate}
\end{thm}

We also define a functor $\mathcal{F}\colon \mathcal{K}_{P}\to \mathcal{K}_{\AJS}$.
Recall that we have the wall-crossing functor $\vartheta_s\colon \mathcal{K}_{\AJS}\to \mathcal{K}_{\AJS}$ for each $s\in S_{\aff}$.
\begin{thm}[Proposition~\ref{prop:compativility of translation functors}, \ref{prop:fully-faithfulness}]
We have the following.
\begin{enumerate}
\item We have $\mathcal{F}(M*B_s)\simeq \vartheta_s(\mathcal{F}(M))$.
\item The functor $\mathcal{F}$ is fully-faithful.
\end{enumerate}
\end{thm}

Let $\mathcal{K}_{\AJS,P}$ be the essential image of $\mathcal{F}$.
One of the main results in \cite{MR1272539} says that $\Coeff\otimes_{S}\mathcal{K}^{\mathrm{f}}_{\AJS,P}\simeq \Proj(\Rep_0(G_1T))$.
Since the action of $\Sbimod$ on $\mathcal{K}_P\simeq \mathcal{K}_{\AJS,P}$ gives an action on $\Coeff\otimes_{S}\mathcal{K}_{\AJS,P}^{\mathrm{f}}$, we now get the action of $\Sbimod$ on $\Proj(\Rep_0(G_1T))$.
We can extend this action to $\Rep_0(G_1T)$, see \ref{subsec:Representation Theory}.

Let $A_0$ be the alcove containing $\rho/p$ where $\rho$ is the half sum of positive roots.
We have an equivalence $\Coeff\otimes_{S}\mathcal{K}_{P}^{\mathrm{f}}\simeq \Coeff\otimes_{S}\mathcal{K}_{\AJS,P}^{\mathrm{f}}\simeq \Proj(\Rep_0(G_1T))$ and $Q(A)$ corresponds to $P(\lambda_A)$ where $\lambda_{w(A_0)} = pw(\rho/p) - \rho$ for $w\in W_{\aff}$ and $P(\lambda_A)$ is the projective cover of the irreducible representation with the highest weight $\lambda_A$.
Let $Z(\mu)\in \Rep(G_1T)$ be the baby Verma module with the highest weight $\mu$ and $(P(\lambda):Z(\mu))$ the multiplicity of $Z(\mu)$ in a Verma flag of $P(\lambda)$.
By the constructions, we have the following.
\begin{thm}[Corollary~\ref{cor:multiplicity}]
The multiplicity $(P(\lambda_A):Z(\lambda_{A'}))$ is equal to the rank of $Q(A)_{\{A'\}}$.
\end{thm}

In the final part, we discuss Lusztig's conjecture on irreducible characters of algebraic representations.
We give a proof of the conjecture based on the theory developed in this paper.

\subsection*{Acknowledgment}
The question treated in this paper was asked by Masaharu Kaneda.
The author had many helpful discussions with him.
He also thank the referees giveng helpful comments and pointing out errors.
The author was supported by JSPS KAKENHI Grant Number 18H01107.

\section{Our combinatorial category}\label{sec:Our combinatorial category}
We use a slightly different notation from the introduction.
In particular, we do not fix the alcove $A_0$.
So we distinguished two actions (from the right and left) of $W_{\aff}$ on $\mathcal{A}$ as in \cite{MR591724}.

\subsection{Notation}
Let $(X,\Delta,X^\vee,\Delta^\vee)$ be a root datum.
Let $\mathcal{A}$ the set of alcoves, namely the set of connected components of $X_\R\setminus\bigcup_{\alpha\in\Delta,n\in\Z}\{\lambda\in X_\R\mid \langle \lambda,\alpha^\vee\rangle = n\}$ where $X_{\R} = X\otimes_{\Z}\R$.
Let $W_{\mathrm{f}}$ be the finite Weyl group and $W'_\aff = W_{\mathrm{f}}\ltimes \Z\Delta$ the affine Weyl group with the natural surjective homomorphism $W'_\aff\to W_{\mathrm{f}}$.
For each $\alpha\in\Delta$ and $n\in\Z$, let $s_{\alpha,n}\colon X\to X$ be the reflection with respect to $\{\lambda\in X_\R\mid \langle \lambda,\alpha^\vee\rangle = n\}$.
As in \cite{MR591724}, let $S_\aff$ be the set of $W'_\aff$-orbits on the set of faces.
Then for each $s\in S_\aff$ and $A\in \mathcal{A}$, we set $As$ as the alcove $\ne A$ which has a common face of type $s$ with $A$.
The subgroup of $\Aut(\mathcal{A})$ (permutations of elements in $\mathcal{A}$) generated by $S_\aff$ is denoted by $W_\aff$.
Then $(W_\aff,S_\aff)$ is a Coxeter system isomorphic to the affine Weyl group.
The Bruhat order on $W_{\aff}$ is denoted by $\ge$.
The group $W_\aff$ acts on $\mathcal{A}$ from the right.

We give related notation and also some facts.
If we fix an alcove $A_0$, then $W'_\aff\simeq \mathcal{A}$ via $w\mapsto wA_0$ and $W'_\aff$ acts on $\mathcal{A}$ by $(w(A_0))x = wx(A_0)$.
This gives an isomorphism $W'_\aff\simeq W_\aff$.
The facts stated below are obvious from this description.

Let $\Lambda$ be the set of maps $\mathcal{A}\to X$ such that $\lambda(xA) = \overline{x}\lambda(A)$ for any $x\in W'_\aff$ and $A\in \mathcal{A}$ where $\overline{x}\in W_{\mathrm{f}}$ is the image of $x$.
We write $\lambda_A = \lambda(A)$ for $\lambda\in \Lambda$ and $A\in \mathcal{A}$.
For each $A\in \mathcal{A}$, $\lambda\mapsto \lambda_A$ gives an isomorphism $\Lambda\xrightarrow{\sim}X$ and the inverse of this isomorphism is denoted by $\nu\mapsto \nu^A$.
The group $W_\aff$ acts on $\Lambda$ by $(x(\lambda))(A) = \lambda(Ax)$.

Let $\Lambda_\aff$ be the set of $\lambda\in\Lambda$ such that $\lambda_A\in \Z\Delta$ for any, or equivalently, some $A\in \mathcal{A}$.
For $\lambda\in \Lambda_\aff$ and $A\in \mathcal{A}$, we define $A\lambda = A + \lambda_A$.
Then for $\lambda_1,\lambda_2\in \Lambda_\aff$, $(A\lambda_1)\lambda_2 = (A + (\lambda_1)_A)\lambda_2 = A + (\lambda_1)_A + (\lambda_2)_{A + (\lambda_1)_A}$.
Since elements in $\Lambda$ are constant on $\Z\Delta$-orbits, we have $(\lambda_2)_{A + (\lambda_1)_A} = (\lambda_2)_A$.
Hence $(A\lambda_1)\lambda_2 = A + (\lambda_1 + \lambda_2)_A$, namely $(A,\lambda)\mapsto A\lambda$ gives an action of $\Lambda_\aff$ on $\mathcal{A}$.
Therefore we get $\Lambda_\aff\hookrightarrow \Aut(\mathcal{A})$ and the image is contained in $W_\aff$.
We regard $\Lambda_\aff\subset W_\aff$.

Let $\lambda\in \Lambda$ and $A,A'\in \mathcal{A}$ and assume that $A,A'$ are in the same $\Lambda_\aff$-orbit.
Namely there exists $\mu\in \Lambda_\aff$ such that $A = A' \mu = A' + \mu_{A'}$.
Since elements in $\Lambda$ are constant on $\Z\Delta$-orbits, we get $\lambda_{A'} = \lambda_A$.
Namely the isomorphism $\lambda\mapsto \lambda_A$ only depends on $\Lambda_{\aff}$-orbit in $\mathcal{A}$.
Hence we also denote the isomorphism by $\lambda\mapsto\lambda_\Omega$ where $\Omega\in \mathcal{A}/\Lambda_\aff$.
The inverse is denoted by $\lambda\mapsto \lambda^\Omega$.
The $\Lambda_\aff$-orbit through $A$ is equal to $\{A + \lambda\mid \lambda\in\Z\Delta\}$.
We denote this by $A + \Z\Delta$.

The following lemma is obvious from the definitions.
\begin{lem}
Let $\lambda\in \Lambda$, $\nu\in X$, $x\in W_\aff$, $y\in W'_\aff$ and $A\in \mathcal{A}$.
\begin{enumerate}
\item $x(\lambda)_A = \lambda_{Ax}$.
\item $y(\lambda_A) = \lambda_{yA}$.
\item $\nu^{A} = x(\nu^{Ax})$.
\item $\nu^A = y(\nu)^{yA}$.
\end{enumerate}
\end{lem}

Fix a positive system $\Delta^+ \subset \Delta$.
Let $\alpha\in\Delta^+$ and $n\in\Z$.
We say $A\le s_{\alpha,n}(A)$ if for any $\lambda\in A$ we have $\langle \lambda,\alpha^\vee\rangle < n$.
The generic Bruhat order $\le$ on $\mathcal{A}$ is the partial order generated by the relations $A\le s_{\alpha,n}(A)$.
The following lemma is obvious from the definition.
\begin{lem}\label{lem:order as a vector in A}
Let $A\in \mathcal{A}$, $w\in W'_{\aff}$ and $\lambda$ is in the closure of $A$.
If $A \le w(A)$, then $w(\lambda) - \lambda\in\R_{\ge 0}\Delta^+$.
\end{lem}

\begin{lem}\label{lem:homeo on each orbit}
Let $A,A'$ such that $A + \lambda = A'$ for $\lambda\in \Z\Delta$.
Then $A\le A'$ if and only if $\lambda\in \Z_{\ge 0}\Delta^+$.
\end{lem}
\begin{proof}
We assume $\lambda\in\Z_{\ge 0}\Delta^+$ and prove that $A\le A'$.
We may assume $\lambda = \alpha\in\Delta^+$.
Take $n\in \Z$ such that $n - 1  < \langle \mu,\alpha^\vee\rangle < n$ for any $\mu\in A$.
For $\mu\in A$, we have $\langle s_{\alpha,n}(\mu),\alpha^\vee\rangle = \langle\mu - (\langle\mu,\alpha^\vee\rangle - n)\alpha,\alpha^\vee \rangle = 2n - \langle \mu,\alpha^\vee\rangle$.
Hence $n < \langle s_{\alpha,n}(\mu),\alpha^\vee\rangle < n + 1$.
Therefore $A \le s_{\alpha,n}(A)\le s_{\alpha,n + 1}s_{\alpha,n}(A) = A + \alpha$.

On the other hand, assume that $A\le A'$.
Take $\nu\in A$.
Then by Lemma~\ref{lem:order as a vector in A}, we have $(\nu + \lambda) - \nu \in\R_{\ge 0}\Delta^+$.
Hence $\lambda\in \R_{\ge 0}\Delta^+$.
Since $\lambda\in\Z\Delta$, we get $\lambda\in\Z_{\ge 0}\Delta^+$.
\end{proof}
A subset $I\subset \mathcal{A}$ is called open (resp.\ closed) if $A\in I$, $A'\le A$ (resp.\ $A'\ge A$) implies $A'\in I$.
This defines a topology on $\mathcal{A}$.
The following lemma is an immediate consequence of the previous lemma and it plays an important role throughout this paper.
\begin{lem}
For each $\Omega\in \mathcal{A}/\Lambda_\aff$ and $x\in W_{\aff}$, the map $x\colon \Omega\to \Omega x$ preserves the order.
\end{lem}

For $A,A'\in \mathcal{A}$, set $[A,A'] = \{A''\in \mathcal{A}\mid A\le A''\le A'\}$.
For $\alpha\in \Delta^+$ and $A\in \mathcal{A}$, take $n\in \Z$ such that $n - 1 < \langle \lambda,\alpha^\vee\rangle < n$ for all $\lambda\in A$ and define $\alpha\uparrow A = s_{\alpha,n}(A)$.
By the definition, $A\le \alpha\uparrow A$.
We define $\alpha\downarrow A$ as the unique element such that $\alpha\uparrow (\alpha\downarrow A) = A$.

Let $M = \bigoplus_i M^i$ be a graded module.
We define $M(k)$ by $M(k)^i = M^{i + k}$.
A graded $S$-module $M$ is called graded free if it is isomorphic to $\bigoplus_i S(n_i)$ where $n_1,\dots,n_r\in\Z$.
(In this paper, graded free means graded free of finite rank.)
We set $\grk(M) = \sum_{i}v^{n_i}\in \Z[v,v^{-1}]$ where $v$ is the indeterminate.

\subsection{The categories}
Fix a noetherian integral domain $\Coeff$.
We define $\Lambda^\vee$ using $X^\vee$ exactly in the same way as we defined $\Lambda$ using $X$.
We put $\Lambda^\vee_{\Coeff} = \Lambda^\vee\otimes_{\Z}\Coeff$, $X^\vee_{\Coeff} = X^\vee\otimes_{\Z}\Coeff$ and $R = S(\Lambda^\vee_{\Coeff})$.
The algebra $R$ is equipped with a grading such that $\deg(\Lambda^\vee_{\Coeff}) = 2$.
\begin{assump}
In the rest of this section, we assume the following.
\begin{enumerate}
\item  We have $2\in \Coeff^\times$ and any $\alpha^\vee\ne\beta^\vee\in(\Delta^\vee)^+$ are linearly independent in $X^\vee_{\Coeff/\mathfrak{m}}$ for any maximal ideal $\mathfrak{m}\subset \Coeff$. This condition is called the BKM condition.
\item The torsion primes of the root system $(X^\vee,\Delta^\vee,X,\Delta)$ are invertible in $\Coeff$.
\end{enumerate}
\end{assump}

\begin{lem}\label{lem:faithful on finite Weyl group}
The representation $X^\vee_{\Coeff}$ of $W_{\mathrm{f}}$ is faithful.
\end{lem}
\begin{proof}
If $w\in W_{\mathrm{f}}$ fixes any element in $X^\vee_{\Coeff}$, it fixes any image of $\alpha\in\Delta$.
By the assumption, $\Delta^\vee\to X^\vee_{\Coeff}$ is injective.
Therefore $w$ fixes any coroot.
Hence $w$ is identity.
\end{proof}

The image of $\alpha^\vee\in \Delta^\vee$ in $X^\vee_{\Coeff}$ is denoted by the same letter.
We also put $S = S(X^\vee_{\Coeff})$.
We give a grading to $S$ via $\deg(X^\vee_{\Coeff}) = 2$.
Let $S_0$ be a commutative flat graded $S$-algebra.
Set $S^\emptyset = S[(\alpha^\vee)^{-1}\mid \alpha\in\Delta]$.
For an $S$-module $M$, we denote $M^\emptyset = S^\emptyset\otimes M$.
Let $S_0$ be a graded $S$-algebra.
We consider the category $\widetilde{\mathcal{K}}'(S_0)$ consisting of $M = (M,\{M^{\emptyset}_{A}\}_{A\in \mathcal{A}})$ such that
\begin{itemize}
\item $M$ is a graded $(S_0,R)$-bimodule which is finitely generated torsion-free as a left $S_0$-module.
\item $M^{\emptyset}_{A}$ is an $(S_0^{\emptyset},R)$-bimodule such that $mf = f_Am$ for any $m\in M^{\emptyset}_{A}$ and $f\in R$.
\item $M^\emptyset = \bigoplus_{A\in \mathcal{A}}M^{\emptyset}_{A}$.
\end{itemize}
A morphism $\varphi\colon M\to N$ is an $(S_0,R)$-bimodule homomorphism of degree zero such that 
\[
\varphi(M_A^{\emptyset})\subset \bigoplus_{A'\ge A}N_{A'}^{\emptyset}
\]
for any $A\in \mathcal{A}$.
We put $\Hom^{\bullet}_{\widetilde{\mathcal{K}}'(S_0)}(M,N) = \bigoplus_i \Hom_{\widetilde{\mathcal{K}}'(S_0)}(M,N(i))$.
This is a graded $(S_{0},R)$-bimodule.
For $M\in \widetilde{\mathcal{K}}'(S_0)$, we put $\supp_{\mathcal{A}}(M) = \{A\in \mathcal{A}\mid M_A^{\emptyset}\ne 0\}$.

\begin{rem}\label{rem:decomposition is preserved automatically}
Let $\Omega\in \mathcal{A}/\Lambda_{\aff}$.
For any $m\in \bigoplus_{A\in\Omega}M_{A}^{\emptyset}$ and $f\in R$ we have $mf = f_{\Omega}m$.
The action of $W'_{\aff}$ on $\mathcal{A}/\Lambda_{\aff}$ factors through $W'_{\aff}\to W_{\mathrm{f}}$ and $W_{\mathrm{f}}$ acts on $\mathcal{A}/\Lambda_{\aff}$ simply transitively.
We have $M^{\emptyset} = \bigoplus_{w\in W_{\mathrm{f}}}(\bigoplus_{A\in w(\Omega)}M_{A}^{\emptyset})$ and for $m\in \bigoplus_{A\in w(\Omega)}M_{A}^{\emptyset}$, $mf = w(f_{\Omega})m$.
Therefore the decomposition of $M^{\emptyset}$ into $\bigoplus_{A\in w(\Omega)}M_{A}^{\emptyset}$ is determined by the $(S_0,R)$-bimodule structure.
Hence any $(S_0,R)$-bimodule homomorphism $M\to N$ sends $\bigoplus_{A\in\Omega}M_{A}^{\emptyset}$ to $\bigoplus_{A\in\Omega}N_{A}^{\emptyset}$.
We will often use this fact.
\end{rem}

\begin{rem}
Here we do not assume that a morphism $M\to N$ sends $M_A^{\emptyset}$ to $N^{\emptyset}_{A}$.
Therefore a submodule $M_A^{\emptyset}\subset M^{\emptyset}$ is not functorial
\end{rem}

For each closed subset $I\subset \mathcal{A}$, we define $M_{I} = M\cap \bigoplus_{A\in I}M_A^{\emptyset}$.
Set
\[
(M_{I})_{A}^{\emptyset}
=
\begin{cases}
M_{A}^{\emptyset} & (A\in I),\\
0 & (A\notin I).
\end{cases}
\]
By the following lemma, $M_{I}\in \widetilde{\mathcal{K}}'(S_0)$.
Hence $M\mapsto M_{I}$ is an endofunctor of $\widetilde{\mathcal{K}}'(S_0)$.

\begin{lem}\label{lem:generic decomposition of M_I}
The module $M_I$ is a submodule of $M$ and we have
\[
(M_{I})^\emptyset = \bigoplus_{A\in I}M_{A}^{\emptyset}.
\]
We also have $M_{I_1\cap I_2} = M_{I_1}\cap M_{I_2}$.
\end{lem}
\begin{proof}
The first part is obvious and for the second part, the left hand side is contained in  the right hand side.
Take $m$ from the right hand side and let $f\in S$ such that $fm\in M$.
Then we have $fm\in M_{I}$ and $m$ is in the left hand side.
The last assertion is obvious.
\end{proof}

For each $\alpha\in\Delta$, set $W_{\alpha,\aff}' = \{1,s_{\alpha}\}\ltimes \Z\alpha\subset W_{\aff}'$.
We also put $S^{\alpha} = S[(\beta^\vee)^{-1}\mid \beta\in \Delta\setminus\{\pm \alpha\}]$ and $M^\alpha = S^{\alpha}\otimes_S M$ for any left $S$-module.
Note that, from our assumption, $\bigcap_{\alpha\in\Delta^+}S^{\alpha} = S$~\cite[9.1 Lemma]{MR1272539}.
We say $M\in \widetilde{\mathcal{K}}(S_0)$ if $M\in \widetilde{\mathcal{K}}'(S_0)$ and satisfies the following two conditions which are taken from \cite{arXiv:1504.01699}.
These are important properties in our arguments.
\begin{itemize}
\item[\property{S}] $M_{I_1\cup I_2} = M_{I_1} + M_{I_2}$ for any closed subsets $I_1,I_2$.
\item[\property{LE}] For any $\alpha\in\Delta$, there exist $M_{\Omega}$ for all $\Omega\in W'_{\alpha,\aff}\backslash \mathcal{A}$ with an injective morphism $M_{\Omega}\hookrightarrow M^{\alpha}$ in $\widetilde{\mathcal{K}}'(S_0^{\alpha})$ such that $\supp_{\mathcal{A}}M_{\Omega}\subset \Omega$ and the induced morphism $\bigoplus_{\Omega\in W'_{\alpha,\aff}\backslash \mathcal{A}}M_{\Omega}\to M^{\alpha}$ is an isomorphism in $\widetilde{\mathcal{K}}'(S_0^{\alpha})$.
\end{itemize}

Assume that $M^{\alpha}\to M^{\emptyset}$ is injective.
If $M^{\alpha} = \bigoplus_{\Omega\in W_{\alpha,\aff}'\backslash \mathcal{A}}(\bigoplus_{A\in \Omega}M^{\emptyset}_{A}\cap M^{\alpha})$ for any $\alpha\in\Delta$, $M$ satisfies \property{LE}.
The converse is not true.
The correct statement is that $M$ satisfies \property{LE} if and only if for any $\alpha\in\Delta$ there exists $N\in \widetilde{\mathcal{K}}'(S_0^{\alpha})$ which is isomorphic to $M^{\alpha}$ and satisfies $N = \bigoplus_{\Omega\in W_{\alpha,\aff}'\backslash \mathcal{A}}(\bigoplus_{A\in\Omega}N^{\emptyset}_{A}\cap N)$.

\begin{lem}\label{lem:S holds for rank 1}
Let $M\in \widetilde{\mathcal{K}}'(S_0)$, $\alpha\in\Delta$ and $A\in \mathcal{A}$.
Assume that $\supp_{\mathcal{A}}(M)\subset W'_{\alpha,\aff}A$.
Then $M$ satisfies \property{S}.
In particular, if $M$ satisfies \property{LE}, then $M^{\alpha}$ satisfies \property{S}.
\end{lem}
\begin{proof}
Set $\Omega = W'_{\alpha,\aff}A$ and let $I_1,I_2\subset \mathcal{A}$ be closed subsets.
We have $\Omega = \{A,\alpha\uparrow A,\alpha\uparrow(\alpha\uparrow A),\dots\}\cup \{\alpha\downarrow A,\alpha\downarrow(\alpha\downarrow A),\dots\}$ and $\Omega$ is a totally ordered subset of $\mathcal{A}$.
Since $\Omega$ is totally ordered, $I_1\cap \Omega\subset I_2\cap \Omega$ or $I_2\cap \Omega\subset I_1\cap\Omega$.
We may assume $I_1\cap \Omega\subset I_2\cap \Omega$.
We can take closed subsets $I'_1$ and $I'_2$ such that $I'_1\subset I'_2$, $I'_1\cap \Omega = I_1\cap \Omega$ and $I'_2\cap \Omega = I_2\cap \Omega$.
Then we have $M_{I'_1} = M_{I_1}$, $M_{I'_2} = M_{I_2}$ and $M_{I'_1\cup I'_2} = M_{I_1\cup I_2}$.
Hence we may assume $I_1 = I'_1$ and $I_2 = I'_2$.
In this case \property{S} obviously holds.
\end{proof}

Let $K\subset \mathcal{A}$ be a locally closed subset, namely $K$ is the intersection of a closed subset $I$ with an open subset $J$.
It is easy to see that $M_{I}/M_{I\setminus J}\simeq M_{I'}/M_{I'\setminus J'}$ naturally for a closed subset $I$ and an open subset $J$ such that $K = I\cap J$.
We define $M_{K} = M_{I}/M_{I\setminus J}$ for $M\in \widetilde{\mathcal{K}}(S_{0})$.
By Lemma~\ref{lem:generic decomposition of M_I}, we have
\[
\bigoplus_{A\in K}M_A^{\emptyset}\xrightarrow{\sim}M^{\emptyset}_K.
\]
By putting $(M_K)^{\emptyset}_A$ as the image of $M_A^{\emptyset}$ by this isomorphism, we have an object $M_K$ of $\widetilde{\mathcal{K}}'(S_0)$.
The following lemma is obvious.
\begin{lem}
We have $\supp_{\mathcal{A}}(M_K) = \supp_{\mathcal{A}}(M)\cap K$ for any locally closed subset $K\subset \mathcal{A}$.
\end{lem}

\begin{lem}\label{lem:composition of M_K}
Let $K_1,K_2\subset \mathcal{A}$ be locally closed subsets.
If $M\in \widetilde{\mathcal{K}}(S_{0})$, then $(M_{K_1})_{K_2} \simeq M_{K_1\cap K_2}$
\end{lem}
\begin{proof}
The proof is divided into several steps.

\noindent (1)
Assume that both $K_1,K_2$ are closed.
Then the lemma follows from the definitions.

\noindent (2)
Assume that $K_1$ is open and $K_2$ is closed.
Set $I_1 = \mathcal{A}\setminus K_1$.
Then we have
\[
(M_{K_1})_{K_2} = M/M_{I_1}\cap \bigoplus_{A\in K_2}(M/M_{I_1})_{A}^{\emptyset}.
\]
Note that $M_{K_2}/(M_{K_2}\cap M_{I_1}) = M_{K_2}/M_{K_2\cap I_1} = M_{K_1\cap K_2}$.
There is a canonical embedding from $M_{K_2}/(M_{K_2}\cap M_{I_1})$ to $(M_{K_1})_{K_2}$.
Let $m\in M$ such that $m + M_{I_1}\in \bigoplus_{A\in K_2}(M/M_{I_1})_{A}^{\emptyset}$.
Then $M_A^{\emptyset}$-component $m_{A}$ of $m$ is $0$ for $A\notin I_1\cup K_2$.
Hence $m\in M_{I_1\cup K_2} = M_{I_1} + M_{K_2}$.
Therefore the canonical embedding is surjective.
We get the lemma.

\noindent (3)
Assume that $K_2$ is closed.
Take a closed subset $I_1$ and an open subset $J_1$ such that $K_1 = I_1\cap J_1$.
Then by (2), $(M_{J_1})_{I_1} \simeq M_{K_1}$.
Hence $(M_{K_1})_{K_2} \simeq ((M_{J_1})_{I_1})_{K_2} = (M_{J_1})_{I_1\cap K_2}$ by (1).
This is isomorphic to $M_{J_1\cap I_1\cap K_2} = M_{K_1\cap K_2}$ by (2).

\noindent (4)
Now we prove the lemma in general.
Let $I_i$ be a closed subset and $J_i$ be an open subset such that $K_i = I_i \cap J_i$ and put $J_i^c = \mathcal{A}\setminus J_i$ for $i = 1,2$.
Then 
\[
(M_{K_1})_{K_2} = (M_{K_1})_{I_2}/(M_{K_1})_{I_2\cap J_2^c} \simeq M_{K_1\cap I_2}/M_{K_1\cap I_2\cap J_2^c}
\]
by (3).
We have $M_{K_1\cap I_2} = M_{I_1\cap I_2}/M_{I_1\cap I_2\cap J_1^c}$ and $M_{K_1\cap I_2\cap J_2^c} = M_{I_1\cap I_2\cap J_2^c}/M_{I_1\cap I_2\cap J_2^c\cap J_1^c}$.
Hence
\[
(M_{K_1})_{K_2} \simeq M_{I_1\cap I_2}/(M_{I_1\cap I_2\cap J_1^c} + M_{I_1\cap I_2\cap J_2^c}).
\]
Since $M_{I_1\cap I_2\cap J_1^c} + M_{I_1\cap I_2\cap J_2^c} = M_{(I_1\cap I_2\cap J_1^c)\cup (I_2\cap I_2\cap J_2^c)} = M_{(I_1\cap I_2)\setminus (J_1\cap J_2)}$, we get the lemma.
\end{proof}
\begin{lem}
If $M\in\widetilde{\mathcal{K}}(S_0)$, then $M_K\in\widetilde{\mathcal{K}}(S_0)$.
\end{lem}
\begin{proof}
Take a closed subset $I$ and an open subset $J$ such that $K = I \cap J$.

We prove $M_{K}$ satisfies \property{S}.
Let $I_1,I_2$ be closed subsets.
Since $(M_{K})_{I_i} = M_{K\cap I_i}$ is a quotient of $M_{I\cap I_i}$, it is sufficient to prove that $M_{I\cap I_1}\oplus M_{I\cap I_2}\to (M_{K})_{I_1\cup I_2}$ is surjective.
The module $(M_{K})_{I_1\cup I_2} = M_{K\cap (I_1\cup I_2)}$ is a quotient of $M_{I\cap (I_1\cup I_2)}$ and since $M_{I\cap (I_1\cup I_2)} = M_{I\cap I_1} + M_{I\cap I_2}$, the map is surjective.

We prove $M_{K}$ satisfies \property{LE}.
We may assume $M = \bigoplus_{\Omega\in W'_{\alpha,\aff}\backslash \mathcal{A}}(\bigoplus_{A\in\Omega}M_{A}^{\emptyset}\cap M^{\alpha})$.
Let $m \in M_{I}^{\alpha}$.
Then for each $\Omega\in W'_{\alpha,\aff}\backslash \mathcal{A}$, we have $m_{\Omega}\in M^{\alpha}\cap \bigoplus_{A\in \Omega}M_{A}^{\emptyset}$ such that $m = \sum m_{\Omega}$.
Then for each $A\in \mathcal{A}$, we have $m_{A} = (m_{\Omega})_{A}$ where $\Omega$ is the unique $W'_{\alpha,\aff}$-orbit containing $A$.
Therefore, since $m\in M_{I}^{\alpha}$, we have $m_{\Omega}\in M_{I}^{\alpha}$.
Hence $m_{\Omega}\in M^{\alpha}_I\cap \bigoplus_{A\in \Omega}(M_I)_{A}^{\emptyset}$.
Namely $M_{I}$ satisfies \property{LE}.
Since $M_{K}$ is a quotient of $M_{I}$, it also satisfies \property{LE}.
\end{proof}

\subsection{Standard filtration}
Note that $\{A\} = \{A'\in \mathcal{A}\mid A'\ge A\}\cap \{A'\in \mathcal{A}\mid A'\le A\}$ is locally closed.
We say that an object $M$ of $\widetilde{\mathcal{K}}(S_0)$ admits a standard filtration if $M_{\{A\}}$ is a graded free $S_0$-module for any $A\in \mathcal{A}$.
Let $\widetilde{\mathcal{K}}_\Delta(S_0)$ be a full subcategory of $\widetilde{\mathcal{K}}(S_0)$ consisting of an object $M$ which admits a standard filtration and $\supp_{\mathcal{A}}(M)$ is finite.
By Lemma~\ref{lem:composition of M_K}, if $M\in \widetilde{\mathcal{K}}_{\Delta}(S_0)$ then $M_{K}\in \widetilde{\mathcal{K}}_{\Delta}(S_0)$ for any locally closed subset $K\subset \mathcal{A}$.

\begin{lem}\label{lem:filtration for FD module}
Let $M_1,\dots,M_l\in \widetilde{\mathcal{K}}(S_0)$ and assume that $\supp_{\mathcal{A}}(M_1),\dots,\supp_{\mathcal{A}}(M_l)$ are all finite.
Let $I\subset \mathcal{A}$ be a closed subset and $A\in I$ such that $I\setminus\{A\}$ is closed.
Then there exist closed subsets $I_0\subset I_1\subset \dotsm \subset I_r$ and $k \in \{1,\dots,r\}$ such that $\#(I_j\setminus I_{j - 1}) = 1$ for any $j = 1,\dots,r$, $I_k\cap (\bigcup_i\supp_{\mathcal{A}}(M_i)) = I\cap (\bigcup_i\supp_{\mathcal{A}}(M_i))$, $I_{k - 1} = I_k\setminus\{A\}$, $(M_i)_{I_0} = 0$ and $(M_i)_{I_r} = M$ for any $i =1,\dots,l$.
In particular, we have $(M_i)_I\simeq (M_i)_{I_k}$ for all $i = 1,\dots,l$.
\end{lem}

\begin{proof}
There exist $A_0^-,A_0^+$ such that $\supp_{\mathcal{A}}(M_i)\subset [A_0^-,A_0^+]$ for any $i = 1,\dots,l$ by \cite[Proposition~3.7]{MR591724}.
Put $I_0 = \{A'\in \mathcal{A}\mid A'\nless A_0^+\}\cap I$.
We enumerate the elements in $(I\setminus \{A\})\cap [A_0^-,A_0^+]$ (resp.\ $[A_0^-,A_0^+]\setminus I$) as $\{A_1,\dots,A_{k - 1}\}$ (resp.\ $\{A_{k + 1},\dots,A_r\}$) such that $A_i\ge A_j$ implies $i\le j$.
Put $A_k = A$.
Then it is easy to see that $I_i = I_0\cup\{A_1,\dots,A_i\}$ is closed and satisfies the conditions of the lemma.
\end{proof}

\begin{lem}
Let $M\in \widetilde{\mathcal{K}}_\Delta(S_0)$ and $K$ a locally closed subset.
Then $M_K$ is graded free as a left $S_0$-module.
\end{lem}
\begin{proof}
Since $M_K\in \widetilde{\mathcal{K}}_\Delta(S_0)$, we may assume $K = \mathcal{A}$.
Take closed subsets $I_0\subset I_1\subset\dotsm \subset I_r$ such that $I_{i + 1}\setminus I_i = \{A_i\}$, $M_{I_0} = 0$ and $M_{I_r} = M$.
Then $M_{I_{i + 1}}/M_{I_{i}} = M_{\{A_i\}}$ is a graded free $S_0$-module.
Hence $M_{I_r}/M_{I_0} = M$ is also graded free.
\end{proof}

Finally we define the category $\widetilde{\mathcal{K}}_P(S_0)$ which plays an important role later.
The definitions are taken from \cite{arXiv:1504.01699}.
\begin{defn}
We say a sequence $M_1\to M_2\to M_3$ in $\widetilde{\mathcal{K}}_{\Delta}(S_0)$ satisfies \property{ES} if the composition $M_1\to M_2\to M_3$ is zero and
\[
0\to (M_1)_{\{A\}}\to (M_2)_{\{A\}}\to (M_3)_{\{A\}}\to 0
\]
is exact for any $A\in \mathcal{A}$.

We define the category $\widetilde{\mathcal{K}}_P(S_0)\subset \widetilde{\mathcal{K}}_{\Delta}(S_0)$ as follows: $M\in \widetilde{\mathcal{K}}_P(S_0)$ if and only if for any sequence $M_1\to M_2\to M_3$ in $\widetilde{\mathcal{K}}_{\Delta}(S_0)$ with \property{ES}, the induced sequence
\[
0\to \Hom^{\bullet}_{\widetilde{\mathcal{K}}_{\Delta}(S_0)}(M,M_1)\to \Hom^{\bullet}_{\widetilde{\mathcal{K}}_{\Delta}(S_0)}(M,M_2)\to \Hom^{\bullet}_{\widetilde{\mathcal{K}}_{\Delta}(S_0)}(M,M_3)\to 0
\]
is exact.
\end{defn}

\begin{lem}\label{lem:exact sequence for any locally closed subset}
Assume that $M_1,M_2,M_3\in \widetilde{\mathcal{K}}(S_0)$ satisfy $\#\supp_{\mathcal{A}}(M_i) < \infty$ ($i = 1,2,3$) and a sequence $M_1\to M_2\to M_3$ satisfies \property{ES}.
Then $0\to (M_1)_{K}\to (M_2)_{K}\to (M_3)_{K}\to 0$ is exact for any locally closed subset $K$.
\end{lem}
\begin{proof}
Replacing $M_i$ with $(M_i)_K$ for $i = 1,2,3$, we may assume $K = \mathcal{A}$.
We can take closed subsets $I_0\subset I_1\subset\dotsb\subset I_r$ such that $(M_i)_{I_0} = 0$, $(M_i)_{I_r} = M_i$ and $\#(I_{j + 1}\setminus I_j) = 1$ for $i = 1,2,3$ and $j = 0,\dots,r$, as in Lemma~\ref{lem:filtration for FD module}.
Then the exactness of $0\to (M_1)_{I_j}\to (M_2)_{I_j}\to (M_3)_{I_j}\to 0$ follows from the induction on $j$ and a standard diagram argument.
\end{proof}

\begin{lem}
Let $M\in \widetilde{\mathcal{K}}_{\Delta}(S_0)$, $I_1\supset I_2$ closed subsets.
Then $M_{I_2}\to M_{I_1}\to M_{I_1}/M_{I_2}$ satisfies \property{ES}.
\end{lem}
\begin{proof}
Note that $M_{I_1}/M_{I_2} = M_{I_1\setminus I_2}$.
The lemma follows from Lemma~\ref{lem:composition of M_K}.
\end{proof}

\subsection{Base change}
Let $S_1$ be a flat commutative graded $S_0$-algebra.
For $M\in \widetilde{\mathcal{K}}'(S_0)$, $S_1\otimes_{S_0}M$ is an $(S_1,R)$-module.
Setting $(S_1\otimes_{S_0}M)_A^{\emptyset} = S_1\otimes_{S_0}M_{A}^{\emptyset}$, we get an object $S_1\otimes_{S_0}M\in \widetilde{\mathcal{K}}'(S_1)$.
Obviously we have $(S_1\otimes_{S_0}M)_{K} \simeq S_1\otimes_{S_0}M_{K}$ for any locally closed subset $K\subset \mathcal{A}$.
Using this, we have $S_1\otimes_{S_0}\widetilde{\mathcal{K}}(S_0)\subset \widetilde{\mathcal{K}}(S_1)$, $S_1\otimes_{S_0}\widetilde{\mathcal{K}}_{\Delta}(S_0)\subset \widetilde{\mathcal{K}}_{\Delta}(S_1)$.

We put $\widetilde{\mathcal{K}}' = \widetilde{\mathcal{K}}'(S)$, $\widetilde{\mathcal{K}} = \widetilde{\mathcal{K}}(S)$, $\widetilde{\mathcal{K}}_{\Delta} = \widetilde{\mathcal{K}}_{\Delta}(S)$ and $\widetilde{\mathcal{K}}_P = \widetilde{\mathcal{K}}_P(S)$.
We also put $(\widetilde{\mathcal{K}}')^* = \widetilde{\mathcal{K}}'(S^*)$, $\widetilde{\mathcal{K}}^* = \widetilde{\mathcal{K}}(S^*)$, $\widetilde{\mathcal{K}}^*_{\Delta} = \widetilde{\mathcal{K}}_{\Delta}(S^*)$ and $\widetilde{\mathcal{K}}^*_P = \widetilde{\mathcal{K}}_P(S^*)$ for $*\in \Delta\cup \{\emptyset\}$.

\subsection{Hecke action}
Let $s\in S_{\aff}$ and we define $\alpha_s\in \Lambda_{\Coeff}$ and $\alpha_s^\vee\in \mathcal{L}_{\Coeff}$ as follows: let $A\in \mathcal{A}$ and $\alpha\in\Delta^+$ such that $s_{\alpha,n} = As$ for some $n\in\Z$.
Then we put $\alpha_s = \alpha^A$ and $\alpha_s^\vee = (\alpha^\vee)^A$.
These depend on a choice of $A$ and $\alpha$.
For each $s\in S_{\aff}$ we fix such $A$ and $\alpha$ and define $\alpha_s,\alpha_s^\vee$.
\begin{lem}
The pair $(\alpha_s,\alpha_s^\vee)$ does not depend on $A,\alpha$ up to sign.
\end{lem}
\begin{proof}
Let $A'\in \mathcal{A}$ and take $\beta\in\Delta^+$ and $m\in\Z$ such that $A's = s_{\beta,m}A'$.
Take $x\in W'_\aff$ such that $A' = xA$.
Then $A's = xAs = xs_{\alpha,n}A = s_{x(\alpha,n)}xA = s_{x(\alpha,n)}A'$.
Hence $\beta = \varepsilon x(\alpha)$ for $\varepsilon = 1$ or $\varepsilon = -1$.
We have $\beta^{A'} = \varepsilon x(\alpha)^{xA} = \varepsilon \alpha^A$ and $(\beta^\vee)^{A'} = \varepsilon (\alpha^\vee)^{A}$.
\end{proof}

We have that $(\Lambda_{\Coeff},\{\alpha_s\}_{s\in S_{\aff}},\{\alpha_s^\vee\}_{s\in S_{\aff}})$ is a realization with Demazure surjectivity~\cite[Definition~3.1]{MR3555156}.
Let $\Sbimod$ be the category introduced in \cite{arXiv:1901.02336_accepted}.
Set $R^{\emptyset} = R[((\alpha^\vee)^A)^{-1}\mid \alpha\in \Delta]$ for $A\in \mathcal{A}$.
It is easy to see that this does not depend on $A$.
We put $B^{\emptyset} = R^{\emptyset}\otimes_{R}B$ for $B\in \Sbimod$.

Recall that we have an object $B_s\in \Sbimod$.
Set $R^s = \{f\in R\mid s(f) = f\}$.
As an $R$-bimodule, $B_s = R\otimes_{R^s}R(1)\simeq \{(f,g)\in R^2\mid f\equiv g\pmod{\alpha_s}\}$ and we have the decomposition of $B_s^{\emptyset} = \bigoplus_{w\in W}(B_s)^{\emptyset}_{w}$ where
\begin{align*}
(B_s)^{\emptyset}_e & = R^{\emptyset}(\delta_s\otimes 1 - 1\otimes s(\delta_s)),\\
(B_s)^{\emptyset}_s & = R^{\emptyset}(\delta_s\otimes 1 - 1\otimes \delta_s),\\
(B_s)^{\emptyset}_w & = 0\quad (w\ne e,s).
\end{align*}
Here $\delta_s\in \Lambda^\vee_{\Coeff}$ satisfies $\langle\alpha_s,\delta_s\rangle = 1$.
The decomposition does not depend on a choice of $\delta_s$.

Let $\Sbimod$ be the category defined in \cite{arXiv:1901.02336_accepted} for $(W_\aff,S_\aff)$ and the representation $\Lambda_{\Coeff}$ of $W_\aff$ with $\{(\alpha_s,\alpha^\vee_s)\mid s\in S_\aff\}$.
We remark that \cite[Assumption~3.2]{arXiv:1901.02336_accepted} is satisfied in this case by \cite[Theorem~1.2, Proposition~3.7]{arXiv:2012.09414}.

\begin{lem}
Let $B\in \Sbimod$.
Then there exists a decomposition $B^{\emptyset} = \bigoplus_{x\in W_\aff}B_x^{\emptyset}$ such that $\Frac(R)\otimes_{R^\emptyset}B_x^{\emptyset}\simeq B_x^{\Frac(R)}$.
Here $B_x^{\Frac(R)}$ is the $\Frac(R)$-bimodule as in the definition of $\Sbimod$.
\end{lem}
\begin{proof}
Assume that $B_1\in \Sbimod$ is a direct summand of $B\in \Sbimod$ and let $e\in \End_{\Sbimod}(B)$ be the idempotent such that $B_1 = e(B)$.
If $B$ satisfies the lemma, then by putting $(B_1)_x^{\emptyset} = e(B_x^{\emptyset})$, we see that $B_1$ also satisfies the lemma.
Therefore we may assume $B = B_{s_1}\otimes\cdots\otimes B_{s_l}$ for $s_i\in S_{\aff}$.
Note that for $B = B_s$, the lemma holds as we have seen in the above.
Hence it is sufficient to prove that if $B_1,B_2$ satisfies the lemma then $B_1\otimes B_2$ also satisfies the lemma.

For $x\in W_{\aff}$ and $b\in (B_1)_x^{\emptyset}$, we have $bf = x(f)b$ for $f\in R$.
Since $\{(\alpha^\vee)^{A}\mid \alpha\in\Delta\}$ is stable under the action of $x$, the formula says that $(B_1)_x^{\emptyset}$ is also a right $R^{\emptyset}$-module.
Therefore $B_1^{\emptyset}$ is also a right $R^{\emptyset}$-module.
Hence $R^{\emptyset}\otimes_{R}B_1\otimes_{R}B_2\simeq B_1^{\emptyset}\otimes_{R}B_2\simeq B_1^{\emptyset}\otimes_{R^{\emptyset}}R^{\emptyset}\otimes_{R}B_2 \simeq B_1^{\emptyset}\otimes_{R^{\emptyset}}B_2^{\emptyset}$.
We put $B_{x}^{\emptyset} = \bigoplus_{yz = x}(B_1)_y^{\emptyset}\otimes_{R^{\emptyset}}(B_2)_z^{\emptyset}$.
Then we get $B^{\emptyset} = \bigoplus_{x\in W_{\aff}}B_x^{\emptyset}$ and we have $\Frac(R)\otimes_{R^{\emptyset}}B_x^{\emptyset}\simeq B_x^{\Frac(R)}$.
\end{proof}

For $M\in \widetilde{\mathcal{K}}'(S_0)$ and $B\in \Sbimod$, we define $M*B\in \widetilde{\mathcal{K}}'(S_0)$ by
\begin{itemize}
\item As an $(S_0,R)$-bimodule, $M*B = M\otimes_R B$.
\item We put $(M*B)_A^\emptyset = \bigoplus_{x\in W_\aff}M_{Ax^{-1}}^{\emptyset}\otimes_{R^\emptyset} B_x^{\emptyset}$.
\end{itemize}
Let $f\colon M\to N$ be a morphism in $\widetilde{\mathcal{K}}'(S_0)$.
We have $f(M_{Ax^{-1}}^{\emptyset})\subset \bigoplus_{A'\in Ax^{-1} + \Z\Delta,A'\ge Ax^{-1}}N_{A'}^{\emptyset}$.
By Lemma~\ref{lem:homeo on each orbit}, for $A'\in Ax^{-1} + \Z\Delta$, $A'\ge Ax^{-1}$ if and only if $A'x\ge A$.
Therefore $\bigoplus_{A'\in Ax^{-1} + \Z\Delta,A'\ge Ax^{-1}}N_{A'}^{\emptyset} = \bigoplus_{A'\in A + \Z\Delta,A'\ge A}N_{A'x^{-1}}^{\emptyset}$ by replacing $A'x$ with $A'$.
Hence 
\[
(f\otimes \id)(M_{Ax^{-1}}^{\emptyset}\otimes B_x^{\emptyset})\subset\bigoplus_{A'\in A + \Z\Delta,A'\ge A}N_{A'x^{-1}}^{\emptyset}\otimes B_x^{\emptyset} \subset \bigoplus_{A'\ge A}(N*B)_{A'}^{\emptyset}.
\]
Therefore $(f\otimes \id)$ gives a morphism in $\widetilde{\mathcal{K}}'(S_0)$.
Similarly, if $f\colon B_1\to B_2$ is a morphism in $\Sbimod$, then $\id\otimes f\colon M*B_1\to M*B_2$ is a morphism in $\mathcal{K}'(S_0)$.

For each $B\in \Sbimod$, $B_x^{\emptyset}$ is free as a left $R^{\emptyset}$-module.
The following lemma follows.
\begin{lem}
We have $\supp_{\mathcal{A}}(M*B) = \{Ax\mid A\in \supp_{\mathcal{A}}(M),x\in \supp_{W_\aff}(B)\}$.
\end{lem}

We regard $M\otimes_R B_s = M\otimes_{R^s}R(1) = M(1)\otimes 1\oplus M(1)\otimes \delta_s$.
Inside it, we have
\begin{equation}\label{eq:decomposition of M*B_s}
\begin{split}
(M*B_s)^{\emptyset}_A & = \{m\delta_s\otimes 1 - m\otimes s(\delta_s)\mid m\in M_{A}^{\emptyset}\}\oplus \{m\delta_s\otimes 1 - m\otimes \delta_s\mid m\in M_{As}^{\emptyset}\}\\
& \simeq M_{A}^{\emptyset}\oplus M_{As}^{\emptyset}.
\end{split}
\end{equation}
The isomorphism is given by $m\otimes f\mapsto (mf,ms(f))$.
Note that the last isomorphism is an isomorphism as left $S_0^{\emptyset}$-modules.
As right $R$-modules, if $m\in (M*B_s)^{\emptyset}_A$ corresponds to $(m_1,m_2)\in M_{A}^{\emptyset}\oplus M_{As}^{\emptyset}$, then $mf$ corresponds to $(m_1f,m_2s(f))$.

\begin{prop}\label{prop:adjointness}
Let $M,N\in \widetilde{\mathcal{K}}'(S_0)$.
We have $\Hom^{\bullet}_{\widetilde{\mathcal{K}}'(S_0)}(M,N*B_s)\simeq \Hom^{\bullet}_{\widetilde{\mathcal{K}}'(S_0)}(M*B_s,N)$.
\end{prop}
\begin{proof}
Take $\delta\in \Lambda^\vee_{\Coeff}$ such that $\langle \alpha_s,\delta\rangle = 1$.
As $(S_0,R)$-bimodules, we have $N*B_s = N\otimes_{R^s}R(1)$ and $M*B_s = M\otimes_{R^s}R(1)$.
For $\varphi\colon M\otimes_{R^s}R(1)\to N$, define $\psi\colon M\to N\otimes_{R^s}R(1)$ by $\psi(m) = \varphi(m\delta \otimes 1)\otimes 1 - \varphi(m\otimes 1)\otimes s(\delta)$.
We know that if $\varphi$ is an $(S_0,R)$-bimodule homomorphism, $\psi$ is also an $(S_0,R)$-bimodule homomorphism and it induces a bijection between the spaces of $(S_0,R)$-bimodule homomorphisms. (See, for example, \cite[Lemma~3.3]{MR2441994}.)
We prove that $\varphi$ is a morphism in $\widetilde{\mathcal{K}}'(S_0)$ if and only if $\psi$ is a morphism in $\widetilde{\mathcal{K}}'(S_0)$. 

Set $a(m) = m\delta \otimes 1 - m\otimes s(\delta)$ and $b(m) = ms(\delta) \otimes 1 - m\otimes s(\delta)$ for $m\in M^\emptyset$.
We also define $a'(n),b'(n)\in N^{\emptyset}\otimes_{R^s}R$ for $n\in N^\emptyset$ by the same way.
Then we have $(M*B_s)_{A}^{\emptyset} = a(M_{A}^{\emptyset}) + b(M_{As}^{\emptyset})$ and the same for $N$ by \eqref{eq:decomposition of M*B_s} for $A\in \mathcal{A}$.

Let $A\in \mathcal{A}$ and $m\in M_{A}^{\emptyset}$.
By the definition, $\psi(m) = \varphi(a(m))\otimes 1 + b'(\varphi(m\otimes 1))$.
Since $a(m)\in (M*B_s)^{\emptyset}_A$, $\varphi(a(m))\otimes 1 = (\alpha_s)_{A}^{-1}\varphi(a(m))\alpha_s\otimes 1 = (\alpha_s)_{A}^{-1}a'(\varphi(a(m))) - (\alpha_s)_{A}^{-1}b'(\varphi(a(m)))$.
On the other hand, we have $m\otimes 1 = (\alpha_s)_{A}^{-1}m\alpha_s\otimes 1 = (\alpha_s)_{A}^{-1}a(m) - (\alpha_s)_{A}^{-1}b(m)$.
Since $\varphi$ and $b'$ are left $S_0$-equivariant, we get $\psi(m) = (\alpha_s)_{A}^{-1}a'(\varphi(a(m))) - (\alpha_s)_{A}^{-1}b'(\varphi(b(m)))$.

Assume that $\varphi$ is a morphism in $\widetilde{\mathcal{K}}'(S_0)$.
Then for any $m\in M_{A}^{\emptyset}$, $\varphi(a(m))\in \bigoplus_{A'\ge A}N_{A'}^{\emptyset}$.
Hence $a'(\varphi(a(m)))\in\bigoplus_{A' \ge A}(N*B_s)_{A'}^{\emptyset}$.
Since $b(m)\in (M*B_s)_{As}^{\emptyset}$, we have $\varphi(b(m))\in \bigoplus_{A' \ge As, A'\in As + \Z\Delta}N_{A'}^{\emptyset}$.
Therefore $b'(\varphi(b(m)))\in \bigoplus_{A' \ge As,A' \in As + \Z\Delta}(N*B_s)_{A's}^{\emptyset}$.
If $A'\in As + \Z\Delta$ satisfies $A'\ge As$, since $s\colon As + \Z\Delta\to A + \Z\Delta$ preserves the order, we get $A's \ge A$.
Hence $b'(\varphi(b(m)))\in \bigoplus_{A' \ge A}(N*B_s)_{A'}^{\emptyset}$.
Therefore $\psi$ is a morphism in $\widetilde{\mathcal{K}}'(S_0)$.

On the other hand, assume that $\psi$ is a morphism in $\widetilde{\mathcal{K}}'(S_0)$.
Consider the map $\Phi\colon N\otimes_{R^s}R\to N$ defined by $n\otimes f\mapsto nf$.
Then $\Phi(a'(n)) = n\alpha_s$ and $\Phi(b'(n)) = 0$.
Therefore $\Phi((N*B_s)_{A}^{\emptyset}) = \Phi(a'(N_{A}^{\emptyset}) + b'(N_{As}^{\emptyset}))\subset N_{A}^{\emptyset}$.
Let $m\in M_{A}^{\emptyset}$.
Then applying $\Phi$ to $\psi(m) = (\alpha_s)_{A}^{-1}a'(\varphi(a(m))) - (\alpha_s)_{A}^{-1}b'(\varphi(b(m)))$, we get $(\alpha_s)_{A}^{-1}\varphi(a(m))\alpha_s\in \bigoplus_{A'\in A + \Z\Delta,A'\ge A}M_{A'}^{\emptyset}$.
Hence $\varphi(a(M_{A}^{\emptyset}))\subset \bigoplus_{A'\ge A}N_{A'}^{\emptyset}$.
Similarly, using $N\otimes_{R^s}R\to N$ defined by $n\otimes f\mapsto ns(f)$, we get $\varphi(b(M_{As}^{\emptyset}))\subset \bigoplus_{A'\ge A}N_{A'}^{\emptyset}$.
Since $(M*B_s)_{A}^{\emptyset} = a(M_{A}^{\emptyset}) + b(M_{As}^{\emptyset})$, $\varphi$ is a morphism in $\widetilde{\mathcal{K}}'(S_0)$.
\end{proof}

\begin{lem}\label{lem:rank 1 decomposition of M*B_s}
Let $M\in \widetilde{\mathcal{K}}'(S_0)$.
\begin{enumerate}
\item For $\alpha\in\Delta$, $s\in S_{\aff}$ and $\Omega\in W'_{\alpha,\aff}\backslash \mathcal{A}$, set $M^{(\Omega)} = M^{\alpha}\cap \bigoplus_{A\in \Omega}M_A^{\emptyset}$.
Then we have the following.
\begin{enumerate}
\item If $\Omega s = \Omega$, then $(M*B_s)^{(\Omega)} \simeq M^{(\Omega)}*B_s$.
\item If $\Omega s \ne \Omega$, then the right action of $\alpha_s$ on $M^{(\Omega)}$ is invertible and we have
\[
(M*B_s)^{(\Omega)}\simeq M^{(\Omega)}\otimes(\delta_s\otimes 1 - 1\otimes s(\delta_s))\oplus M^{(\Omega s)}\otimes(\delta_s\otimes 1 - 1\otimes \delta_s)
\]
where $\langle\alpha_s,\delta_s\rangle = 1$.
\end{enumerate}
\item If $M\in \widetilde{\mathcal{K}}'(S_0)$ satisfies \property{LE}, then $M*B$ also satisfies \property{LE} for any $B\in \Sbimod$.
\end{enumerate}
\end{lem}
\begin{proof}
We have
\begin{align*}
(M*B_s)^{(\Omega)} & = M^{\alpha}*B_s\cap \bigoplus_{A\in \Omega}(M*B_s)_{A}^{\emptyset}\\
& = M^{\alpha}*B_s\cap \left(\bigoplus_{A\in \Omega}M_A^{\emptyset}\otimes (B_s)^{\emptyset}_e\oplus \bigoplus_{A\in \Omega}M_{As}^{\emptyset}\otimes (B_s)^{\emptyset}_s\right).
\end{align*}
If $\Omega s = \Omega$, then in the second direct sum, we can replace $As$ with $A$.
Therefore
\begin{align*}
(M*B_s)^{(\Omega)} & = M^{\alpha}*B_s\cap \left(\bigoplus_{A\in \Omega}M_A^{\emptyset}\otimes (B_s)^{\emptyset}_e\oplus \bigoplus_{A\in \Omega}M_{A}^{\emptyset}\otimes (B_s)^{\emptyset}_s\right)\\
& = M^{\alpha}*B_s\cap \bigoplus_{A\in \Omega}M_A^{\emptyset}\otimes (B_s)^{\emptyset}\\
& = (M^{\alpha}\cap \bigoplus_{A\in \Omega}M_A^{\emptyset})\otimes B_s\\
& = M^{(\Omega)}* B_s.
\end{align*}

Assume that $\Omega s\ne \Omega$ and take $A\in\Omega$.
Set $\beta^\vee = (\alpha_s^\vee)_A$.
Then the assumption $\Omega s\ne \Omega$ tells us that $\beta^\vee\ne \pm \alpha^\vee$.
Hence $\beta^\vee$ is invertible in $S^\alpha$.
The element $s_\alpha(\beta^\vee)$ is also invertible.

Let $\delta\in X^\vee_{\Coeff}$ such that $\langle\alpha,\delta\rangle = 1$.
For $m \in M^{(\Omega)}$, there exists $m_1\in \bigoplus_{A'\in A + \Z \alpha}M_{A'}^{\emptyset}$ and $m_2\in \bigoplus_{A'\in s_{(\alpha,0)}A + \Z \alpha}M_{A'}^{\emptyset}$ such that $m = m_1 + m_2$.
For each $f\in R$, $m_1f = f_Am_1$ and $m_2f = s_\alpha(f_A)m_2$.
By calculations using this, we have
\[
\left(\frac{1}{\beta^\vee}m + \frac{\langle\alpha,\beta^\vee\rangle}{\beta s_{\alpha}(\beta^\vee)}(\delta m - m\delta^A)\right)\alpha_s^\vee = m.
\]
Hence the right action of $\alpha_s^\vee$ is invertible.

Therefore, we have $(M*B_s)^{(\Omega)} = (M*B_s[\alpha_s^{-1}])^{(\Omega)}$ here $B_s[(\alpha_s^\vee)^{-1}] = B_s\otimes_{R}R[(\alpha_s^\vee)^{-1}]$.
Since $B_s[(\alpha_s^\vee)^{-1}] = R[(\alpha_s^\vee)^{-1}](\delta_s\otimes 1 - 1\otimes s(\delta_s))\oplus R[(\alpha_s^\vee)^{-1}](\delta_s\otimes 1 - 1\otimes \delta_s)$ with $R[(\alpha_s^\vee)^{-1}](\delta_s\otimes 1 - 1\otimes s(\delta_s))\subset (B_s)_e^{\emptyset}$ and $R[(\alpha_s^\vee)^{-1}](\delta_s\otimes 1 - 1\otimes \delta_s)\subset (B_s)_s^{\emptyset}$, the definition of $(M*B_s)^{(\Omega)}$ implies (b).

\noindent
(2)
Fix $\alpha\in\Delta$.
By replacing $M^{\alpha}$ with an object which is isomorphic to $M^{\alpha}$, we may assume $M^{\alpha} = \bigoplus_{\Omega\in W_{\alpha,\aff}'\backslash \mathcal{A}}(\bigoplus_{A\in\Omega}M_A^{\emptyset}\cap M^{\alpha})$.
Let $\{\Omega_i\}$ be a complete representatives of $\{\Omega\in W'_{\alpha,\aff}\backslash \mathcal{A}\mid \Omega s\ne \Omega\}/\{e,s\}$.
Then we have 
\begin{align*}
\bigoplus_{\Omega\in W'_{\alpha,\aff}\backslash \mathcal{A}}(M^{\alpha}*B_s)^{(\Omega)}
& = 
\bigoplus_{\Omega s = \Omega}(M*B_s)^{(\Omega)} \oplus
\bigoplus_{i}((M*B_s)^{(\Omega_i)}\oplus (M*B_s)^{(\Omega_i s)})\\
& = 
\bigoplus_{\Omega s = \Omega}M^{(\Omega)}*B_s \oplus
\bigoplus_{i}((M*B_s)^{(\Omega_i)}\oplus (M*B_s)^{(\Omega_i s)}).
\end{align*}
The argument of the proof of (1)(b), we have $M^{(\Omega_i)}\otimes (\delta_s \otimes 1 - 1\otimes s(\delta_s))\oplus M^{(\Omega_i)}\otimes(\delta_s\otimes 1 - 1\otimes \delta_s) = M^{(\Omega_i)}\otimes B_s[\alpha_s^{-1}] = M^{(\Omega_i)}\otimes B_s$.
Therefore, by (1)(b), $((M*B_s)^{(\Omega_i)}\oplus (M*B_s)^{(\Omega_i s)}) = M^{(\Omega_i)}\otimes B_s\oplus M^{(\Omega_i s)}\otimes B_s$.
Hence
\begin{align*}
\bigoplus_{\Omega\in W'_{\alpha,\aff}\backslash \mathcal{A}}(M*B_s)^{(\Omega)}
& =
\bigoplus_{\Omega s = \Omega}M^{(\Omega)}*B_s \oplus
\bigoplus_{i}(M^{(\Omega_i )}*B_s\oplus M^{(\Omega_i s)}*B_s)\\
& =
\bigoplus_{\Omega\in W'_{\alpha,\aff}\backslash \mathcal{A}}M^{(\Omega)}*B_s\\
& = M^{\alpha}*B_s.
\end{align*}
Hence $M*B_s$ satisfies \property{LE}.
\end{proof}

\subsection{An example}
We give an example of our category.
Let $(X = \Z,\Delta = \{\alpha = 2\},X^\vee = \Z,\Delta^\vee = \{\alpha^\vee = 1\})$ be the root system of type $A_1$.
The Weyl group $W_{\mathrm{f}}$ is $\{e,s_{\alpha}\}$.
Let $s_1\in S_{\aff}$ (resp.\ $s_0\in S_{\aff}$) be the element corresponding to $W'_{\aff}\{0\}$ (resp.\ $W'_{\aff}\{1\}$).
Then $S_{\aff} = \{s_0,s_1\}$.
The set of alcoves is given by $\mathcal{A} = \{A_n = \{r\in \R = X\otimes_{\Z}\R\mid n < r < n + 1\}\mid n\in\Z\}$.
We have $A_ns_1 = A_{n - 1}$ if $n$ is even and $A_ns_1 = A_{n + 1}$ if $n$ is odd.
The algebra $S = S(X_{\Coeff}^\vee)$ is isomorphic to the polynomial ring $\Coeff[\alpha^\vee]$.

Define $Q_{A_n}\in \widetilde{\mathcal{K}}' = \widetilde{\mathcal{K}}'(S)$ as follows.
As an $(S,R)$-bimodule, we define $Q_{A_n} = \{(f,g)\in S^2\mid f\equiv g\pmod{\alpha^\vee}\}$, here $S$ acts naturally and $r\in R$ acts by $(f,g)r = (r_{A_n}f,r_{A_{n + 1}}g)$.
We put $(Q_{A_n}^{\emptyset})_{A_n} = S^{\emptyset}\oplus 0$, $(Q_{A_n}^{\emptyset})_{A_{n + 1}} = 0\oplus S^{\emptyset}$ and $(Q_{A_n}^{\emptyset})_{A_m} = 0$ for $m\ne n,n + 1$.
(This object will be denoted by $Q_{A_n,\alpha}$ later.)

We have $\supp_{\mathcal{A}}(Q_{A_n}) = \{A_n,A_{n + 1}\}$.
We prove $Q_{A_0}*B_{s_1} \simeq Q_{A_{-1}}\oplus Q_{A_{1}}$.
We have $\supp_{\mathcal{A}}(Q_{A_0}*B_{s_1}) = \{A_0,A_{1},A_0s_1,A_{1}s_1\} = \{A_{- 1},A_{0},A_{1},A_{2}\}$.

In the below, by an isoimorphism $f\mapsto f_{A_0}$, we identify $R\simeq S = \Coeff[\alpha^\vee]$.
Hence $Q_{A_n} = \{(a,b)\in \Coeff[\alpha^\vee]^2\mid a\equiv b\pmod{\alpha^\vee}\}$.
Put $s = s_{\alpha}$ which acts on $\Coeff[\alpha^\vee]$.
The right action of $R\simeq \Coeff[\alpha^\vee]$ on $Q_{A_0},Q_{A_1},Q_{A_{-1}}$ are given as follows: for $(a,b)\in Q_{A_0}$, we have $(a,b)f = (af,bs(f))$ and for $(c,d)\in Q_{A_1},Q_{A_{-1}}$, we have $(c,d)f = (cs(f),df)$.

We have $B_{s_1}\simeq \{(f,g)\in \Coeff[\alpha^\vee]\mid f \equiv g\pmod{\alpha^\vee}\}$, $(B_{s_1}^{\emptyset})_{e} = \Coeff[\alpha^\vee]^{\emptyset} \oplus 0$ and $(B_{s_1}^{\emptyset})_{s_1} = 0\oplus \Coeff[\alpha^\vee]^{\emptyset}$ where $\Coeff[\alpha^\vee]^{\emptyset} = \Coeff[(\alpha^\vee)^{\pm 1}]$.
We have
\begin{align*}
& (Q_{A_0}*B_{s_1})^{\emptyset}_{A_{-1}} = (\Coeff[\alpha^\vee]^{\emptyset}\oplus 0)\otimes (0\oplus \Coeff[\alpha^\vee]^{\emptyset}), \\
& (Q_{A_0}*B_{s_1})^{\emptyset}_{A_{0}} = (\Coeff[\alpha^\vee]^{\emptyset}\oplus 0)\otimes (\Coeff[\alpha^\vee]^{\emptyset}\oplus 0),\\
& (Q_{A_0}*B_{s_1})^{\emptyset}_{A_{1}} = (0\oplus \Coeff[\alpha^\vee]^{\emptyset})\otimes (\Coeff[\alpha^\vee]^{\emptyset}\oplus 0), \\
& (Q_{A_0}*B_{s_1})^{\emptyset}_{A_{2}} = (0\oplus \Coeff[\alpha^\vee]^{\emptyset})\otimes (0\oplus \Coeff[\alpha^\vee]^{\emptyset}).
\end{align*}
(These correspond to $A_{-1} = A_0s_1$, $A_{1} = A_{1}e$, $A_{1} = A_{1}e$ and $A_{2} = A_{1}s_1$, respectively.)

We define $p_1\colon Q_{A_0}*B_s\to Q_{A_{- 1}}$ by $p_1((a,b)\otimes (f,g)) = (ag,af)$ and $p_2\colon Q_{A_0}*B_s\to Q_{A_{1}}$ by $p_2((a,b)\otimes (f,g)) = ((bs(f) - ag)/\alpha^\vee,(bs(g) - af)/\alpha^\vee)$.
In the definition of $p_2$, we note that $bs(f)\equiv ag,bs(g)\equiv af\pmod{\alpha^\vee}$ since $a\equiv b,s(f)\equiv f,s(g)\equiv g,f\equiv g\pmod{\alpha^\vee}$.
These are $\Coeff[\alpha^\vee]$-bimodule homomorphisms and from the above description, $p_1$ is a morphism in $\widetilde{\mathcal{K}}'$.
We have $p_2((1,0)\otimes (0,1)) = (-1/\alpha^\vee,0)$.
Hence $p_2((Q_{A_0}*B_{s_1})_{A_{-1}}^{\emptyset})\subset (Q_{A_{1}})_{A_1}^{\emptyset}$.
We also have $p_2((Q_{A_0}*B_{s_1})^{\emptyset}_{A_1})\subset (Q_{A_{1}})_{A_1}^{\emptyset}$, $p_2((Q_{A_0}*B_{s_1})^{\emptyset}_{A_0}),p_2((Q_{A_0}*B_{s_1})^{\emptyset}_{A_2})\subset (Q_{A_{1}})_{A_2}^{\emptyset}$.
Therefore $p_2$ is also a morphism in $\widetilde{\mathcal{K}}'$.

We define $i_1\colon Q_{A_{- 1}}\to Q_{A_0}*B_{s_1}$ by $i_1(a,b) = (b,a)\otimes (1,1) + ((a - b)/\alpha^\vee,(a - b)/\alpha^\vee)\otimes (0,\alpha^\vee)$.
In $(Q_{A_0}*B_{s_1})^{\emptyset}$, $i_1$ is given by $i_1(a,b) = (b,a)\otimes (1,0) + (a,b)\otimes (0,1)$.
It is easy to see that $i_1$ is a left $\Coeff[\alpha^\vee]$-module homomorphism.
For $f\in \Coeff[\alpha^\vee]$, we have $i_1(a,b)f = (b,a)\otimes (f,0) + (a,b)\otimes (0,s(f)) = (b,a)f\otimes (1,0) + (a,b)s(f)\otimes (0,1) = (bf,as(f))\otimes (1,0) + (as(f),bf)\otimes (0,1) = i_1(as(f),bf) = i_1((a,b)f)$.
Therefore $i_1$ is a $\Coeff[\alpha^\vee]$-bimodule homomorphism.
We can also check that $i_1$ is a morphism in $\widetilde{\mathcal{K}}'$.
We also define $i_2\colon Q_{A_1}\to Q_{A_0}*B_{s_1}$ by $i_2(a,b) = (0,\alpha^\vee)\otimes (s(a),s(b))$.
Then it is straightforward to check that $i_2$ is a morphism in $\widetilde{\mathcal{K}}'$.
Finally a straightforward calculations imply $p_1\circ i_1 = \id$, $p_2 \circ i_2 = \id$, $i_1\circ  p_1 + i_2\circ p_2 = \id$.
Hence $Q_{A_0}*B_{s_1}\simeq Q_{A_{-1}}\oplus Q_{A_1}$.

Note that the decomposition $Q_{A_0}*B_{s_1} = \Ima i_1\oplus \Ima i_2$ is not compatible with respect to the decomposition over $\Coeff[\alpha^\vee]^{\emptyset}$ since $i_1$ is not compatible with the decomposition.

\subsection{Hecke actions preserve $\widetilde{\mathcal{K}}_{\Delta}$}
We assume that $\Coeff$ is local.
Then since any direct summands of any graded free $S$-module is also graded free, a direct summand of an object in $\widetilde{\mathcal{K}}_{\Delta}$ is also in $\widetilde{\mathcal{K}}_{\Delta}$.
The aim of this subsection is to prove the following proposition.
\begin{prop}\label{prop:K_Delta is stable under the Hecke actions}
We have $\widetilde{\mathcal{K}}_{\Delta} * \Sbimod\subset \widetilde{\mathcal{K}}_{\Delta}$.
\end{prop}
We fix $M \in \widetilde{\mathcal{K}}_{\Delta}$ and $s\in S_{\aff}$ in this subsection and prove $M*B_s\in \widetilde{\mathcal{K}}_{\Delta}$.
The most difficult part is to prove that $M*B_s$ satisfies \property{S}.
First we remark that, since $M*B_s$ satisfies \property{LE} by Lemma~\ref{lem:rank 1 decomposition of M*B_s}, $(M*B_s)^{\alpha}$ satisfies \property{S} by Lemma~\ref{lem:S holds for rank 1}.

\begin{lem}\label{lem:(M*B_s)_I for s-invariant I}
If $I$ is a closed $s$-invariant subset of $\mathcal{A}$, then $(M*B_s)_I\simeq M_I*B_s$.
\end{lem}
\begin{proof}
We have $(M*B_s)_{I}^{\emptyset} = \bigoplus_{A\in I}M_{A}^{\emptyset}\otimes (B_s)_{e}^{\emptyset}\oplus \bigoplus_{A\in I}M_{As}^{\emptyset}\otimes (B_s)_{s}^{\emptyset}$.
Since $I$ is $s$-invariant, $\bigoplus_{A\in I}M_{As}^{\emptyset}\otimes (B_s)_{s}^{\emptyset} = \bigoplus_{A\in I}M_{A}^{\emptyset}\otimes (B_s)_s^{\emptyset}$.
Hence $(M*B_s)_{I}^{\emptyset} = \bigoplus_{A\in I}M_{A}^{\emptyset}\otimes ((B_s)_{e}^{\emptyset}\oplus (B_s)_{s}^{\emptyset}) = \bigoplus_{A\in I}M_{A}^{\emptyset}\otimes B_{s}^{\emptyset} = M_I^{\emptyset}\otimes B_{s}^{\emptyset}$.
\end{proof}

\begin{lem}\label{lem:stalk of M*B_s}
Let $A\in \mathcal{A}$ such that $As < A$ and $I$ (resp.\ $J$) be an $s$-invariant closed (resp.\ open) subset such that $I\cap J = \{A,As\}$.
Set $N = M*B_s$.
Then we have
\[
N_{I\setminus \{As\}}/N_{I\setminus\{A,As\}}\simeq M_{\{A,As\}}(-1),\quad 
N_{I}/N_{I\setminus\{As\}}\simeq M_{\{A,As\}}(1).
\]
as left $S$-modules.
\end{lem}
\begin{proof}
First we note that $I\setminus\{A,As\} = I\setminus J$ and  $I\setminus\{As\} = (I\setminus J)\cup\{A'\in \mathcal{A}\mid A'\ge A\}$ are closed.
We have an exact sequence
\begin{equation}\label{eq:exact sequence of stalk of M*B_s}
0\to N_{I\setminus\{As\}}/N_{I\setminus\{A,As\}}\to N_{I}/N_{I\setminus\{A,As\}}\to N_{I}/N_{I\setminus\{As\}}\to 0.
\end{equation}

We have $(N_I/N_{I\setminus\{A,As\}})^{\emptyset} = N_{A}^{\emptyset}\oplus N_{As}^{\emptyset}$ and we have the following commutative diagram:
\[
\begin{tikzcd}
0\arrow[r] & (N_{I\setminus\{As\}}/N_{I\setminus\{A,As\}})^{\emptyset}\arrow[r]\arrow[d,dash,"\sim",sloped] & (N_{I}/N_{I\setminus\{A,As\}})^{\emptyset}\arrow[r]\arrow[d,dash,"\sim",sloped] & (N_{I}/N_{I\setminus\{As\}})^{\emptyset}\arrow[r]\arrow[d,dash,"\sim",sloped] & 0\\
0\arrow[r] & N^{\emptyset}_{A}\arrow[r] & N^{\emptyset}_{A}\oplus N^{\emptyset}_{As}\arrow[r] & N_{As}^{\emptyset}\arrow[r] & 0.
\end{tikzcd}
\]
Therefore $N_{I\setminus\{As\}}/N_{I\setminus\{A,As\}} = (N_I/N_{I\setminus\{A,As\}})\cap (N_A^{\emptyset}\oplus 0)$.

Set $L = N_I/N_{I\setminus \{A,As\}}$.
By Lemma~\ref{lem:(M*B_s)_I for s-invariant I}, $L\simeq M_{\{A,As\}}\otimes_{R^s}R(1)$.
We have $L^{\emptyset} = L_{A}^{\emptyset}\oplus L_{As}^{\emptyset}$.
We determine $L\cap (L_{A}^{\emptyset}\oplus 0)$.

By \eqref{eq:decomposition of M*B_s}, we have $L_{A}^{\emptyset}\simeq M_{A}^{\emptyset}\oplus M_ {As}^{\emptyset}$ and $L_{As}^{\emptyset}\simeq M_{As}^{\emptyset}\oplus M_ {A}^{\emptyset}$.
In general, we write $m_{A'}$ for the image of $m\in M$ in $M_{A'}^{\emptyset}$ where $A'\in \mathcal{A}$.
The image of $m_1\otimes 1 + m_2\otimes \delta\in L = M_{\{A,As\}}\otimes_{R^s}R(1)$ in each direct summand is
\begin{gather*}
m_{1,A} + m_{2,A}\delta\in M_{A}^{\emptyset}\subset L_{A}^{\emptyset},\\
m_{1,As} + m_{2,As}s(\delta)\in M_{As}^{\emptyset}\subset L_{A}^{\emptyset},\\
m_{1,As} + m_{2,As}\delta\in M_{As}^{\emptyset}\subset L_{As}^{\emptyset},\\
m_{1,A} + m_{2,A}s(\delta)\in M_{A}^{\emptyset}\subset L_{As}^{\emptyset}.
\end{gather*}
Therefore $m_1\otimes 1 + m_2\otimes \delta\in L_{A}^{\emptyset}$ if and only if $m_{1,As} + m_{2,As}\delta = 0$, $m_{1,A} + m_{2,A}s(\delta) = 0$.
Note that $m_{2,As}\delta = (s(\delta))_Am_{2,As}$ and $m_{2,A}s(\delta) = (s(\delta))_{A}m_{2,A}$.
Therefore $(m_1 + (s(\delta))_Am_2)_{A'} = 0$ for $A' = A,As$.
Hence $m_1 + (s(\delta))_Am_2 = 0$.
Therefore we have
\[
L\cap (L_{A}^{\emptyset}\oplus 0) = \{m_2\otimes \delta - (s(\delta))_Am_2\otimes 1\mid m_2\in M_{\{A,As\}}\}(1)
\]
which is isomorphic to $M_{\{A,As\}}(-1)$.

The map $L\simeq M_{\{A,As\}}\otimes_{R^s}R(1)\ni m\otimes f\mapsto (s(f))_Am\in M_{\{A,As\}}(1)$ is surjective and, by the above argument, the kernel is $L\cap (L_{A}^{\emptyset}\oplus 0)\simeq N_{I\setminus\{As\}}/N_{I\setminus\{A,As\}}$.
Therefore by the exact sequence \eqref{eq:exact sequence of stalk of M*B_s}, we have $N_{I}/N_{I\setminus\{As\}}\simeq M_{\{A,As\}}(1)$.
\end{proof}

\begin{lem}\label{lem:stalk of M*B_s, any closed subset}
Let $A\in \mathcal{A}$ such that $As < A$, $I$ a closed subset and $J$ an open subset.
Then we have the following.
\begin{enumerate}
\item If $I\cap J = \{As\}$, then $(M*B_s)_I/(M*B_s)_{I\setminus J}\simeq M_{\{A,As\}}(1)$ as left $S$-modules.
\item If $I\cap J = \{A\}$, then $(M*B_s)_I/(M*B_s)_{I\setminus J}\simeq M_{\{A,As\}}(-1)$ as left $S$-modules.
\end{enumerate}
\end{lem}
\begin{proof}
Set $N = M*B_s\in\widetilde{\mathcal{K}}'$.

\noindent (1)
Put $I_1 = \{A'\in \mathcal{A}\mid A'\ge As\}$.
This is $s$-invariant.
Since $I$ is closed and contains $As$, we have $I_1\subset I$.
Hence $N_{I_1}/N_{I_1\setminus \{As\}}\hookrightarrow N_{I}/N_{I\setminus\{As\}}$.
By Lemma~\ref{lem:stalk of M*B_s}, we have $N_{I_1}/N_{I_1\setminus \{As\}}\simeq M_{\{A,As\}}(-1)$.
Hence we have $M_{\{A,As\}}(-1)\hookrightarrow N_{I}/N_{I\setminus\{As\}}$.

Let $\nu\in X^\vee_{\Coeff}$ and $S_{(\nu)}$ the localization at the prime ideal $(\nu)$.
Set $N_{(\nu)} = S_{(\nu)}\otimes_S N$.
The algebra $S_{(\nu)}$ is an $S^{\alpha}$-algebra for a certain $\alpha\in\Delta$.
Therefore $N_{(\nu)}$ satisfies \property{S}.
Hence the above embedding $(M_{(\nu)})_{\{A,As\}}(-1)\hookrightarrow (N_{(\nu)})_{I}/(N_{(\nu)})_{I\setminus\{As\}}$ is an isomorphism.
Since $M$ admits a standard filtration, $M_{\{A,As\}}$ is graded free as an $S$-module.
Therefore $M_{\{A,As\}}(-1) = \bigcap_{\nu\in X_{\Coeff}}(S_{(\nu)}\otimes_S M_{\{A,As\}}(-1)) = \bigcap_{\nu\in X_{\Coeff}}((N_{(\nu)})_{I}/(N_{(\nu)})_{I\setminus\{As\}})\supset N_{I}/N_{I\setminus\{As\}}$.
We get the lemma.

\noindent (2)
First we prove that there exists an embedding $(M*B_s)_I/(M*B_s)_{I\setminus J}\hookrightarrow M_{\{A,As\}}(-1)$.
We may assume $J = \{A'\in \mathcal{A}\mid A'\le A\}$ since $I\setminus J$ is not changed.
Then $J$ is $s$-invariant.
Put $I_1 = I\cup Is$.
Then $I_1$ is an $s$-invariant closed subset and $I_1\cap J = (I\cap J)\cup (Is\cap J) = (I\cap J)\cup(I\cap J)s = \{A,As\}$.
We have $I_1\setminus\{As\}\supset I$.
Hence we have an embedding $N_I/N_{I\setminus J}\hookrightarrow N_{I_1\setminus\{As\}}/N_{I_1\setminus\{A,As\}}\simeq M_{\{A,As\}(-1)}$.
We prove that this embedding is surjective.

First we assume that $\Coeff$ is a field.
Take a sequence of closed subsets $I_0 \subset \dotsm \subset I_r$ such that $\#(I_{i + 1}\setminus I_i) = 1$, $N_{I_0} = 0$, $N_{I_r} = N$ and there exists $k = 1,\dots,r$ such that $I_{k - 1}\cap \supp_{\mathcal{A}}(N) = I\cap \supp_{\mathcal{A}}(N)$ and $I_k = I_{k - 1}\cup \{A\}$ (Lemma~\ref{lem:filtration for FD module}).
Let $A_i\in \mathcal{A}$ such that $I_i = I_{i - 1}\cup \{A_i\}$.
Since $N_{I_i}$ is a filtration of $N$, for each $l$, the $l$-th graded piece $N^{l}$ satisfies $\dim_{\Coeff} N^{l} = \sum_i (N_{I_{i}}/N_{I_{i - 1}})^{l}$.
By (1) and the existence of an embedding we proved, $\dim_{\Coeff}(N_{I_{i}}/N_{I_{i - 1}})^{l}\le \dim_{\Coeff}(M_{\{A_i,A_is\}})^{l + \varepsilon(A_i)}$ where $\varepsilon(A_i) = 1$ if $A_is > A_i$ and $\varepsilon(A_i) = -1$ otherwise.
We have
\begin{align*}
& \dim_{\Coeff}(M_{\{A_i,A_is\}})^{l + \varepsilon(A_i)}\\
& = \sum_i(\dim_{\Coeff}(M_{\{A_i\}})^{l+ \varepsilon(A_i)} + \dim_{\Coeff}(M_{\{A_is\}})^{l + \varepsilon(A_i)})\\
& = 
\sum_{A_is > A_i}\dim_{\Coeff}(M_{\{A_i\}}^{l + 1}) + \sum_{A_is > A_i}\dim_{\Coeff}(M_{\{A_is\}}^{l + 1})\\
& \quad + \sum_{A_is < A_i}\dim_{\Coeff}(M_{\{A_i\}}^{l - 1})+ \sum_{A_is < A_i}\dim_{\Coeff}(M_{\{A_is\}}^{l - 1}).
\end{align*}
By replacing $A_i$ with $A_is$ in the second and fourth sum, we have
\begin{align*}
& \sum_i(\dim_{\Coeff}(M_{\{A_i\}})^{l+ \varepsilon(A_i)} + \dim_{\Coeff}(M_{\{A_is\}})^{l + \varepsilon(A_i)})\\
& = 
\sum_{A_is > A_i}\dim_{\Coeff}(M_{\{A_i\}}^{l + 1}) + \sum_{A_is < A_i}\dim_{\Coeff}(M_{\{A_i\}}^{l + 1}))\\
& \quad + \sum_{A_is < A_i}\dim_{\Coeff}(M_{\{A_i\}}^{l - 1}) + \sum_{A_is > A_i}\dim_{\Coeff}(M_{\{A_i\}}^{l - 1})\\
& = \sum_i (\dim_{\Coeff} M_{\{A_i\}}^{l + 1} + \dim_{\Coeff} M_{\{A_i\}}^{l - 1}).
\end{align*}
Since $\{M_{\{A_i\}}\}$ are subquotients of a filtration $\{M_{I_i}\}$ on $M$, we have $\sum_i\dim_{\Coeff} (M_{\{A_i\}})^{l'} = \dim_{\Coeff} M^{l'}$.
Hence $\sum_i (\dim_{\Coeff} M_{\{A_i\}}^{l + 1} + \dim_{\Coeff} M_{\{A_i\}}^{l - 1}) = \dim_{\Coeff} M^{l + 1} + \dim_{\Coeff} M^{l - 1}$.

On the other hand, since $N = M*B_s = M\otimes_{R^s}R(1) = M(1)\otimes 1\oplus M(1)\otimes \delta_s$ where $\delta_s$ satisfies $\langle\delta_s,\alpha_s^\vee\rangle = 1$, we have $\dim_{\Coeff} N^l = \dim_{\Coeff} M^{l + 1} + \dim_{\Coeff} M^{l - 1}$.
Therefore we get
\[
\dim_{\Coeff} N^l = \sum_i\dim_{\Coeff}(N_{I_{i}}/N_{I_{i - 1}})^{l}\le \sum_i\dim_{\Coeff}(M_{\{A_i,A_is\}})^{l + \varepsilon(A_i)} = \dim_{\Coeff} N^l.
\]
Hence the embedding has to be a bijection

Now let $\Coeff$ be a general Noetherian integral domain.
Assume that we can prove that $(N_{I_i}/N_{I_{i - 1}})\otimes_{\Coeff}(\Coeff/\mathfrak{m})\simeq (M_{\{A_i,A_is\}}(\varepsilon(A_i)))\otimes_{\Coeff}(\Coeff/\mathfrak{m})$ for each maximal ideal $\mathfrak{m}$ in $\Coeff$.
Since $M_{\{A_i,A_is\}}^l$ is finitely generated as a $\Coeff$-module, by Nakayama's lemma, $(N_{I_i}/N_{I_{i - 1}})^l_{\mathfrak{m}}\to (M_{\{A_i,A_{i}s\}})^{l + \varepsilon(A_i)}_{\mathfrak{m}}$ is surjective where $(\bullet)_{\mathfrak{m}}$ means the localization at $\mathfrak{m}$.
Since this is true for any maximal ideal $\mathfrak{m}$, the map $(N_{I_i}/N_{I_{i - 1}})^l\to M^l_{\{A_i,A_{i}s\}}$ is surjective for any $l\in\Z$, hence it is an isomorphism.
Therefore it is sufficient to prove $(N_{I_i}/N_{I_{i - 1}})\otimes_{\Coeff}(\Coeff/\mathfrak{m})\simeq (M_{\{A_i,A_is\}}(\varepsilon(A_i)))\otimes_{\Coeff}(\Coeff/\mathfrak{m})$.
In the rest of the proof, we omit the grading.

To prove this, we need some properties on the base change to $\Coeff/\mathfrak{m}$.
Let $L\in \widetilde{\mathcal{K}}'$.
Then we have $L\otimes_{\Coeff}(\Coeff/\mathfrak{m})$ is a $(S/\mathfrak{m}S,R/\mathfrak{m}R)$-bimodule and we have $S^{\emptyset}\otimes_{S}L\otimes_{\Coeff}(\Coeff/\mathfrak{m}) \simeq \bigoplus_{A\in \mathcal{A}}L^{\emptyset}_{A}\otimes_{\Coeff}(\Coeff/\mathfrak{m})$.
Therefore it defines an object in $\widetilde{\mathcal{K}}'_{\Coeff/\mathfrak{m}}$, here the suffix $\Coeff/\mathfrak{m}$ means that in the definition of $\widetilde{\mathcal{K}}'$ we replace $\Coeff$ with $\Coeff/\mathfrak{m}$.
Let $K\subset \mathcal{A}$ be a closed subset.
Then we have a map $L_K\otimes_{\Coeff}(\Coeff/\mathfrak{m})\to L\otimes_{\Coeff}(\Coeff/\mathfrak{m})$.
Since $\supp_{\mathcal{A}}(L_K\otimes_{\Coeff}(\Coeff/\mathfrak{m}))\subset K$, the image of this homomorphism is in $(L\otimes_{\Coeff}(\Coeff/\mathfrak{m}))_K$.
Hence we get a map $L_K\otimes_{\Coeff}(\Coeff/\mathfrak{m})\to (L\otimes_{\Coeff}(\Coeff/\mathfrak{m}))_K$.
We claim:
\begin{enumerate}
\item The map is surjective.
\item If $L/L_K$ is graded free, then this map is an isomorphism.
\end{enumerate}
We prove (1) first.
By the exact sequence $0\to L_K\to L\to L/L_K\to 0$, we have an exact sequence $L_K\otimes_{\Coeff}(\Coeff/\mathfrak{m})\to L\otimes_{\Coeff}(\Coeff/\mathfrak{m})\to (L/L_K)\otimes_{\Coeff}(\Coeff/\mathfrak{m})\to 0$.
Since $\supp_{\mathcal{A}}((L/L_K)\otimes_{\Coeff}(\Coeff/\mathfrak{m}))\subset \mathcal{A}\setminus K$, the map $L\otimes_{\Coeff}(\Coeff/\mathfrak{m})\to (L/L_K)\otimes_{\Coeff}(\Coeff/\mathfrak{m})$ factors through $L\otimes_{\Coeff}(\Coeff/\mathfrak{m})\to (L\otimes_{\Coeff}(\Coeff/\mathfrak{m}))/(L\otimes_{\Coeff}(\Coeff/\mathfrak{m}))_K$.
Hence $(L\otimes_{\Coeff}(\Coeff/\mathfrak{m}))_K\subset \Ker (L\otimes_{\Coeff}(\Coeff/\mathfrak{m})\to (L/L_K)\otimes_{\Coeff}(\Coeff/\mathfrak{m})) = \Ima(L_K\otimes_{\Coeff}(\Coeff/\mathfrak{m})\to L\otimes_{\Coeff}(\Coeff/\mathfrak{m}))$.
Therefore we get (1).
If $L/L_K$ is graded free, then $L/L_K$ is free as a $\Coeff$-module.
Hence $L_K\otimes_{\Coeff}(\Coeff/\mathfrak{m})\to L\otimes_{\Coeff}(\Coeff/\mathfrak{m})$ is injective.
Therefore we have (2).

In particular, if $L$ satisfies \property{S}, then $L\otimes_{\Coeff}(\Coeff/\mathfrak{m})$ also satisfies \property{S}.
Indeed, let $K_1,K_2$ be closed subsets.
Then we have a commutative diagram
\[
\begin{tikzcd}
L_{K_1}\otimes_{\Coeff}(\Coeff/\mathfrak{m})\oplus L_{K_2}\otimes_{\Coeff}(\Coeff/\mathfrak{m})\arrow[r,twoheadrightarrow]\arrow[d,twoheadrightarrow] & (L\otimes_{\Coeff}(\Coeff/\mathfrak{m}))_{K_1}\oplus (L\otimes_{\Coeff}(\Coeff/\mathfrak{m}))_{K_2}\arrow[d]\\
L_{K_1\cup K_2}\otimes_{\Coeff}(\Coeff/\mathfrak{m})\arrow[r,twoheadrightarrow] & (L\otimes_{\Coeff}(\Coeff/\mathfrak{m}))_{K_1\cup K_2}.
\end{tikzcd}
\]
Here the horizontal maps are surjective by (1) in the above and the left vertical map is surjective since $L$ satisfies \property{S}.
Hence the right vertical maps is surjective and it means that $L\otimes_{\Coeff}(\Coeff/\mathfrak{m})$ satisfies \property{S}.

We also have that if $L$  satisfies \property{LE} then $L\otimes_{\Coeff}(\Coeff/\mathfrak{m})$ satisfies \property{LE}.
Let $\alpha\in \Delta$ and decompose $L^{\alpha}$ as $L^{\alpha} = \bigoplus_{\Omega\in W'_{\alpha}\backslash \mathcal{A}}L_{\Omega}$ such that $\supp L_{\Omega}\subset \Omega$.
Then $(L\otimes_{\Coeff}(\Coeff/\mathfrak{m}))^{\alpha}\simeq \bigoplus_{\Omega\in W'_{\alpha}\backslash \mathcal{A}}L_{\Omega}\otimes_{\Coeff}(\Coeff/\mathfrak{m})$ and it gives a desired decomposition in \property{LE}.

Let $K_1\subset K_2\subset \mathcal{A}$ be closed subsets and $L\in \widetilde{\mathcal{K}}_{\Delta}$.
Since $L\in \widetilde{\mathcal{K}}_{\Delta}$, $L/L_{K_1}$ and $L/L_{K_2}$ are both graded free.
Hence $L_{K_1}\otimes_{\Coeff}(\Coeff/\mathfrak{m})\simeq (L\otimes_{\Coeff}(\Coeff/\mathfrak{m}))_{K_1}\subset (L\otimes_{\Coeff}(\Coeff/\mathfrak{m}))_{K_2}\simeq L_{K_2}\otimes_{\Coeff}(\Coeff/\mathfrak{m})$.
By the right exactness of the tensor product, we have $(L_{K_2}/L_{K_1})\otimes_{\Coeff}(\Coeff/\mathfrak{m})\simeq (L_{K_2}\otimes_{\Coeff}(\Coeff/\mathfrak{m}))/(L_{K_1}\otimes_{\Coeff}(\Coeff/\mathfrak{m}))\simeq (L\otimes_{\Coeff}(\Coeff/\mathfrak{m}))_{K_2}/(L\otimes_{\Coeff}(\Coeff/\mathfrak{m}))_{K_1}$.
Therefore, for any locally closed subset $K\subset \mathcal{A}$, we have $L_K\otimes_{\Coeff}(\Coeff/\mathfrak{m})\simeq (L\otimes_{\Coeff}(\Coeff/\mathfrak{m}))_K$.
In particular, $L\otimes_{\Coeff}(\Coeff/\mathfrak{m})\in \widetilde{\mathcal{K}}_{\Delta,\Coeff/\mathfrak{m}}$.

We return to the proof of the lemma.
We have $M_{\{A_i,A_is\}}\otimes_{\Coeff}(\Coeff/\mathfrak{m})\simeq (M\otimes_{\Coeff}(\Coeff/\mathfrak{m}))_{\{A_i,A_is\}}$.
Hence we have the following commutative diagram
\[
\begin{tikzcd}
(N_{I_i}/N_{I_{i - 1}})\otimes_{\Coeff}(\Coeff/\mathfrak{m})\arrow[r]\arrow[d] & M_{\{A_i,A_is\}}\otimes_{\Coeff}(\Coeff/\mathfrak{m})\arrow[d,"\sim" sloped]\\
(N\otimes_{\Coeff}(\Coeff/\mathfrak{m}))_{I_i}/(N\otimes_{\Coeff}(\Coeff/\mathfrak{m}))_{I_{i - 1}}\arrow[r,"\sim" sloped] & (M\otimes_{\Coeff}(\Coeff/\mathfrak{m}))_{\{A_i,A_is\}}.
\end{tikzcd}
\]
Note that the bottom homomorphism is an isomorphism since the lemma is proved if $\Coeff$ is a field.

We prove that the left vertical map is an isomorphism by backward induction on $i$.
By inductive hypothesis, $N_{I_{i'}}/N_{I_{i' - 1}}\simeq M_{\{A_{i'},A_{i'}s\}}$ for any $i' > i$ and in particular it is graded free.
Hence $N/N_{I_i}$ is also graded free.
Therefore we have $N_{I_i}\otimes_{\Coeff}(\Coeff/\mathfrak{m})\simeq (N\otimes_{\Coeff}(\Coeff/\mathfrak{m}))_{I_i}$.
Now we get the desired result by applying the five lemma to the following commutative diagram with exact columns
\[
\begin{tikzcd}
& 0\arrow[d]\\
N_{I_{i - 1}}\otimes_{\Coeff}(\Coeff/\mathfrak{m})\arrow[d]\arrow[r,twoheadrightarrow] & (N\otimes_{\Coeff}(\Coeff/\mathfrak{m}))_{I_{i - 1}}\arrow[d]\\
N_{I_i}\otimes_{\Coeff}(\Coeff/\mathfrak{m})\arrow[d]\arrow[r,"\sim"] & (N\otimes_{\Coeff}(\Coeff/\mathfrak{m}))_{I_i}\arrow[d] \\
(N_{I_{i}}/N_{I_{i - 1}})\otimes_{\Coeff}(\Coeff/\mathfrak{m})\arrow[d]\arrow[r] & (N\otimes_{\Coeff}(\Coeff/\mathfrak{m}))_{I_i}/(N\otimes_{\Coeff}(\Coeff/\mathfrak{m}))_{I_{i - 1}}\arrow[d]\\
0 & 0.
\end{tikzcd}
\]
%
%
%
\end{proof}

\begin{lem}\label{lem:(M*B_s)_K is graded free}
Set $N = M*B_s$.
Then for each closed subset $I_1\supset I_2$, $N_{I_1}/N_{I_2}$ is a graded free $S$-module.
\end{lem}
\begin{proof}
Take $A_0,A_1\in \mathcal{A}$ such that $\supp_{\mathcal{A}}N\subset [A_0,A_1]$.
Replacing $I_1$ with $I_1\cap \{A\in \mathcal{A}\mid A\ge A_0\}$ and $I_2$ with $I_2\cup\{A\in \mathcal{A}\mid A\not\le A_1\}$, we may assume $I_1\setminus I_2$ is finite.
We can take a sequence of closed subsets $I_2 = I'_0 \subset I'_1\subset\dots\subset I'_r = I_1$ such that $\#(I'_{i}\setminus I'_{i - 1}) = 1$.
Let $A_i$ such that $I'_{i} = I'_{i - 1}\cup \{A_i\}$.
Then by Lemma~\ref{lem:stalk of M*B_s, any closed subset}, $N_{I'_i}/N_{I'_{i - 1}}\simeq M_{\{A_i,A_is\}}(\varepsilon(A_i))$ where $\varepsilon(A_i)\in \{\pm 1\}$ is as in the proof of Lemma~\ref{lem:stalk of M*B_s, any closed subset}.
In particular this is graded free and hence $M_{I_1}/M_{I_2} = M_{I'_r}/M_{I'_0}$ is also graded free.
\end{proof}

\begin{proof}[Proof of Proposition~\ref{prop:K_Delta is stable under the Hecke actions}]
Set $N = M*B_s$.
We prove that $N$ satisfies \property{S}.
Let $I_1,I_2$ are closed subsets and we prove the surjectivity of $N_{I_1}/N_{I_1\cap I_2}\hookrightarrow N_{I_1\cup I_2}/N_{I_2}$.
For each $\nu\in X^\vee_{\Coeff}$, let $S_{(\nu)}$ be the localization at the prime ideal $(\nu)$.
Then $N_{(\nu)} = S_{(\nu)}\otimes_{S}N$ satisfies \property{S}.
Hence this embedding is surjective after applying $S_{(\nu)}\otimes_S$.
We denote $L_{(\nu)} = S_{(\nu)}\otimes_{S}L$ for a left $S$-module $L$.
Since $N_{I_1}/N_{I_1\cap I_2}$ is graded free by Lemma~\ref{lem:(M*B_s)_K is graded free}, we have $N_{I_1}/N_{I_1\cap I_2} = \bigcap_{\nu} (N_{I_1}/N_{I_1\cap I_2})_{(\nu)}$.
Hence $N_{I_1}/N_{I_1\cap I_2} = \bigcap_{\nu} (N_{I_1}/N_{I_1\cap I_2})_{(\nu)} = \bigcap_{\nu} (N_{I_1\cup I_2}/N_{I_2})_{(\nu)}\supset N_{I_1\cup I_2}/N_{I_2}$.
We get the surjectivity.

Now $N_{\{A\}}$ is well-defined and isomorphic to $M_{\{A,As\}}(\varepsilon(A))$ where $\varepsilon(A)\in \{\pm 1\}$ is as in the proof of Lemma~\ref{lem:stalk of M*B_s, any closed subset}.
Hence $N_{\{A\}}$ is graded free, namely $N$ admits a standard filtration.
\end{proof}

As a consequence of Lemma~\ref{lem:stalk of M*B_s, any closed subset}, we get the following corollary.
\begin{cor}\label{cor:Hecke action}
If $M\in \widetilde{\mathcal{K}}_{\Delta}$, then we have
\[
(M*B_s)_{\{A\}} \simeq
\begin{cases}
M_{\{A,As\}}(-1) & (As < A),\\
M_{\{A,As\}}(1) & (As > A).
\end{cases}
\]
Therefore we have
\[
\grk((M*B_s)_{\{A\}})
= 
\begin{cases}
v^{-1}(\grk(M_{\{A\}}) + \grk(M_{\{As\}})) & (As < A),\\
v(\grk(M_{\{A\}}) + \grk(M_{\{As\}})) & (As > A)
\end{cases}
\]
for each $A\in \mathcal{A}$ and $s\in S_\aff$.
\end{cor}

The action of $\Sbimod$ preserves $\widetilde{\mathcal{K}}_P$ too.
\begin{prop}\label{prop:K_P is stable under the Hecke actions}
We have $\widetilde{\mathcal{K}}_P*\Sbimod\subset \widetilde{\mathcal{K}}_P$.
\end{prop}
\begin{proof}
Let $M\in \widetilde{\mathcal{K}}_P$ and $s\in S_\aff$.
We prove $M* B_s\in \widetilde{\mathcal{K}}_{P}$.
We have already proved that $M*B_s\in \widetilde{\mathcal{K}}_{\Delta}$.

Assume that a sequence $M_1\to M_2\to M_3$ in $\widetilde{\mathcal{K}}_{\Delta}$ satisfies \property{ES}.
By Lemma~\ref{lem:exact sequence for any locally closed subset}, $0\to (M_1)_{\{A,As\}}\to (M_2)_{\{A,As\}}\to (M_3)_{\{A,As\}}\to 0$ is also exact for any $A\in \mathcal{A}$.
Hence $0\to (M_1*B_s)_{\{A\}}\to (M_2*B_s)_{\{A\}}\to (M_3*B_s)_{\{A\}}\to 0$ is exact.
Namely $M_1*B_s\to M_2*B_s\to M_3*B_s$ also satisfies \property{ES}.
Since $M\in \widetilde{\mathcal{K}}_P$, the sequence $0\to \Hom^{\bullet}(M,M_1*B_s)\to \Hom^{\bullet}(M,M_2*B_s)\to \Hom^{\bullet}(M,M_3*B_s)\to 0$ is exact.
By Proposition~\ref{prop:adjointness}, $M*B_s\in \widetilde{\mathcal{K}}_P$.
\end{proof}

\subsection{Indecomposable objects}
Assume that $\Coeff$ is complete local.
For $M,N\in \widetilde{\mathcal{K}}'$, $\Hom^{\bullet}_S(M,N)$ is finitely generated as an $S$-module since $M,N$ are finitely generated and $S$ is Noetherian.
Hence $\Hom^{\bullet}_{\widetilde{\mathcal{K}}'}(M,N)\subset \Hom^{\bullet}_S(M,N)$ is also finitely generated.
Therefore, $\Hom_{\widetilde{\mathcal{K}}'}(M,N)$ is finitely generated $\Coeff$-module.
Hence $\widetilde{\mathcal{K}}'$ has Krull-Schmidt property.
This is also true for $\widetilde{\mathcal{K}}_P$.

Set $(\R\Delta)_{\integer} = \{\lambda\in \R\Delta\mid \langle\lambda,\Delta^\vee\rangle\subset\Z\}$ be the set of integral weights.
For $\lambda\in (\R\Delta)_{\integer}$, let $\Pi_{\lambda}$ be the set of alcoves $A$ such that $\langle \lambda,\alpha^\vee\rangle - 1 < \langle a,\alpha^\vee\rangle < \langle \lambda,\alpha^\vee\rangle$ for any $a\in A$ and simple root $\alpha$.
The set $\Pi_{\lambda}$ is called a box and each $A\in \mathcal{A}$ is contained in a box.
Each $\Pi_{\lambda}$ has the unique maximal element $A_{\lambda}^-$.
Let $W'_{\lambda} = \Stab_{W'_\aff}(\lambda)$ be the stabilizer.
Then $A_{\lambda}^-$ is the minimal element in $W_{\lambda}'A_{\lambda}^-$.
The set $W'_{\lambda}A_{\lambda}^-$ is the set of alcoves whose closure contains $\lambda$.

We define $Q_{\lambda}\in\widetilde{\mathcal{K}}$ as follows.
Consider the orbit $W'_{\lambda}A_{\lambda}^-$ through $A_{\lambda}^{-}$.
As an $(S,R)$-bimodule, it is given by
\[
Q_{\lambda} = \{(z_A)\in S^{W'_{\lambda}A_{\lambda}^-}\mid \text{$z_A\equiv z_{s_{\alpha,\langle\lambda,\alpha^\vee\rangle}A}\pmod{\alpha^\vee}$ for $\alpha\in\Delta$ and $A\in W'_{\lambda}A_{\lambda}^-$}\}
\]
where the right action of $R$ is given by $(z_A)f  = (f_Az_A)$.
We have $Q_{\lambda}^{\emptyset} = (S^{\emptyset})^{W'_{\lambda}A_{\lambda}^-}$.
The module $(Q_{\lambda})^{\emptyset}_{A}$ is the $A$-component if $A\in W'_{\lambda}A_{\lambda}^-$ and $0$ otherwise.

The definition of $Q_{\lambda}$ comes from the structure sheaf of the moment graph associated to $W_{\mathrm{f}}$.
The structure sheaf is defined by
\[
\mathcal{Z} = \{(z_x)_{x\in W_{\mathrm{f}}}\in S^{W_{\mathrm{f}}}\mid z_x\equiv z_{s_{\alpha}x}\pmod{\alpha^\vee}\}.
\]
The natural map $W'_{\lambda}\hookrightarrow W'_{\aff}\to W_{\mathrm{f}}$ is an isomorphism.
The map $W_{\mathrm{f}}\simeq W'_{\lambda}\xrightarrow{w\mapsto w(A_{\lambda}^-)}W'_{\lambda}A_{\lambda}^{-}$ is a bijection which preserves orders and by this bijection we have $\mathcal{Z}\simeq Q_{\lambda}$.

The following are well-known. (See \cite{arXiv:2004.09014} for example.)
\begin{itemize}
\item The map $S\otimes_{S^{W_{\mathrm{f}}}}S\to \mathcal{Z}$ defined by $f\otimes g\mapsto (x^{-1}(f)g)_{x\in W_{\mathrm{f}}}$ is an isomorphism.
\item Let $K\subset W_{\mathrm{f}}$ be a closed subset and $w\in K$ such that $K\setminus\{w\}$ is closed.
Put $\mathcal{Z}_{K} = \{(z_x)\in \mathcal{Z}\mid \text{$z_x = 0$ for $x\notin K$}\}$ and the same for $\mathcal{Z}_{K\setminus\{w\}}$.
Then $\mathcal{Z}_{K}/\mathcal{Z}_{K\setminus \{w\}}\simeq S(-2\ell(w_0w))$ as a left $S$-module.
\end{itemize}

Let $d\colon \mathcal{A}\times \mathcal{A}\to \Z$ be the function defined in \cite[1.4]{MR591724}.
From the second property we get the following.
\begin{lem}\label{lem:successive quotient of Q_lambda}
Let $A\in W'_{\lambda}A_{\lambda}^-$, $I\subset \mathcal{A}$ a closed subset such that $A\in I$ and $I\setminus\{A\}$ is closed.
Then we have $(Q_{\lambda})_I/(Q_{\lambda})_{I\setminus \{A\}}\simeq S(2d(A,A_{\lambda}^-))$.
In particular, we have $Q_{\lambda}\in \mathcal{K}_{\Delta}$.
\end{lem}

\begin{lem}
Let $S_0$ be a commutative flat graded $S$-algebra.
We have $\Hom^{\bullet}_{\widetilde{\mathcal{K}}'(S_0)}(S_0\otimes_{S}Q_{\lambda},M)\simeq M_{\{A'\in \mathcal{A}\mid A'\ge A_{\lambda}^-\}}$.
Therefore $S_0\otimes_{S}Q_{\lambda}\in \widetilde{\mathcal{K}}_P(S_0)$.
\end{lem}
\begin{proof}
Since $S_0$ is flat, we have 
\[
S_0\otimes_{S}Q_{\lambda} = \{(z_A)\in S_0^{W_{\mathrm{f}}A_{0}^-}\mid \text{$z_A\equiv z_{s_{\alpha,\langle\lambda,\alpha^\vee\rangle}A}\pmod{\alpha^\vee}$ for $\alpha\in\Delta$ and $A\in W_{\mathrm{f}}A_{\lambda}^-$}\}.
\]
Put $I = \{A'\in \mathcal{A}\mid A'\ge A_{\lambda}^{-}\}$ and $q = (1)_{A\in W_{\mathrm{f}}A_{\lambda}^-}\in S_0\otimes_{S}Q_{\lambda}$.

Any $(S_0,R)$-bimodule is regarded as an $S_0\otimes R$-module.
Let $M\in \widetilde{\mathcal{K}}_{\Delta}(S_0)$ and $m\in M$.
According to the decomposition $M^{\emptyset} = \bigoplus_{A\in \mathcal{A}}M_A^{\emptyset}$, $m$ can be written as $m = \sum_{A\in \mathcal{A}}m_A$.
Consider $S^{W_{\mathrm{f}}} = \{f\in S\mid \text{$w(f) = f$ for all $w\in W_{\mathrm{f}}$}\}$.
Then we have the following.
\begin{itemize}
\item For $A\in \mathcal{A}$ and $f\in S^{W_{\mathrm{f}}}$, $f^A$ dose not depend on $A$.
\item For $f\in S$, we have $fm = \sum fm_A = \sum m_Af^A$.
\end{itemize}
Therefore we have an embedding $S^{W_{\mathrm{f}}}\hookrightarrow R$ naturally and any $M$ is an $S_0\otimes_{S^{W_\mathrm{f}}}R$-module.
Then we have a map $S\otimes_{S^{W_{\mathrm{f}}}}R\to Q_{\lambda}$ defined by $f\otimes g\mapsto (fg_{w(A_{\lambda}^-)})$ and by the property of $\mathcal{Z}$ we have remarked, this is an isomorphism.
Therefore $Q_{\lambda}$ is a free $S\otimes_{S^{W_\mathrm{f}}}R$-module of rank one with a basis $q$.
We also remark that $q\in S_{0}\otimes_{S}Q_{\lambda} = (S_{0}\otimes_{S}Q_{\lambda})_{I}$.
Therefore $\varphi\mapsto \varphi(q)$ gives an embedding
\[
\Hom^{\bullet}_{\widetilde{\mathcal{K}}_{\Delta}(S_0)}(S_0\otimes_{S}Q_{\lambda},M)\hookrightarrow M_I.
\]
Let $m\in M_I$ and $\varphi\colon S_0\otimes_{S}Q_{\lambda}\to M$ be an $(S_0,R)$-bimodule homomorphism such that $\varphi(q) = m$.
We prove that this is a morphism in $\widetilde{\mathcal{K}}(S_0)$.
Let $A\in W_{\lambda}'A_{\lambda}^-$.
Then $\varphi((Q_{\lambda})_{A}^{\emptyset})\subset \bigoplus_{A'\in A + \Z\Delta,A'\in I}M_{A'}^{\emptyset}$.
Therefore the lemma follows from the following lemma.
\end{proof}

\begin{lem}\label{lem:lemma on W_lambda-orbit}
Let $A\in W'_{\lambda}A_{\lambda}^-$.
Then $(A + \Z\Delta)\cap \{A'\in \mathcal{A}\mid A'\ge A_{\lambda}^-\} = \{A'\in A + \Z\Delta\mid A'\ge A\}$.
\end{lem}
\begin{proof}
Since $A_{\lambda}^-$ is the minimal element in $W'_{\lambda}A_{\lambda}^-$, the right hand side is contained in the left hand side.
Let $A'$ be in the left hand side.
Take $x\in W'_{\lambda}$ and $\mu\in \Z\Delta$ such that $A = x(A_{\lambda}^{-})$ and $A' = A + \mu$.
Then $A' = x(A_{\lambda}^{-}) + \mu$.
Since $A'\ge A_{\lambda}^{-}$ and $\lambda$ is in the closure of $A_{\lambda}^{-}$, we have $x(\lambda) + \mu - \lambda\in \R_{\ge 0}\Delta^+$ by Lemma~\ref{lem:order as a vector in A}.
Since $x\in W'_{\lambda} = \Stab_{W'_{\aff}}(\lambda)$, $x(\lambda) = \lambda$.
Therefore $\mu \in\R_{\ge 0}\Delta^+$.
Hence $A' = A + \mu \ge A$.
\end{proof}

Let $A\in \Pi_{\lambda}$ and take $w\in W_{\aff}$ such that $A = A_{\lambda}^{-}w$.
As in the proof of \cite[Proposition~4.2]{MR591724}, for any $x < w$ and $A'\in W'_{\lambda}A_{\lambda}^-$, we have $A'x > A_{\lambda}^-w$.
Let $w = s_1\dotsm s_l$ be a reduced expression.
Then $Q_{\lambda}*B_{s_1}*\dotsm *B_{s_l}$ satisfies the following.
\begin{lem}\label{lem:support of Q*B*..*B}
We have the following.
\begin{enumerate}
\item $(Q_{\lambda}*B_{s_1}*\dotsb*B_{s_l})_{\{A\}}\simeq S(l)$  as a left $S$-module.
\item $\supp_{\mathcal{A}}(Q_{\lambda}*B_{s_1}*\dotsb*B_{s_l})\subset \{A'\in \mathcal{A}\mid A'\ge A\}$.
\end{enumerate}
\end{lem}
\begin{proof}
The second one is obvious from what we mentioned before the lemma.
We prove (1) by induction on $l$.
Set $M = Q_{\lambda}*B_{s_1}*\dotsb * B_{s_{l - 1}}$ and $s = s_l$.
By Lemma~\ref{lem:stalk of M*B_s, any closed subset}, $(M*B_s)_{\{A\}}\simeq M_{\{A,As\}}(1)$.
By (2), $A\notin \supp_{\mathcal{A}}(M)$.
Hence $M_{\{A,As\}}\simeq M_{\{As\}}$.
Therefore $(M*B_s)_{\{A\}}\simeq M_{\{As\}}(1)$ and the inductive hypothesis implies (1).
\end{proof}

\begin{thm}\label{thm:indecomposables in widetilde(K)}
We have the following.
\begin{enumerate}
\item For any $A\in \mathcal{A}$, there exists an indecomposable object $Q(A)\in\widetilde{\mathcal{K}}_P$ such that $\supp_{\mathcal{A}}(Q(A))\subset \{A'\in \mathcal{A}\mid A'\ge A\}$ and $Q(A)_{\{A\}}\simeq S$.
Moreover, $Q(A)$ is unique up to isomorphisms.
\item Any object in $\widetilde{\mathcal{K}}_P$ is a direct sum of $Q(A)(n)$ where $A\in \mathcal{A}$ and $n\in \Z$.
\end{enumerate}
\end{thm}
\begin{proof}
Fix $s_1,\dots,s_l$ as in the above.
By Lemma~\ref{lem:support of Q*B*..*B}, there is the unique indecomposable module $Q(A)$ such that $Q(A)_{\{A\}} \simeq S$ and $Q(A)(l)$ is a direct summand of $Q_{\lambda}*B_{s_1}*\dotsb *B_{s_l}$.
It is sufficient to prove that any object $M\in\widetilde{\mathcal{K}}_P$ is a direct sum of $Q(A)(n)$'s.
By induction on the rank of $M$, it is sufficient to prove that $Q(A)(n)$ is a direct summand of $M$ for some $A\in \mathcal{A}$ and $n\in\Z$ if $M\ne 0$.

Let $M\in\widetilde{\mathcal{K}}_P$ and let $A\in\supp_{\mathcal{A}}(M)$ be a minimal element.
Then $M_{\{A\}}\ne 0$.
Since $M$ admits a standard filtration, $M_{\{A\}}$ is graded free.
Hence there exists $n$ such that $S(n)\simeq Q(A)(n)_{\{A\}}$ is a direct summand of $M_{\{A\}}$.
Let $i\colon Q(A)(n)_{\{A\}}\to M_{\{A\}}$ (resp.\ $p\colon M_{\{A\}}\to Q(A)(n)_{\{A\}}$) be the embedding from (resp.\ projection to) the direct summand.

Let $I$ be a closed subset which contains $\supp_{\mathcal{A}}(M)$ such that $I\setminus\{A\}$ is closed.
Then $I\supset \{A'\in \mathcal{A}\mid A'\ge A\}\supset \supp_{\mathcal{A}}(Q(A))$.
Therefore we have two sequences
\begin{gather*}
M_{I\setminus\{A\}}\to M_I = M\to M_{\{A\}},\\
Q(A)(n)_{I\setminus\{A\}}\to Q(A)(n)_I = Q(A)(n)\to Q(A)(n)_{\{A\}},
\end{gather*}
which satisfy \property{ES}.
Consider the homomorphism $Q(A)(n)\to Q(A)(n)_{\{A\}}\xrightarrow{i}M_{\{A\}}$.
Since $Q(A)(n)\in \widetilde{\mathcal{K}}_P$, there exists a lift $\widetilde{i}\colon Q(A)(n)\to M$ of the above homomorphism.
Similarly we have a morphism $\widetilde{p}\colon M\to Q(A)(n)$ which is a lift of $p$.
The composition $\widetilde{p}\circ\widetilde{i}\in \End(Q(A)(n))$ induces the identity on $Q(A)(n)_{\{A\}}$.
Therefore $1 - \widetilde{p}\circ\widetilde{i}$ is not a unit.
Since $Q(A)(n)$ is indecomposable, the endomorphism ring of $Q(A)(n)$ is local.
Therefore $\widetilde{p}\circ\widetilde{i}$ is an isomorphism.
Hence $Q(A)(n)$ is a direct summand of $M$.
\end{proof}

\begin{cor}
Any object in $\widetilde{\mathcal{K}}_P$ is a direct summand of a direct sum of objects of a form $Q_{\lambda}*B_{s_1}*\dotsm *B_{s_l}(n)$ where $\lambda\in (\R\Delta)_{\integer}$, $n\in\Z$ and $s_1,\dots,s_l\in S_{\aff}$.
\end{cor}
\begin{proof}
This is obvious from Theorem~\ref{thm:indecomposables in widetilde(K)} and the proof of the theorem.
\end{proof}

\begin{cor}
Let $M,N\in \widetilde{\mathcal{K}}_P$.
Then $\Hom^{\bullet}_{\widetilde{\mathcal{K}}_P}(M,N)$ is graded free of finite rank as an $S$-module.
\end{cor}
\begin{proof}
We may assume $M = Q_{\lambda} *B_{s_1}*\dotsb * B_{s_l}(n)$ for some $\lambda\in (\R\Delta)_{\integer}$, $n\in\Z$ and $s_1,\dots,s_l\in S_{\aff}$.
Hence, by Proposition~\ref{prop:adjointness}, we may assume $M = Q_{\lambda}$.
Then $\Hom^{\bullet}_{\widetilde{\mathcal{K}}_P}(M,N)\simeq N_{\{A'\in \mathcal{A}\mid A'\ge A_{\lambda}^{-}\}}$ and this is graded free since $N$ admits a standard filtration.
\end{proof}
\begin{cor}
Let $M,N\in \widetilde{\mathcal{K}}_P$.
Then for any flat commutative graded $S$-algebra $S_0$, we have $S_0\otimes_{S}\Hom^{\bullet}_{\widetilde{\mathcal{K}}_P}(M,N)\simeq \Hom^{\bullet}_{\widetilde{\mathcal{K}}_P(S_0)}(S_0\otimes_{S}M,S_0\otimes_{S}N)$.
\end{cor}
\begin{proof}
As in the proof of the previous corollary, we may assume $M = Q_{\lambda}$.
Set $I = \{A'\in \mathcal{A}\mid A'\ge A_{\lambda}^-\}$.
Then the corollary is equivalent to $S_0\otimes_{S}N_I\simeq (S_0\otimes_{S}N)_{I}$.
This is clear.
\end{proof}

\subsection{The categorification}\label{subsec:The categorification}
We follow notation of Soergel~\cite{MR1444322} for the Hecke algebra and the periodic module.
The $\Z[v,v^{-1}]$-algebra $\mathcal{H}$ is generated by $\{H_w\mid w\in W_{\aff}\}$ and defined by the following relations.
\begin{itemize}
\item $(H_s - v^{-1})(H_s + v) = 0$ for any $s\in S_{\aff}$.
\item If $\ell(w_1) + \ell(w_2) = \ell(w_1w_2)$ for $w_1,w_2\in W_{\aff}$, we have $H_{w_1w_2} = H_{w_1}H_{w_2}$.
\end{itemize}
It is well-known that $\{H_w\mid w\in W_{\aff}\}$ is a $\Z[v,v^{-1}]$-basis of $\mathcal{H}$.

Set $\mathcal{P} = \bigoplus_{A\in \mathcal{A}}\Z[v,v^{-1}]A$ and we define a right action of $\mathcal{H}$ \cite[Lemma~4.1]{MR1444322} by
\[
AH_s = 
\begin{cases}
As & (As > A),\\
As + (v^{-1} - v)A & (As < A).
\end{cases}
\]

For an additive category $\mathcal{B}$, let $[\mathcal{B}]$ be the split Grothendieck group of $\mathcal{B}$.
We have $[\Sbimod]\simeq \mathcal{H}$ \cite[Theorem~4.3]{arXiv:1901.02336_accepted} and under this isomorphism, $[B_s]\in [\Sbimod]$ corresponds to $H_s + v\in \mathcal{H}$.
By $[M][B] = [M*B]$, $[\mathcal{K}_P]$ is a right $[\Sbimod]$-module.
Fix a length function $\ell\colon \mathcal{A}\to \Z$ in the sense of \cite[2.11]{MR591724}.
Define $\ch\colon [\mathcal{K}_P]\to \mathcal{P}$ by 
\[
\ch(M) = \sum_{A\in \mathcal{A}}v^{\ell(A)}\grk(M_{\{A\}})A.
\]
Then by Corollary~\ref{cor:Hecke action}, $\ch$ is a $[\Sbimod]\simeq \mathcal{H}$-module homomorphism.

For each $\lambda\in (\R\Delta)_{\integer}$, set $e_{\lambda} = \sum_{A\in W'_{\lambda}A_{\lambda}^-}v^{-\ell(A)}A$.
We put $\mathcal{P}^{0} = \sum_{\lambda\in (\R\Delta)_{\integer}}e_{\lambda}\mathcal{H}\subset \mathcal{P}$.
\begin{lem}
We have $\ch(Q_\lambda) = v^{2\ell(A_{\lambda}^{-})}e_{\lambda}$.
\end{lem}
\begin{proof}
It follows from Lemma~\ref{lem:successive quotient of Q_lambda}.
\end{proof}

\begin{thm}\label{thm:categorification for tilde(K)}
We have $\ch\colon [\widetilde{\mathcal{K}}_P]\xrightarrow{\sim} \mathcal{P}^{0}$.
\end{thm}
\begin{proof}
Since $e_{\lambda} = v^{-2\ell(A_{\lambda}^{-})}\ch(Q_\lambda)\in \Ima(\ch)$, the image of $\ch$ is contained in $\mathcal{P}^{0}$ and it surjects to $\mathcal{P}^{0}$.
The $\mathcal{H}$-module $[\widetilde{\mathcal{K}}_P]$ has a $\Z[v,v^{-1}]$-basis $[Q(A)]$ by Theorem~\ref{thm:indecomposables in widetilde(K)}.
Since $\ch(Q(A))\in v^{\ell(A)}A + \sum_{A' > A}\Z[v,v^{-1}]A'$, $\{\ch(Q(A))\mid A\in \mathcal{A}\}$ is linearly independent.
Hence $\ch$ is injective.
\end{proof}

\subsection{A relation with a work of Fiebig-Lanini}\label{subsec:A relation with a work of Fiebig-Lanini}
In \cite{arXiv:1504.01699}, Fiebig and Lanini constructed a category denoted by $\mathbf{C}$ and proved that this is an exact category.
They also constructed a wall-crossing functor $\theta_s$ for $s\in S_{\aff}$ on $\mathbf{C}$ and proved that projective objects are preserved by wall-crossing functors.
In this subsection, we prove the following.
We identify $W'_{\aff}\simeq W_{\aff}$ and $S\simeq R$ by using $A_0^+$.
\begin{thm}
The category $\widetilde{\mathcal{K}}_P$ is equivalent to the category of projective objects in $\mathbf{C}$.
The action of $B_s$ on $\widetilde{\mathcal{K}}_P$ corresponds to $\theta_s$ for $s\in S_{\aff}$.
\end{thm}
Let $M\in \widetilde{\mathcal{K}}_P$ and $J\subset \mathcal{A}$ an open subset.
Then $M_J$ is an $R$-bimodule (as we identify $S\simeq R$) and the left action of $f\in R^{W_{\mathrm{f}}}$ is equal to the right action of $f$.
Hence $M_J$ is an $R\otimes_{R^{W_{\mathrm{f}}}}R$-module.
The algebra $R\otimes_{R^{W_{\mathrm{f}}}}R$ is isomorphic to the structure algebra $\mathcal{Z}$ on the moment graph attached to $W_{\mathrm{f}}$.
Hence we get a functor $F$ from $\widetilde{\mathcal{K}}_P$ to the category of $\mathcal{Z}$-coefficient presheaves on $\mathcal{A}$.

We prove that $F$ is fully-faithful.
Since $M = F(M)(\mathcal{A})$ as an $R$-module, $F$ induces an injective map between space of morphisms, namely $F$ is faithful.
Let $f\colon F(M)\to F(N)$ be a morphism between sheaves.
We define $\varphi\colon M\to N$ by $M = F(M)(\mathcal{A})\to F(N)(\mathcal{A}) = N$.
Then this is an $R$-bimodule morphism.
Moreover, $\varphi$ induces $M/M_{\mathcal{A}\setminus J} = F(M)(J) \to F(N)(J) = N/N_{\mathcal{A}\setminus J}$ for any open subset $J$.
Hence $\varphi(M_I)\subset N_I$ for any closed subset $I\subset \mathcal{A}$.
Therefore $\varphi$ is a morphism in $\widetilde{\mathcal{K}}_P$ and therefore $F$ is full.

Next we prove that $F(M*B_s)\simeq \theta_s(F(M))$ for $M\in \widetilde{\mathcal{K}}_P$.
Let $s\in S_{\aff}$ and $\epsilon_s$ the functor defined in \cite[8.1]{arXiv:1504.01699}.
Then an argument of the proof in \cite[Proposition~5.3]{arXiv:1901.02336_accepted} gives $\epsilon_s(M) \simeq M\otimes_{R}B_s$ as $\mathcal{Z}$-modules.
(Here, in the right hand side, we consider a $\mathcal{Z}$-module as an $R$-bimodule via $\mathcal{Z} = R\otimes_{R^{W_{\mathrm{f}}}}R$.)
Let $J\subset \mathcal{A}$ be an open subset and $J^{\flat}$ (resp.\ $J^{\sharp}$) be the largest (resp.\ smallest) $s$-invariant open subset which is contained in (resp.\ contains) $J$.
Then we have morphisms
\[
(M*B_s)_{J^{\sharp}}\xrightarrow{j^{\sharp}} (M*B_s)_{J}\xrightarrow{j^{\flat}} (M*B_s)_{J^{\flat}}
\]
such that $j^{\sharp},j^{\flat}$ are surjective.
We have $(M*B_s)_{J^{\sharp}}\simeq M_{J^{\sharp}}*B_s$ and $(M*B_s)_{J^{\flat}} \simeq M_{J^{\flat}}*B_s$ by Lemma~\ref{lem:(M*B_s)_I for s-invariant I}.
We have $\supp_{\mathcal{A}}(\Ker j_1)\subset J^{\sharp}\setminus J$ and $\supp_{\mathcal{A}}(\Ker j_2)\subset J\setminus J^{\flat}$.
Hence, by \cite[Lemma~2.8]{arXiv:1504.01699}, $(M*B_s)_{J}$ satisfies the condition in \cite[8.3]{arXiv:1504.01699} and we get $F(M*B_s)(J) \simeq \theta_s(F(M))(J)$.
Therefore we get $F(M*B_s)\simeq \theta_s(F(M))$.

Finally we prove that the image of $F$ is projective and the functor from $\widetilde{\mathcal{K}}_P$ to the category of projective objects in $\mathbf{C}$ is essentially surjective.
Let $\underline{\mathcal{K}}_{\lambda}$ be a projective object in $\mathbf{C}$ defined in \cite[Section~6]{arXiv:1504.01699}.
From the definitions, we have $F(Q_{\lambda}) = \underline{\mathcal{K}}_{\lambda}$.
Any $M\in \widetilde{\mathcal{K}}_{P}$ is a direct sum of direct summands of objects of a form $M*B_{s_1}*\cdots *B_{s_l}(n)$ for $s_1,\ldots,s_l\in S_{\aff}$ and $n\in\Z$.
Since $F(M*B_{s_1}*\cdots *B_{s_l}(n)) = \theta_{s_l}\cdots \theta_{s_1}\underline{\mathcal{K}}_{\lambda}$ is projective in $\mathbf{C}$ by \cite[Corollary~8.7]{arXiv:1504.01699}, $F(M)$ is projective in $\mathbf{C}$ for any $M\in \widetilde{\mathcal{K}}_{P}$.
Moreover, by the proof of \cite[Theorem~8.8]{arXiv:1504.01699}, any projective object in $\mathbf{C}$ is a direct sum of direct summands of objects of a form $\theta_{s_l}\cdots \theta_{s_1}\underline{\mathcal{K}}_{\lambda}$.
Since $F$ is fully-faithful, the essential image of $F$ is closed under taking a direct summand.
Hence $F$ is essentially surjective.

\section{The category of Andersen-Jantzen-Soergel}
Throughout this section, we assume that $\Coeff$ is noetherian complete local ring.
\subsection{Our combinatorial category}
In this subsection we introduced some categories using the categories introduced in the previous section.
The categories will be related to the combinatorial categories of Andersen-Jantzen-Soergel.

Let $S_0$ be a flat commutative graded $S$-algebra.
Let $\mathcal{K}'(S_0)$ be the category whose objects are the same as those of $\widetilde{\mathcal{K}}'(S_0)$ and the spaces of morphisms are defined by
\[
\Hom_{\mathcal{K}'(S_0)}(M,N) = \Hom_{\widetilde{\mathcal{K}}'(S_0)}(M,N)/\{\varphi\in\Hom_{\widetilde{\mathcal{K}}'(S_0)}(M,N)\mid \varphi(M_{A}^{\emptyset})\subset \bigoplus_{A' > A}N_{A'}^{\emptyset}\}.
\]
We also define $\mathcal{K}(S_0)$ and $\mathcal{K}_{\Delta}(S_0)$ by the same way.

\begin{lem}
Let $M,N\in \widetilde{\mathcal{K}}'(S_0)$, $\varphi\colon M\to N$ and $B\in \Sbimod$.
If $\varphi(M_{A}^{\emptyset}) \subset \bigoplus_{A' > A}N_{A'}^{\emptyset}$ for any $A\in \mathcal{A}$, then $\varphi\otimes\id\colon M*B\to N*B$ satisfies $(\varphi\otimes \id)((M*B)_{A}^{\emptyset}) \subset \bigoplus_{A' > A}(N*B)_{A'}^{\emptyset}$ for any $A\in \mathcal{A}$.
\end{lem}
\begin{proof}
Recall that we have $(M*B)_A^{\emptyset} = \bigoplus_{x\in W_\aff}M_{Ax^{-1}}^{\emptyset}\otimes B_x^{\emptyset}$.
We have $\varphi(M_{Ax^{-1}}^{\emptyset})\otimes B_x^{\emptyset}\subset \bigoplus_{A'x^{-1}\in Ax^{-1} + \Z\Delta,A'x^{-1} > Ax^{-1}}N_{A'x^{-1}}^{\emptyset}\otimes B_x^{\emptyset}$.
Since $x\colon (Ax^{-1} + \Z\Delta)\to (A + \Z\Delta)$ preserves the order, $A'x^{-1} > Ax^{-1}$ if and only if $A' > A$.
Therefore $(\varphi\otimes\id)(M*B)_A^{\emptyset}\subset \bigoplus_{x\in W_\aff,A' > A}N_{A'x^{-1}}^{\emptyset}\otimes B_x^{\emptyset} = \bigoplus_{A'>A}(N*B)_{A'}^{\emptyset}$.
\end{proof}
Therefore $(M,B)\mapsto M*B$ defines a bi-functor $\mathcal{K}'(S_0)\times \Sbimod\to \mathcal{K}'(S_0)$ and also $\mathcal{K}_{\Delta}(S_0)\times\Sbimod\to \mathcal{K}_{\Delta}(S_0)$.

\begin{prop}\label{prop:adjointness for K}
Let $M,N\in \mathcal{K}'(S_0)$ and $s\in S_{\aff}$.
Then $\Hom_{\mathcal{K}'(S_0)}(M*B_s,N)\simeq \Hom_{\mathcal{K}'(S_0)}(M,N*B_s)$.
\end{prop}
\begin{proof}
Let $\varphi$ and $\psi$ as in the proof of Proposition~\ref{prop:adjointness}.
Then the proof of Proposition~\ref{prop:adjointness} shows that $\varphi(M_{A}^{\emptyset})\subset\bigoplus_{A' > A}(N*B_s)_{A'}^{\emptyset}$ for any $A\in \mathcal{A}$ if and only if $\psi((M*B_s)^{\emptyset}_{A}) \subset\bigoplus_{A' > A}N_{A'}^{\emptyset}$ for any $A\in\mathcal{A}$.
The proposition follows.
\end{proof}

For each morphism $\varphi\colon M\to N$ in $\widetilde{\mathcal{K}}(S_0)$ and $A\in \mathcal{A}$, we have a homomorphism $\varphi_{\{A\}}\colon M_{\{A\}}\to N_{\{A\}}$.
Note that $\varphi(M_A^{\emptyset})\subset \bigoplus_{A' > A}N_{A'}^{\emptyset}$ if and only if $\varphi_{\{A\}} = 0$.
Hence $M\mapsto M_{\{A\}}$ defines a functor from $\mathcal{K}(S_0)$ to the category of graded $S_0$-modules.
Using this, we define as follows: A sequence $M_1\to M_2\to M_3$ in $\mathcal{K}(S_0)$ satisfies \property{ES} if the composition $M_1\to M_2\to M_3$ is zero in $\mathcal{K}(S_0)$ and $0\to (M_1)_{\{A\}}\to (M_2)_{\{A\}}\to (M_3)_{\{A\}}\to 0$ is exact for any $A\in \mathcal{A}$.
Note that a sequence $M_1\to M_2\to M_3$ in $\widetilde{\mathcal{K}}$ may not satisfy \property{ES} even when it satisfies \property{ES} in $\mathcal{K}$ since the composition $M_1\to M_2\to M_3$ may be zero only in $\mathcal{K}$.

For the definition of $\mathcal{K}_{P}(S_0)$, we use the same condition to define $\widetilde{\mathcal{K}}_{P}(S_0)$.
For $M\in \mathcal{K}_{\Delta}(S_0)$, we say $M\in \mathcal{K}_{P}(S_0)$ if for any sequence $M_1\to M_2\to M_3$ in $\mathcal{K}_{\Delta}(S_0)$ which satisfies \property{ES}, the induced homomorphism $0\to \Hom_{\mathcal{K}_{\Delta}(S_0)}^{\bullet}(M,M_1)\to \Hom_{\mathcal{K}_{\Delta}(S_0)}^{\bullet}(M,M_2)\to \Hom_{\mathcal{K}_{\Delta}(S_0)}^{\bullet}(M,M_3)\to 0$ is exact.
Note that this definition is not the same as that in the introduction.
We will prove that two definitions coincide with each other later.

\begin{prop}\label{prop:indecomposable is indecomposable}
An indecomposable object in $\widetilde{\mathcal{K}}'(S_0)$ such that $\supp_{\mathcal{A}}(M)$ is finite is also indecomposable as an object of $\mathcal{K}'(S_0)$.
\end{prop}
\begin{proof}
Let $M\in \widetilde{\mathcal{K}}'(S_0)$ and assume that $\supp_{\mathcal{A}}(M)$ is finite.
Then $\{\varphi\in \End_{\widetilde{\mathcal{K}}'(S_0)}(M)\mid \varphi(M_{A}^{\emptyset})\subset \bigoplus_{A' > A}M_{A'}^{\emptyset}\ (A\in \mathcal{A})\}$ is a two-sided ideal of $\End_{\widetilde{\mathcal{K}}'(S_0)}(M)$ and, since $\supp_{\mathcal{A}}(M)$ is finite, this is nilpotent.
Therefore the idempotent lifting property implies the proposition.
\end{proof}

\begin{lem}\label{lem:ES in K induces exact sequence in a special case}
Let $K\subset \mathcal{A}$ be a locally closed subset such that for any $A\in K$ we have $(A + \Z\Delta) \cap K = \{A\}$.
Then we have the following.
\begin{enumerate}
\item For a morphism $\varphi\colon M\to N$ in $\widetilde{\mathcal{K}}(S_0)$ which is zero in $\mathcal{K}(S_0)$, the homomorphism $M_{K}\to N_{K}$ is zero in $\widetilde{\mathcal{K}}(S_0)$.
\item Let $M_1\to M_2\to M_3$ be a sequence in $\widetilde{\mathcal{K}}(S_0)$ and assume that the sequence $M_1\to M_2\to M_3$ satisfies \property{ES} as a seqeune in $\mathcal{K}(S_0)$.
Then $(M_1)_{K}\to (M_2)_{K}\to (M_3)_{K}$ satisfies \property{ES} as a sequence in $\widetilde{\mathcal{K}}(S_0)$. In particular, $0\to (M_1)_{K}\to (M_2)_{K}\to (M_3)_{K}\to 0$ is an exact sequence of $(S_0,R)$-bimodules.
\end{enumerate}
\end{lem}
\begin{proof}
(1)
We have $M^{\emptyset}_{K} = \bigoplus_{A\in K}M_{A}^{\emptyset}$ and $N_{K}^{\emptyset} = \bigoplus_{A\in K}N_{A}^{\emptyset}$.
Since $\varphi = 0$ in $\mathcal{K}$, we have $\varphi(M_{A}^{\emptyset})\subset \bigoplus_{A' > A}N_{A'}^{\emptyset}$ for any $A\in K$.
We also know that $\varphi(M_{A}^{\emptyset})\subset \bigoplus_{A' \in A + \Z\Delta}N_{A'}^{\emptyset}$.
By the assumption, there is no $A'\in A + \Z\Delta$ such that $A' > A$ and $A'\in K$.
Hence $\varphi(M_{A}^{\emptyset}) = 0$.

(2)
By (1), the composition $(M_1)_{K}\to (M_2)_{K}\to (M_3)_{K}$ is zero.
\end{proof}

\begin{lem}\label{lem:B preserves ES in K}
Assume that a sequence $M_1\to M_2\to M_3$ in $\mathcal{K}_{\Delta}(S_0)$ satisfies \property{ES}.
Then $M_1*B\to M_2*B\to M_3*B$ also satisfies \property{ES}.
\end{lem}
\begin{proof}
We may assume $B = B_s$ where $s\in S_{\aff}$.
We take lifts of $M_1\to M_2$ and $M_2\to M_3$ in $\widetilde{\mathcal{K}}(S_0)$ and we regard $M_1\to M_2\to M_3$ also as a sequence in $\widetilde{\mathcal{K}}(S_0)$.
As in Corollary~\ref{cor:Hecke action}, we have $(M_{i}*B_s)_{\{A\}}\simeq (M_{i})_{\{A,As\}}(\varepsilon(A))$  where $\varepsilon(A)$ is as in the proof of Lemma~\ref{lem:stalk of M*B_s, any closed subset}.
By the previous lemma, $0\to (M_{1})_{\{A,As\}}\to (M_{2})_{\{A,As\}}\to (M_{3})_{\{A,As\}}\to 0$ is exact.
Therefore $0\to (M_1*B_s)_{\{A\}}\to (M_2*B_s)_{\{A\}}\to (M_3*B_s)_{\{A\}}\to 0$ is exact.
Hence a sequence $M_1*B_s\to M_2*B_s\to M_3*B_s$ in $\mathcal{K}_{\Delta}(S_0)$ satisfies \property{ES}.
\end{proof}

Combining Proposition~\ref{prop:adjointness for K}, we have $\mathcal{K}_{P}(S_0)*\Sbimod\subset \mathcal{K}_{P}(S_0)$.

\begin{lem}\label{lem:homomorphism from Q_lambda}
Let $\lambda\in (\R\Delta)_{\integer}$.
The subset $W'_{\lambda}A_{\lambda}^{-}$ is locally closed and we have a natural isomorphism $\Hom^{\bullet}_{\mathcal{K}(S_0)}(S_0\otimes_{S}Q_{\lambda},M)\simeq M_{W'_{\lambda}A_{\lambda}^{-}}$ for $M\in \mathcal{K}_{\Delta}(S_0)$.
\end{lem}
\begin{proof}
Set $I = \{A'\in \mathcal{A}\mid A'\ge A_{\lambda}^{-}\}$.
We prove $I\setminus W'_{\lambda}A_{\lambda^{-}}$ is closed.
Let $A_1\in W'_{\lambda}A_{\lambda}^{-}$ and $A_2\in I$ satisfies $A_2\le A_1$.
We prove $A_2\in W'_{\lambda}A_{\lambda}^{-}$.
This proves that $I\setminus W'_{\lambda}A_{\lambda^{-}}$ is closed.
Take $A_3\in W'_{\lambda}A_{\lambda}^{-}$ such that $A_2\in A_3 + \Z\Delta$.
Then by Lemma~\ref{lem:lemma on W_lambda-orbit}, we have $A_2\ge A_3$.
Take $x\in W'_{\lambda}$ and $\mu\in \Z\Delta$ such that $A_1 = x(A_3)$ and $A_2 = A_3 + \mu$.
Then $A_1\ge A_2\ge A_3$ implies $x(\lambda) - (\lambda + \mu)\in\R_{\ge 0}\Delta^+$ and $(\lambda + \mu) - \lambda\in \R_{\ge 0}\Delta^+$.
As $x(\lambda) = \lambda$, we have $\mu = 0$.
Hence $A_2 = A_3\in W'_{\lambda}A_{\lambda}^{-}$.

We have $\Hom^{\bullet}_{\widetilde{\mathcal{K}}(S_0)}(Q_{\lambda},M)\simeq M_{I}$ where $I = \{A'\in \mathcal{A}\mid A'\ge A_{\lambda}^{-}\}$ and, under this correspondence, $\{\varphi\in \Hom^{\bullet}_{\widetilde{\mathcal{K}}(S_0)}(Q_{\lambda},M)\mid \varphi((Q_{\lambda})_A^{\emptyset})\subset \bigoplus_{A' > A}M_{A'}^{\emptyset}\}$ exactly corresponds to $\{m\in M_{I}\mid \text{$m_A = 0$ for any $A\in W'_{\lambda}A_{\lambda^{-}}$}\}$.
Since $I\setminus W'_{\lambda}A_{\lambda^{-}}$ is closed, $\{m\in M_{I}\mid \text{$m_A = 0$ for any $A\in W'_{\lambda}A_{\lambda^{-}}$}\} = M_{I\setminus W'_{\lambda}A_{\lambda}^{-}}$.
Hence $\Hom^{\bullet}_{\mathcal{K}(S_0)}(Q_{\lambda},M)\simeq M_{W'_{\lambda}A_{\lambda}^{-}}$.
\end{proof}

\begin{prop}\label{prop:object in K_P = object in tilde(K)_P}
The objects of $\mathcal{K}_P$ are the same as those of $\widetilde{\mathcal{K}}_{P}$.
\end{prop}
\begin{proof}
First we prove that any $M\in\widetilde{\mathcal{K}}_P$ belongs to $\mathcal{K}_P$.
By Theorem~\ref{thm:indecomposables in widetilde(K)}, we may assume $M = Q_{\lambda}*B_{s_1}*\dotsm *B_{s_l}(n)$ for some $\lambda\in (\R\Delta)_{\integer}$, $s_1,\dots,s_l\in S_{\aff}$ and $n\in\Z$.
By Proposition~\ref{prop:adjointness for K} and Lemma~\ref{lem:B preserves ES in K}, we may assume $M = Q_{\lambda}$.

We have $\Hom_{\mathcal{K}}(Q_{\lambda},M) \simeq M_{W'_{\lambda}A_{\lambda}^{-}}$.
Since $W'_{\lambda}A_{\lambda}^{-}$ satisfies the condition of Lemma~\ref{lem:ES in K induces exact sequence in a special case}, this implies $Q_{\lambda}\in\mathcal{K}_{P}$.

The object $Q(A)$ is indecomposable by Proposition~\ref{prop:indecomposable is indecomposable}.
Using the argument in the proof of Theorem~\ref{thm:indecomposables in widetilde(K)}, any object in $\mathcal{K}_{P}$ is a direct sum of $Q(A)(n)$.
Hence the proposition is proved.
\end{proof}

Hence our $\mathcal{K}_{P}$ is the same as that in the introduction.

\begin{cor}\label{cor:base change for K_P}
Let $M\in \mathcal{K}_{P}$, $N\in \mathcal{K}_{\Delta}$ and $S_0$ a flat commutative graded $S$-algebra.
\begin{enumerate}
\item The natural map $S_0\otimes_{S}\Hom^{\bullet}_{\mathcal{K}_P}(M,N)\to \Hom^{\bullet}_{\mathcal{K}_P(S_0)}(S_0\otimes_{S}M,S_0\otimes_{S}N)$ is an isomorphism.
\item We have $S_0\otimes_{S}M\in \mathcal{K}_{P}(S_0)$.
\end{enumerate}\end{cor}
\begin{proof}
We may assume $M = Q_{\lambda}*B_{s_1}*\dotsm *B_{s_l}(n)$ for some $\lambda\in (\R\Delta)_{\integer}$, $s_1,\dots,s_l\in S_{\aff}$ and $n\in\Z$.

(1)
By Proposition~\ref{prop:adjointness for K}, we may assume $M = Q_{\lambda}$.
In this case, the corollary is equivalent to $S_0\otimes_S(N_{W'_{\lambda}A_{\lambda}^{-}})\simeq (S_0\otimes_{S}N)_{W'_{\lambda}A_{\lambda}^{-}}$.
This is clear.

(2)
By Lemma~\ref{lem:B preserves ES in K}, we may assume $M = Q_{\lambda}$.
Then $S_0\otimes_{S}Q_{\lambda}\in \mathcal{K}_{P}(S_0)$ by Lemma~\ref{lem:ES in K induces exact sequence in a special case} and \ref{lem:homomorphism from Q_lambda}.
\end{proof}

We can define $\ch\colon [\mathcal{K}_P]\to \mathcal{P}^{0}$ by the same formula as $\ch\colon [\widetilde{\mathcal{K}}_P]\to \mathcal{P}^{0}$.
By the previous proposition with Theorem~\ref{thm:categorification for tilde(K)}, we get the following.
\begin{thm}\label{thm:categorification for K}
We have $[\mathcal{K}_P]\simeq \mathcal{P}^{0}$.
\end{thm}

\subsection{A formula on homomorphisms}
Let $m\mapsto \overline{m}$ be a map from $\mathcal{P}^{0}$ to $\mathcal{P}^{0}$ defined in \cite[Theorem~4.3]{MR1444322}.
For $m\in \mathcal{P}^{0}$ and $m' \in \mathcal{P}$, take $c_A,d_A\in \Z[v,v^{-1}]$ such that $\overline{m} = \sum_{A\in \mathcal{A}}c_AA$ and $m' = \sum_{A\in \mathcal{A}}d_AA$.
Set $(m,m')_{\mathcal{P}} = \sum_{A\in\mathcal{A}}c_Ad_A$.
We define $\omega\colon \mathcal{H}\to \mathcal{H}$ by $\omega(\sum_{x\in W}a_x(v)H_x) = \sum_{x\in W}a_x(v^{-1})H_{x}^{-1}$.
Then we have
\[
(mh,m')_{\mathcal{P}} = (m,m'\omega(h))_{\mathcal{P}}
\]
where $m\in \mathcal{P}^{0}$, $m'\in \mathcal{P}$ and $h\in \mathcal{H}$.
This easily follows from the definitions.
Let $w_0\in W_{\mathrm{f}}$ be the longest element.
\begin{thm}\label{thm:hom formula}
Let $P\in \mathcal{K}_{P}$ and $M\in \mathcal{K}_{\Delta}$.
Then $\Hom^{\bullet}_{\mathcal{K}_{\Delta}}(P,M)$ is graded free left $S$-module and the graded rank is given by
\[
\grk\Hom^{\bullet}_{\mathcal{K}_{\Delta}}(P,M) = v^{-2\ell(w_0)}(\ch(P),\ch(M))_{\mathcal{P}}.
\]
\end{thm}
\begin{proof}
Since $[\mathcal{K}_P]$ is generated by elements of a form $[Q_{\lambda}*B_{s_1}*\dotsb*B_{s_l}]$ with $\lambda\in (\R\Delta)_{\integer}$ and $s_1,\dots,s_l\in S_{\aff}$, we may assume $P$ has this form.
Moreover, by Lemma~\ref{prop:adjointness for K} and the formula before the theorem, we may assume $P = Q_{\lambda}$.
In this case, we have $\Hom^{\bullet}_{\mathcal{K}_{\Delta}}(P,M) \simeq M_{W_{\lambda}'A_{\lambda}^-}$ and this is graded free by the definition of $\mathcal{K}_{\Delta}$.
Moreover, the graded rank of $M_{W_{\lambda}'A_{\lambda}^-}$ is $\sum_{A\in W_{\lambda}'A_{\lambda}^-}\grk(M_{\{A\}})$.

Let $S_{\lambda}$ be the set of reflections in $W'_{\lambda}$ along the walls of $A_{\lambda}^-$.
Then this is a generator of $W'_{\lambda}$ and $(W'_{\lambda},S_{\lambda})$ is a Coxeter system.
The length function of this Coxeter system is denoted by $\ell_{\lambda}$.

We calculate $(\ch(Q_{\lambda}),\ch(M))$.
We put $(\sum_{A\in \mathcal{A}} c_AA,\sum_{A\in \mathcal{A}}d_AA)' = \sum_{A\in \mathcal{A}} c_Ad_A$.
Let $E_{\lambda}\in \mathcal{P}$ be the element defined in \cite[4]{MR1444322} and $A_{\lambda}^+$ the maximal element in $W'_{\lambda}A_{\lambda}^-$.
Then we have $E_{\lambda} = \sum_{w\in W'_{\lambda}}v^{\ell_{\lambda}(w)}wA_{\lambda}^+$.
Since $\ell(w(A_{\lambda}^+)) = \ell(A_{\lambda}^+) - \ell_{\lambda}(w)$, we have $e_{\lambda} = \sum_{w\in W'_{\lambda}}v^{-\ell(w(A_{\lambda}^+))}w(A_{\lambda}^+) = v^{-\ell(A_{\lambda}^+)}E_\lambda$.
Therefore $\ch(Q_{\lambda}) = v^{2\ell(A_{\lambda}^-)}e_{\lambda} = v^{2\ell(A_{\lambda}^-)-\ell(A_{\lambda}^+)}E_{\lambda}$.
Since $\overline{E_{\lambda}} = E_{\lambda}$, we get $\overline{\ch(Q_{\lambda})} = v^{-2\ell(A_{\lambda}^-)+\ell(A_{\lambda}^+)}E_{\lambda} = v^{-2\ell(A_{\lambda}^-)+2\ell(A_{\lambda}^+)}e_{\lambda} = v^{2\ell(w_0)}e_{\lambda}$.
Hence
\begin{align*}
(\ch(Q_\lambda),\ch(M))_{\mathcal{P}} & = v^{2\ell(w_0)}(e_\lambda,\ch(M))'\\
& = v^{2\ell(w_0)}\left(\sum_{A\in W'_{\lambda}A_{\lambda}^-}v^{-\ell(A)}A,\sum_{A\in \mathcal{A}}v^{\ell(A)}\grk(M_{\{A\}})A\right)'\\
& = v^{2\ell(w_0)}\sum_{A\in W'_{\lambda}A_{\lambda}^-}\grk(M_{\{A\}})\\
& = v^{2\ell(w_0)}\grk\Hom^{\bullet}_{\mathcal{K}_P}(Q_{\lambda},M).
\end{align*}
We get the theorem.
\end{proof}

\subsection{The category $\mathcal{K}_{P}^{\alpha}$}
In this subsection, we analyze $\mathcal{K}^{\alpha}_{P} = \mathcal{K}_P(S^{\alpha})$.
First we define an object $Q_{A,\alpha}$ where $A\in \mathcal{A}$ and $\alpha\in\Delta^+$.
Set $Q_{A,\alpha} = \{(a,b)\in S^2\mid a\equiv b\pmod{\alpha^\vee}\}$ and define a right action of $R$ on $Q_{A,\alpha}$ by $(x,y)f = (f_Ax,s_\alpha(f_A)y)$ for $(x,y)\in Q_{A,\alpha}$ and $f\in R$.
We have $Q_{A,\alpha}^\emptyset = S^\emptyset\oplus S^\emptyset$ and we set
\[
(Q_{A,\alpha})_{A'}^{\emptyset}
= 
\begin{cases}
S^\emptyset\oplus 0 & (A' = A),\\
0\oplus S^{\emptyset} & (A' = \alpha\uparrow A),\\
0 & (\text{otherwise}).
\end{cases}
\]
It is easy to see that $Q^{\alpha}_{A,\alpha} = S^{\alpha}\otimes_{S}Q_{A,\alpha}$ is indecomposable.

\begin{lem}
We have $Q^\alpha_{A,\alpha}\in \mathcal{K}_{P}^{\alpha}$.
\end{lem}
\begin{proof}
It is easy to see that $Q_{A,\alpha}^{\alpha}\in \mathcal{K}_{\Delta}^{\alpha}$.
Let $M\in \mathcal{K}^{\alpha}_{\Delta}$ and we analyze $\Hom^{\bullet}_{\mathcal{K}^{\alpha}_{\Delta}}(Q_{A,\alpha}^{\alpha},M)$.
By \property{LE}, $M \simeq \bigoplus_i M_i$ such that $\supp_{\mathcal{A}}(M_i)\subset W'_{\alpha,\aff}A_i$ for some $A_i\in \mathcal{A}$.
We have $\Hom^{\bullet}_{\mathcal{K}_{\Delta}^{\alpha}}(Q^{\alpha}_{A,\alpha},M_i) = 0$ if $A\notin W'_{\alpha,\aff}A_i$.
Therefore it is sufficient to prove the following: if a sequence $M_1\to M_2\to M_3$ in $\mathcal{K}^{\alpha}_{\Delta}$ satisfies \property{ES} and $\supp_{\mathcal{A}}(M_i)\subset W'_{\alpha,\aff}A$, then $0\to \Hom_{\mathcal{K}^{\alpha}_{\Delta}}^{\bullet}(Q_{A,\alpha}^{\alpha},M_1)\to \Hom_{\mathcal{K}^{\alpha}_{\Delta}}^{\bullet}(Q_{A,\alpha}^{\alpha},M_2)\to \Hom_{\mathcal{K}^{\alpha}_{\Delta}}^{\bullet}(Q_{A,\alpha}^{\alpha},M_3)\to 0$ is exact.
We can apply a similar argument of the proof of Proposition~\ref{prop:object in K_P = object in tilde(K)_P}.
\end{proof}

We can apply the argument in the proof of Theorem~\ref{thm:indecomposables in widetilde(K)} and get the following proposition.
\begin{prop}\label{prop:rank 1 projectives}
Any object in $\mathcal{K}_{P}^{\alpha}$ is a direct sum of $Q_{A,\alpha}^{\alpha}(n)$ where $A\in \mathcal{A}$ and $n\in\Z$.
\end{prop}

\subsection{The comibinatorial category of Andersen-Jantzen-Soergel}
We recall the comibinatorial category of Andersen-Jantzen-Soergel \cite{MR1272539}.
We use a version of Fiebig \cite{MR2726602}.
We denote the category by $\mathcal{K}_{\AJS}$.

Let $S_0$ be a flat commutative graded $S$-algebra and we define the category which we denote $\mathcal{K}_{\AJS}(S_0)$.
An object of $\mathcal{K}_{\AJS}(S_0)$ is $\mathcal{M} = ((\mathcal{M}(A))_{A\in \mathcal{A}},(\mathcal{M}(A,\alpha))_{A\in \mathcal{A},\alpha\in\Delta^+})$ where $\mathcal{M}(A)$ is a graded $(S_0)^{\emptyset}$-module and $\mathcal{M}(A,\alpha)\subset \mathcal{M}(A)\oplus \mathcal{M}(\alpha\uparrow A)$ is a graded sub-$(S_0)^{\alpha}$-module.
A morphism $f\colon \mathcal{M}\to \mathcal{N}$ in $\mathcal{K}_{\AJS}(S_0)$ is a collection of degree zero $(S_0)^{\emptyset}$-homomorphisms $f_{A}\colon \mathcal{M}(A)\to \mathcal{N}(A)$ which sends $\mathcal{M}(A,\alpha)$ to $\mathcal{N}(A,\alpha)$ for any $A\in \mathcal{A}$ and $\alpha\in\Delta^+$.
Put $\mathcal{K}_{\AJS} = \mathcal{K}_{\AJS}(S)$ and $\mathcal{K}_{\AJS}^{*} = \mathcal{K}_{\AJS}(S^*)$ for $*\in \{\emptyset\}\cup \Delta$.

For each $s\in S_{\aff}$, the translation functor $\vartheta_s\colon \mathcal{K}_{\AJS}(S_0)\to \mathcal{K}_{\AJS}(S_0)$ is defined as
\[
\vartheta_s(\mathcal{M})(A) = \mathcal{M}(A)\oplus \mathcal{M}(As)
\]
and
\[
\vartheta_s(\mathcal{M})(A,\alpha) = 
\begin{cases}
\mathcal{M}(A,\alpha)\oplus \mathcal{M}(As,\alpha) & (As\notin W'_{\alpha,\aff}A),\\
\{(x,y)\in \mathcal{M}(A,\alpha)^2\mid x - y\in \alpha^\vee \mathcal{M}(A,\alpha)\} & (As = \alpha\uparrow A),\\
\alpha^\vee\mathcal{M}(As,\alpha)\oplus \mathcal{M}(\alpha\uparrow A,\alpha) & (As = \alpha\downarrow A).
\end{cases}
\]

We define $\mathcal{F}(S_0)\colon \mathcal{K}_P(S_0)\to \mathcal{K}_{\AJS}(S_0)$ as follows: 
first we put
\[
(\mathcal{F}(S_0)(M))(A) = M_A^{\emptyset}.
\]
To define $(\mathcal{F}(S_0)(M))(A,\alpha)$, we take $X\in \widetilde{\mathcal{K}}_{P}(S_0^{\alpha})$ and an isomorphism $\varphi\colon X\to M^{\alpha}$ in $\widetilde{\mathcal{K}}_P(S_0)$ such that $X = \bigoplus_{\Omega\in W'_{\alpha,\aff}\backslash \mathcal{A}}(X\cap \bigoplus_{A\in \Omega}X_A^{\emptyset})$.
Such $X$ exists since $M$ satisfies \property{LE}.
Then we have an isomorphism $X_A^{\emptyset}\simeq (X_{\ge A}/X_{>A})^{\emptyset}\simeq ((M^{\alpha})_{\ge A}/(M^{\alpha})_{>A})^{\emptyset}\simeq M^{\emptyset}_{A}$.
In general, for $Y\in \mathcal{K}_P(S_0)$, $y\in Y^{\emptyset}$ and $A\in \mathcal{A}$, write $y_A$ for the $Y_A^{\emptyset}$-component of $y$ along the decomposition $Y^{\emptyset} = \bigoplus_{A\in \mathcal{A}}Y^{\emptyset}_A$.
Then this isomorphism can be written as $x\mapsto \varphi(x)_A$.
Here we use the same letter $\varphi$ for the induced map $X^{\emptyset}\to M^{\emptyset}$.

Now let $(\mathcal{F}(S_0)(M))(A,\alpha)$ be the image of 
\[
X_{\ge A}\to X_{A}^{\emptyset}\oplus X_{\alpha\uparrow A}^{\emptyset}\simeq M_{A}^{\emptyset}\oplus M_{\alpha\uparrow A}^{\emptyset}.
\]
In other words, the image is the set of $(\varphi(x_A)_A,\varphi(x_{\alpha\uparrow A})_{\alpha\uparrow A})$ where $x\in X_{\ge A}$.
We may assume $x\in \bigoplus_{A'\in W'_{\alpha,\aff}A}X_{A'}^{\emptyset}$.
Of course we have to prove that this space does not depend on a choice of $X$.
We use the following lemma.
\begin{lem}
Let $X,Y\in \widetilde{\mathcal{K}}_P(S_0)$, $f\colon X\to Y$ be a morphism, $A\in \mathcal{A}$ and $\alpha\in\Delta^{+}$.
Assume that $x\in X^{\emptyset}_{\ge A}$ satisfies $x_{A'} = 0$ for $A'\notin W'_{\alpha,\aff}A$.
\begin{enumerate}
\item We have $f(x)_A = f(x_A)_A$ and $f(x)_{\alpha\uparrow A} = f(x_{\alpha\uparrow A})_{\alpha\uparrow A}$.
\item Let $g\colon Y\to Z$ be another morphism in $\widetilde{\mathcal{K}}_P(S_0)$.
Then $g(f(x)_{A'})_{A'} = g(f(x))_{A'}$ for $A'\in \{A,\alpha\uparrow A\}$
\end{enumerate}
\end{lem}
\begin{proof}
We prove (1).
Let $A''\in \mathcal{A}$.
Then $f(x)_{A''} = \sum_{A'\in\mathcal{A}}f(x_{A'})_{A''}$.
We have 
\begin{itemize}
\item $x_{A'} = 0$ unless $A'\ge A$ since $x\in X_{\ge A}$.
\item $x_{A'} = 0$ unless $A'\in W'_{\alpha,\aff}A$ from the condition on $x$.
\item $f(x_{A'})_{A''} = 0$ unless $A''\ge A'$ from the definition of morphisms in $\widetilde{\mathcal{K}}_P(S_0)$.
\end{itemize}
Therefore, in the sum $\sum_{A'\in\mathcal{A}}f(x_{A'})_{A''}$, we may assume $A'$ satisfies $A\le A'\le A'',A'\in W'_{\alpha,\aff}A$.
If $A'' = A$, then $A\le A'\le A''$ implies $A' = A$.
Hence $f(x)_A = f(x_A)_A$.
If $A'' = \alpha\uparrow A$, we have $A\le A'\le \alpha\uparrow A$ and $A'\in W'_{\alpha,\aff}A$.
Thus we have $A' = A$ or $\alpha\uparrow A$.
However, by Remark~\ref{rem:decomposition is preserved automatically}, we have $f(x_{A})_{\alpha\uparrow A} = 0$.
Hence $f(x)_{\alpha\uparrow A} = f(x_{\alpha\uparrow A})_{\alpha\uparrow A}$.

We prove (2).
We have $f(x_{A'})\in \bigoplus_{A'' \ge A'}Y_{A''}^{\emptyset}$.
Hence $f(x_{A'}) - f(x_{A'})_{A'}\in \bigoplus_{A'' > A'}Y_{A''}^{\emptyset}$.
Therefore $g(f(x_{A'})) - g(f(x_{A'})_{A'})\in \bigoplus_{A'' > A'}Z_{A''}^{\emptyset}$.
Hence $g(f(x_{A'}))_{A'} = g(f(x_{A'})_{A'})_{A'}$.
By (1), the right hand side is $g(f(x)_{A'})_{A'}$ and the left hand side is $g(f(x_{A'}))_{A'} = (g\circ f)(x_{A'})_{A'} = (g\circ f)(x)_{A'} = g(f(x))_{A'}$.
\end{proof}
Let $\varphi'\colon X'\to M^{\alpha}$ be another isomorphism and set $\psi = (\varphi')^{-1}\circ\varphi$.
For $A'\in \{A,\alpha\uparrow A\}$, we have $\varphi(x_{A'})_{A'} = \varphi(x)_{A'} = \varphi'(\psi(x))_{A'} = \varphi'(\psi(x)_{A'})_{A'}$.
Hence $(\varphi(x_{A})_{A},\varphi(x_{\alpha\uparrow A})_{\alpha\uparrow A}) = (\varphi'(\psi(x)_{A})_{A},\varphi'(\psi(x)_{\alpha\uparrow A})_{\alpha\uparrow A})$.
As $\psi$ is a morphism, $\psi(x)\in X'_{\ge A}$.
Hence the right hand side is in $(\mathcal{F}(S_0)(M))(A,\alpha)$ determined by $X'$.
Therefore the space $(\mathcal{F}(S_0)(M))(A,\alpha)$ determined by $X$ is contained in the space $(\mathcal{F}(S_0)(M))(A,\alpha)$ determined by $X'$.
By swapping $X$ with $X'$, we get the reverse inclusion and therefore the space $(\mathcal{F}(S_0)(M))(A,\alpha)$ does not depend on a choice of $X$.

Let $f\colon M\to N$ be a morphism in $\mathcal{K}_P(S_0)$ and take a lift $\widetilde{f}\in\Hom_{\widetilde{\mathcal{K}}_P(S_0)}(M,N)$ of $f$.
Then we have a homomorphism $(\mathcal{F}(S_0)(f))(A)\colon M_A^{\emptyset}\to N_A^{\emptyset}$ defined by $M_A^{\emptyset}\hookrightarrow \bigoplus_{A'\ge A}M_{A'}^{\emptyset}\xrightarrow{\widetilde{f}}\bigoplus_{A'\ge A}N_{A'}^{\emptyset}\twoheadrightarrow N_{A}^{\emptyset}$.
In other words, we put $(\mathcal{F}(S_0)(f))(A)(m) = \widetilde{f}(m)_{A}$.
It is easy to see that this does not depend on a lift $\widetilde{f}$.

We prove that the collection $((\mathcal{F}(S_0)(f))(A))_{A\in \mathcal{A}}$ preserves $(\mathcal{F}(S_0)(M))(A,\alpha)$.
Take $X\in \widetilde{\mathcal{K}}_{P}(S_0^{\alpha})$ and $\varphi\colon X\xrightarrow{\sim} M^{\alpha}$ as in the definition of $(\mathcal{F}(S_0)(M))(A,\alpha)$.
We also take $\psi\colon Y\xrightarrow{\sim} N^{\alpha}$.
Let $(x_1,x_2)\in (\mathcal{F}(S_0)(M))(A,\alpha)$.
Then there exists $x\in X_{\ge A}$ such that $(x_1,x_2) = (\varphi(x_A)_A,\varphi(x_{\alpha\uparrow A})_{\alpha\uparrow A})$.
We may assume $x\in \bigoplus_{A'\in W'_{\alpha,\aff}A}X_{A'}^{\emptyset}$.
We put $\widetilde{g} = \psi^{-1}\circ \widetilde{f}$.
Then $(\mathcal{F}(S_0)(f))(A)(x_1) = \widetilde{f}(\varphi(x_A)_A)_A = \psi(\widetilde{g}(\varphi(x_A)_A))_A$.
By applying the lemma (2) to $\varphi(x_A)_A$, we have $\psi(\widetilde{g}(\varphi(x_A)_A))_A = \psi(\widetilde{g}(\varphi(x_A)_A)_A)_A$ and the again by the lemma (1) and (2), this is equal to $\psi(\widetilde{g}(\varphi(x))_{A})_A$.
Similarly we have $(\mathcal{F}(S_0)(f))(\alpha\uparrow A)(x_2) = \psi(\widetilde{g}(\varphi(x))_{\alpha\uparrow A})_{\alpha\uparrow A}$.
Since $(\psi(\widetilde{g}(\varphi(x))_{A})_A,\psi(\widetilde{g}(\varphi(x))_{\alpha\uparrow A})_{\alpha\uparrow A})$ is the image of $\widetilde{g}(\varphi(x))\in Y_{\ge A}$ under $Y_{\ge A}\to Y_{A}^{\emptyset}\oplus Y_{\alpha\uparrow A}^{\emptyset}\simeq M_{A}^{\emptyset}\oplus M_{\alpha\uparrow A}^{\emptyset}$, it is in $(\mathcal{F}(S_0)(N))(A,\alpha)$.
Hence we have proved that the collection $((\mathcal{F}(S_0)(f))(A))_{A\in \mathcal{A}}$ defines a morphism $\mathcal{F}(S_0)(M)\to \mathcal{F}(S_0)(N)$.
Hence $\mathcal{F}(S_0)$ is a functor.

Put $\mathcal{F} = \mathcal{F}(S)$ and $\mathcal{F}^{*} = \mathcal{F}(S^*)$ for $*\in \{\emptyset\}\cup \Delta$.

\begin{prop}\label{prop:compativility of translation functors}
We have $\mathcal{F}(M*B_s)\simeq \vartheta_s(\mathcal{F}(M))$.
\end{prop}

\begin{proof}
Before giving a proof, we give some notation.
Let $\alpha\in\Delta$ and $M\in \mathcal{K}(S_0)$.
Assume that $M^{\alpha} = \bigoplus_{\Omega\in W'_{\alpha,\aff}\backslash \mathcal{A}}(M^{\alpha}\cap \bigoplus_{A\in\Omega}M_{A}^{\emptyset})$.
Put $M^{(\Omega)} = M^{\alpha}\cap \bigoplus_{A\in\Omega}M_{A}^{\emptyset}$.
Then $(\mathcal{F}(M))(A,\alpha)$ is the image of $M^{(W'_{\alpha,\aff}A)}$ in $M_{A}^{\emptyset}\oplus M_{\alpha\uparrow A}^{\emptyset}$.
As $\supp M^{(W'_{\alpha,\aff}A)}\subset W'_{\alpha,\aff}A$ and $W'_{\alpha,\aff}\cap [A,\alpha\uparrow A] = \{A,\alpha\uparrow A\}$, we have $(\mathcal{F}(M))(A,\alpha) = M^{(W'_{\alpha,\aff}A)}_{[A,\alpha\uparrow A]}$.

Take $\delta_s\in \Lambda^\vee_{\Coeff}$ such that $\langle\alpha_s,\delta_s\rangle = 1$  and put $b_e = (\alpha_s^\vee)^{-1}(\delta_s\otimes 1 - 1\otimes s(\delta_s))$ and $b_s = (\alpha_s^\vee)^{-1}(\delta_s\otimes 1 - 1\otimes \delta_s)$.
Note that this does not depend on a choice of $\delta_s$.
We fix $(B_s)_e^{\emptyset}\simeq R^{\emptyset}$ and $(B_s)_s^{\emptyset}\simeq R^{\emptyset}$ as
\begin{gather*}
R^{\emptyset}\ni 1\mapsto b_e\in (B_s)_e^{\emptyset},\\
R^{\emptyset}\ni 1\mapsto b_s\in (B_s)_s^{\emptyset}.
\end{gather*}

We have $(M*B_s)_{A}^{\emptyset} = M_{A}^{\emptyset}\otimes (B_s)_e^{\emptyset}\oplus M_{As}^{\emptyset}\otimes (B_s)_s^{\emptyset}\simeq M_{A}^{\emptyset}\oplus M_{As}^{\emptyset} = \vartheta_s(\mathcal{F}(M))(A)$, here we used the above fixed isomorphisms.
We check $\mathcal{F}(M*B_s)(A,\alpha) \simeq \vartheta_s(\mathcal{F}(M))(A,\alpha)$ under this isomorphism.
We may assume $M^{\alpha} = \bigoplus_{\Omega\in W'_{\alpha,\aff}\backslash\mathcal{A}}(M^{\alpha}\cap \bigoplus_{A\in \Omega}M_A^{\emptyset})$.

First we assume that $As\notin W'_{\alpha,\aff}A$.
Then we have $(M*B_s)^{(W'_{\alpha,\aff}A)} = M^{(W'_{\alpha,\aff}A)}\otimes b_e\oplus M^{(W'_{\alpha,\aff}As)}\otimes b_s$ by Lemma~\ref{lem:rank 1 decomposition of M*B_s}.
As $b_e\in (B_s)_e^{\emptyset}$ (resp.\ $b_s\in (B_s)_s^{\emptyset}$) and $[A,\alpha\uparrow A]s\cap W'_{\alpha,\aff}As = [As,\alpha\uparrow As]\cap W'_{\alpha,\aff}As$, we have 
\[
(M*B_s)^{(W'_{\alpha,\aff}A)}_{[A,\alpha\uparrow A]} = M^{(W'_{\alpha,\aff}A)}_{[A,\alpha\uparrow A]}\otimes b_e\oplus M^{(W'_{\alpha,\aff}As)}_{[As,\alpha\uparrow As]}\otimes b_s
\]
Therefore $\mathcal{F}(M*B_s)(A,\alpha) = \mathcal{F}(M)(A,\alpha)\oplus \mathcal{F}(M)(As,\alpha) = \vartheta_s(\mathcal{F}(M))(A,\alpha)$.

Next assume that $As = \alpha\uparrow A$.
Then we have $[A,\alpha\uparrow A] = [A,As] = \{A,As\}$.
Hence $\mathcal{F}(M*B_s)(A,\alpha) = (M*B_s)^{\alpha}_{\{A,As\}}$.
Since $[A,As] = \{A,As\}$ is $s$-invariant, by Lemma~\ref{lem:(M*B_s)_I for s-invariant I}, we have $(M*B_s)^{\alpha}_{[A,As]}\simeq M_{[A,As]}^{\alpha}\otimes_{R} B_s = \mathcal{F}(M)(A,\alpha)\otimes_R B_s$.
Our claim is that the image of $M_{\{A,As\}}^{\alpha}\otimes_{R}B_s$ in $(M_{\{A,As\}}*B_s)^{\emptyset} \simeq (M_{A}^{\emptyset}\oplus M_{As}^{\emptyset})\oplus (M_{As}^{\emptyset}\oplus M_{A}^{\emptyset})$ is equal to $\{(x,y)\in M^{\alpha}_{\{A,As\}}\mid x - y\in \alpha^\vee M^{\alpha}_{\{A,As\}}\}$.
We write the image of $m\in M$ in $M_{A'}^{\emptyset}$ by $m_{A'}$ for $A'\in \mathcal{A}$.
We have $M_{\{A,As\}}^{\alpha}\otimes_{R}B_s = M_{\{A,As\}}^{\alpha}\otimes_{R^s}R$ and the image of $m_1\otimes 1 + m_2\otimes\delta_s\in M_{\{A,As\}}^{\alpha}\otimes_{R^s}R$ in $(M_{A}^{\emptyset}\oplus M_{As}^{\emptyset})\oplus (M_{As}^{\emptyset}\oplus M_{A}^{\emptyset})$ is
\[
((m_{1,A} + \delta_s^Am_{2,A},m_{1,As} + \delta_s^Am_{2,As}),(m_{1,As} + s(\delta_s)^Am_{2,As},m_{1,A} + s(\delta_s)^Am_{1,A})).
\]
Therefore we have
\begin{multline*}
(m_{1,A} + \delta_s^Am_{2,A},m_{1,As} + \delta_s^Am_{2,As}) - (m_{1,A} + s(\delta_s)^Am_{2,A},m_{1,As} + s(\delta_s)^Am_{2,As})
\\=
(\alpha_s^\vee)^A(m_{2,A},m_{2,As})
\end{multline*}
which is in $\alpha^\vee M_{\{A,As\}}^{\alpha}$ since $(\alpha_s^\vee)^A \in \{\pm 1\}\alpha^\vee$.
From this formula it is easy to see the reverse inclusion.

Finally we assume that $As = \alpha \downarrow A$.
Note that $As < A < \alpha\uparrow A < (\alpha\uparrow A)s$.
Put $N = M^{(W'_{\alpha,\aff}A)}$.
We have $\mathcal{F}(N*B_s)(A,\alpha) \subset \mathcal{F}(N*B_{s})(A)\oplus \mathcal{F}(N*B_s)(\alpha\uparrow A) = (N_{As}^{\emptyset}\oplus N_{A}^{\emptyset})\oplus(N_{\alpha\uparrow A}^{\emptyset}\oplus N_{(\alpha\uparrow A)s}^{\emptyset})$.
We describe the image of $(N*B_s)_{[A,\alpha\uparrow A]}$ in $(N_{As}^{\emptyset}\oplus N_{A}^{\emptyset})\oplus(N_{\alpha\uparrow A}^{\emptyset}\oplus N_{(\alpha\uparrow A)s}^{\emptyset})$, or equivalently the image of $(N*B_s)_{I}$ where $I = \{A'\in \mathcal{A}\mid A'\ge As\}\setminus\{As\}$.

Set $I' = \{A'\in \mathcal{A}\mid A'\ge As\}$.
Then $I'\supset I$ and $I'$ is $s$-invariant.
Hence $(N*B_s)_{I'} = N_{I'}\otimes B_s = N_{I'}\otimes_{R^s}R$ by Lemma~\ref{lem:(M*B_s)_I for s-invariant I}.
Consider the projection $(N*B_s)_{I'}\to (N*B_s)_{As}\oplus (N*B_s)_{A}\oplus (N*B_s)_{\alpha\uparrow A} = (N_{As}^{\emptyset}\oplus N_{A}^{\emptyset})\oplus(N_{A}^{\emptyset}\oplus N_{As}^{\emptyset})\oplus (N_{\alpha\uparrow A}^{\emptyset}\oplus N_{(\alpha\uparrow A)s}^{\emptyset})$.
This is given by
\[
\begin{tikzcd}
N_{I'}\otimes_{R^s}R \arrow[r]\arrow[d,phantom,"\ni",sloped] & N_{As}^{\emptyset}\oplus N_{A}^{\emptyset}\oplus N_{A}^{\emptyset}\oplus N_{As}^{\emptyset}\oplus N_{\alpha\uparrow A}^{\emptyset}\oplus N_{(\alpha\uparrow A)s}^{\emptyset}\arrow[d,phantom,"\ni",sloped]\\
m\otimes f \arrow[r,mapsto] &  (m_{As}f,m_{A}s(f),m_{A}f,m_{As}s(f),m_{\alpha\uparrow A}f,m_{(\alpha\uparrow A)s}s(f)).
\end{tikzcd}
\]

Any element in $N_{I'}\otimes_{R^s}R$ is written as $m_1\otimes 1 + m_2\otimes \delta_s$ for $m_1,m_2\in N_{I'}$.
It is in $(N*B_s)_{I}$ if and only the projection to $(N*B_s)_{As}^{\emptyset} \simeq N_{As}^{\emptyset}\oplus N_{A}^{\emptyset}$ is zero.
This projection is given by $(m_{1,As} + s_\alpha(\delta_s^A)m_{2,As},m_{1,A} + s_\alpha(\delta_s^A) m_{2,A})$.
Hence it is sufficient to prove that the image of 
\[
\{m_1\otimes 1 + m_2\otimes \delta_s\in N_{I'}\otimes_{R^s}R\mid \text{$(m_1 + s_\alpha(\delta_s^A)m_{2})_{A'} = 0$ for $A' = A,As$}\}
\]
in $(N*B_s)_{A}^{\emptyset}\oplus (N*B_s)_{\alpha\uparrow A}^{\emptyset} = N_{A}^{\emptyset}\oplus N_{As}^{\emptyset}\oplus N_{\alpha\uparrow A}^{\emptyset}\oplus N_{(\alpha\uparrow A)s}^{\emptyset}$ is $\alpha^\vee N_{[As,A]}\oplus N_{[\alpha\uparrow A,(\alpha\uparrow A)s]}$.
(Note that $A = \alpha\uparrow (As)$ and $(\alpha\uparrow A)s = \alpha\uparrow (\alpha\uparrow A)$.)

The image of $m_1\otimes 1 + m_2\otimes \delta_s$ in $N_{A}^{\emptyset}\oplus N_{As}^{\emptyset}\oplus N_{\alpha\uparrow A}^{\emptyset}\oplus N_{(\alpha\uparrow A)s}^{\emptyset}$ is given by
\[
(m_{1,A} + \delta_s^Am_{2,A},
m_{1,As} + \delta_s^Am_{2,As},
m_{1,\alpha\uparrow A} + s_\alpha(\delta_s^A)m_{2,\alpha\uparrow A},
m_{1,(\alpha\uparrow A)s} + s_\alpha(\delta_s^A)m_{2,(\alpha\uparrow A)s}).
\]
Define $\varepsilon\in\{\pm 1\}$ by $\alpha_s^A = \varepsilon \alpha$.
Since $m_{1,A} + s_{\alpha}(\delta_s^A)m_{2,A} = 0$, we have $m_{1,A} + \delta_s^Am_{2,A} = (\delta_s^A - s_{\alpha}(\delta_s^A))m_{2,A} = \varepsilon \alpha^\vee m_{2,A}$.
By the same argument, we have $m_{1,As} + \delta_s^Am_{2,As} = \varepsilon \alpha^\vee m_{2,As}$.
Therefore $(m_{1,A} + \delta_s^Am_{2,A},m_{1,As} + \delta_s^Am_{2,As}) = \alpha^\vee (\varepsilon m_{2,A},\varepsilon m_{2,As})\in \alpha^\vee N_{[A,As]}^{\emptyset}$.
Therefore the image is in $\alpha^\vee N_{[As,A]}\oplus N_{[\alpha\uparrow A,(\alpha\uparrow A)s]}$.

On the other hand, let $m'_1\in N_{[As,A]}$ and $m'_2\in N_{[\alpha\uparrow A,(\alpha\uparrow A)s]}$.
Take a lift $m_1\in N_{I'}$ (resp.\ $m_2\in M_{I''}$) of $m'_1$ (resp.\ $m'_2$) where $I'' = \{A'\in A\mid A'\ge \alpha\uparrow A\}$.
Put $n = m_2\otimes 1 + \varepsilon(m_1\otimes \delta_s - (s(\delta_s))^Am_1\otimes 1)$.
Then since $m_2\in M_{I''}$, $m_{2,A} = 0$, $m_{2,As} = 0$.
Now it is straightforward to see $n\in (M*B_s)_{I}$ and the image of $n$ is $(\alpha^\vee m'_{1,A},\alpha^\vee m'_{1,As},m'_{2,\alpha\uparrow A},m'_{2,(\alpha\uparrow A)s})$.
We get the proposition.
\end{proof}

\subsection{Some calculations of homomorphisms}
In this subsection we fix a flat commutative graded $S$-algebra $S_0$.
We define some morphisms as follows.
These will be used only in this subsection.
Let $A\in \mathcal{A}$ and $\alpha\in\Delta^+$.
\begin{gather*}
i_0\colon Q_{A,\alpha}\to Q_{A,\alpha}\quad (f,g)\mapsto (0,\alpha^\vee g),\\
i_0^+\colon Q_{A,\alpha}\to Q_{\alpha\uparrow A,\alpha}\quad (f,g)\mapsto (g,f),\\
i_0^-\colon Q_{A,\alpha}\to Q_{\alpha\downarrow A,\alpha}\quad (f,g)\mapsto (0,\alpha^\vee f).
\end{gather*}
It is straightforward to see that these are morphisms in $\widetilde{\mathcal{K}}$.
We also denote the images of these morphisms in $\mathcal{K}$ by the same letters.

\begin{lem}\label{lem:End(Q)}
We have $\End^{\bullet}_{\mathcal{K}(S_0)}(S_0\otimes_{S}Q_{A,\alpha}) = \End^{\bullet}_{\widetilde{\mathcal{K}}(S_0)}(S_0\otimes_{S}Q_{A,\alpha}) = S_0\id\oplus S_0i_0$.
\end{lem}
\begin{proof}
Put $M = S_0\otimes_{S}Q_{A,\alpha}$.
Note that $\supp_{\mathcal{A}}(M) = \{A,\alpha\uparrow A\}$.
Let $\varphi\in \End_{\widetilde{\mathcal{K}}(S_0)}(S_0\otimes_{S}Q_{A,\alpha})$.
We have $\varphi(M_{A}^{\emptyset})\subset \bigoplus_{A'\in A + \Z\Delta}M_{A'}^{\emptyset} = M_{A}^{\emptyset}$.
By the same argument, we also have $\varphi(M_{\alpha\uparrow A}^{\emptyset})\subset M_{\alpha\uparrow A}^{\emptyset}$.
Therefore $\varphi$ preserves $M_{A'}^{\emptyset}$ for any $A'\in \mathcal{A}$.
Hence we get the first equality of the lemma.

We prove $\varphi\in S_0\id + S_0i_0$.
Since $\varphi$ preserves $M_{A'}^{\emptyset}$, we have $\varphi(f,g) = (\varphi_1(f),\varphi_2(g))$ for some $\varphi_1,\varphi_2\colon S_0^{\emptyset}\to S_0^{\emptyset}$.
Restricting to $\{(f,g)\in M\mid g = 0\} = \alpha^\vee S_0\oplus 0$, $\varphi_1$ sends $\alpha^\vee S_0$ to $\alpha^\vee S_0$.
Therefore it is given by $\varphi_1(f) = cf$ for some $c\in S_{0}$.
Replacing $\varphi$ with $\varphi - c\id$, we may assume $\varphi_{1} = 0$.
The image of $\varphi$ is contained in $\{(f,g)\in M\mid f = 0\} = 0\oplus \alpha^\vee S_0$.
Hence $\varphi_2(g) = \alpha^\vee dg$ for some $d\in S_0$ and we have $\varphi = di_0$.
\end{proof}

\begin{lem}
We have $\Hom^{\bullet}_{\mathcal{K}(S_0)}(S_0\otimes_{S}Q_{A,\alpha},S_0\otimes_{S}Q_{\alpha\uparrow A,\alpha}) = S_0i_0^{+}$.
\end{lem}
\begin{proof}
Let $\varphi\colon S_0\otimes_{S}Q_{A,\alpha}\to S_0\otimes_{S}Q_{\alpha \uparrow A,\alpha}$ be a morphism in $\widetilde{\mathcal{K}}(S_0)$.
By a similar argument of the proof of Lemma~\ref{lem:End(Q)}, $\varphi$ is given by $\varphi(f,g) = (\varphi_1(g),\varphi_2(f))$ for $\varphi_i\colon S_0^{\emptyset}\to S_0^{\emptyset}$ such that $\varphi_i(\alpha^\vee S_0)\subset \alpha^\vee S_0$ for $i = 1,2$.
Hence $\varphi_1(f) = cf$ for some $c\in S_0$.
It is clear that $\varphi - ci_0^{+}$ is zero as a morphism in $\mathcal{K}(S_0)$.
Hence we get the lemma.
\end{proof}

\begin{lem}
We have $\Hom^{\bullet}_{\mathcal{K}(S_0)}(S_0\otimes_{S}Q_{A,\alpha},S_0\otimes_{S}Q_{\alpha\downarrow A,\alpha}) = S_0i_0^{-}$.
\end{lem}
\begin{proof}
Set $M = S_0\otimes_{S}Q_{A,\alpha}$ and $N = S_0\otimes_{S}Q_{\alpha \downarrow A,\alpha}$ and let $\varphi\colon M\to N$ be a morphism in $\widetilde{\mathcal{K}}(S_0)$.
We have $\varphi(M_{\alpha\uparrow A}^{\emptyset})\subset \bigoplus_{A' \ge \alpha \uparrow A}N_{A'}^{\emptyset} = 0$ and $\varphi(M_{A}^{\emptyset}) \subset \bigoplus_{A' \in A + \Z\Delta}N_{A'}^{\emptyset} = N_{A}^{\emptyset}$.
Hence $\varphi(f,g) = (0,\varphi_1(f))$ for some $\varphi_1\colon S_{0}^{\emptyset}\to S_{0}^{\emptyset}$.
For any $f\in S_0$ we have $\varphi(f,f) = (0,\varphi_1(f))\in N$.
Hence $\varphi_1(f)\in \alpha^\vee S_0$.
Therefore $\varphi_1(f) = c\alpha^\vee f$ for some $c\in S_0$.
Hence $\varphi = ci_0^{-}$.
\end{proof}

\begin{lem}
If $A_1\ne \alpha\downarrow A_2,A_2,\alpha\uparrow A_2$, then $\Hom_{\mathcal{K}(S_0)}(Q_{A_1,\alpha},Q_{A_2,\alpha}) = 0$.
\end{lem}
\begin{proof}
It follows from $\supp_{\mathcal{A}}(Q_{A_1,\alpha})\cap \supp_{\mathcal{A}}(Q_{A_2,\alpha}) = \emptyset$.
\end{proof}

Next we calculate homomorphisms in $\mathcal{K}_{\AJS}$.
Set $\mathcal{Q}_{A,\alpha} = \mathcal{F}(Q_{A,\alpha})$.
\begin{lem}
The object $\mathcal{Q}_{A,\alpha}$ is given by
\begin{gather*}
\mathcal{Q}_{A,\alpha}(A') = 
\begin{cases}
S^{\emptyset} & (A' = A,\alpha\uparrow A),\\
0 & (\text{otherwise}),
\end{cases}\\
\mathcal{Q}_{A,\alpha}(A',\beta) = 
\begin{cases}
S^{\beta}\oplus 0 & (A' = A,\alpha\uparrow A,\ \beta\ne\alpha),\\
0\oplus S^{\beta} & (\beta\uparrow A' = A,\alpha\uparrow A,\ \beta\ne\alpha),\\
\alpha^\vee S^{\alpha}\oplus 0 & (A' = \alpha\uparrow A,\ \beta = \alpha),\\
\{(f,g)\in (S^{\alpha})^2\mid f\equiv g\pmod{\alpha^\vee}\} & (A' = A,\ \beta = \alpha),\\
0\oplus S^{\alpha} & (A' = \alpha\downarrow A,\ \beta = \alpha),\\
0 & (\text{otherwise}).
\end{cases}
\end{gather*}
\end{lem}
\begin{proof}
The formula of $\mathcal{Q}_{A,\alpha}(A)$ is obvious.
If $\beta\ne \alpha$, then $S^{\beta}\otimes_{S}Q_{A,\alpha} = S^{\beta}\oplus S^{\beta}$.
Hence the formula of $\mathcal{Q}_{A,\alpha}(A',\beta)$ with $\beta\ne\alpha$ follows.
The other formula follow from a direct calculation.
\end{proof}

Set $\iota_0 = \mathcal{F}(i_0)$, $\iota_0^{+} = \mathcal{F}(i_0^{+})$, $\iota_0^{-} = \mathcal{F}(i_0^{-})$.
These morphisms are described as follows.
\begin{gather*}
\iota_0\colon \mathcal{Q}_{A,\alpha}\to \mathcal{Q}_{A,\alpha}\quad (\iota_0)_{A} = 0,(\iota_0)_{\alpha\uparrow A} = \alpha \id,\\
\iota_0^{+}\colon \mathcal{Q}_{A,\alpha}\to \mathcal{Q}_{\alpha\uparrow A,\alpha} \quad (\iota_0^{+})_{A} = 0,(\iota_0^{+})_{\alpha\uparrow A} = \id,\\
\iota_0^{-}\colon \mathcal{Q}_{A,\alpha}\to \mathcal{Q}_{\alpha\downarrow A,\alpha} \quad (\iota_0^{-})_{A} = \alpha \id,(\iota_0^{-})_{\alpha\uparrow A} = 0.
\end{gather*}

\begin{lem}
We have $\End^{\bullet}_{\mathcal{K}_{\AJS}(S_0)}(S_0\otimes_{S}\mathcal{Q}_{A,\alpha}) = S_0\id\oplus S_0\iota_0$.
\end{lem}
\begin{proof}
Set $\mathcal{M} = S_0\otimes_{S}\mathcal{Q}_{A,\alpha}$ and let $\varphi\colon \mathcal{M}\to \mathcal{M}$ be a morphism.
Since $\mathcal{M}(A') = 0$ for $A'\ne A,\alpha\uparrow A$, we have $\varphi_{A'} = 0$ for such $A'$.
The morphism $\varphi$ preserves $\mathcal{M}(\beta\downarrow A,\beta) = 0\oplus S_0^{\beta}$ for any $\beta\in\Delta^+$.
Hence $\varphi_A(S_0^{\beta})\subset S_0^{\beta}$.
Therefore $\varphi_A(S_0)\subset S_0$ and hence $\varphi_A = c\id$ for some $c\in S_0$.
We also have $\varphi_{\alpha\uparrow A} = d\id$ for some $d\in S_0$.

We prove $\varphi\in S_0\id + S_0\iota_{0}$.
By replacing $\varphi$ with $\varphi - c\id$, we may assume $\varphi_{A} = 0$.
We have $(\varphi_A(f),\varphi_{\alpha\uparrow A}(g))\in \mathcal{M}(A,\alpha)$ for any $(f,g)\in \mathcal{M}(A,\alpha)$.
Since $\varphi_A(f) = 0$, we have $\varphi_{\alpha\uparrow A}(g)\in \alpha^\vee S_0^{\alpha}$.
Therefore $d\in \alpha^\vee S_0^{\alpha}\cap S_0 = \alpha^\vee S_0$.
We have $\varphi = (d/\alpha^\vee)\iota_0$.
\end{proof}

\begin{lem}
We have $\Hom^{\bullet}_{\mathcal{K}_{\AJS}(S_0)}(S_0\otimes_{S}\mathcal{Q}_{A,\alpha},S_0\otimes_{S}\mathcal{Q}_{\alpha\uparrow A,\alpha}) = S_0\iota_0^{+}$.
\end{lem}
\begin{proof}
Set $\mathcal{M} = S_0\otimes_{S}\mathcal{Q}_{A,\alpha}$ and $\mathcal{N} = S_0\otimes_{S}\mathcal{Q}_{\alpha\uparrow A,\alpha}$.
Let $\varphi\colon \mathcal{M}\to \mathcal{N}$ be a morphism.
Then $\varphi_{A'} = 0$ for $A'\ne \alpha\uparrow A$.
For $\beta\in\Delta^+\setminus\{\alpha\}$, since $\varphi$ sends $\mathcal{M}(\alpha\uparrow A,\beta) = S_0^{\beta}\oplus 0$ to $\mathcal{N}(\alpha\uparrow A,\beta) = S_0^{\beta}\oplus 0$, we have $\varphi_{\alpha\uparrow A}(S_0^{\beta})\subset S_0^{\beta}$.
Since $\varphi$ sends $\mathcal{M}(A,\alpha)$ to $\mathcal{N}(A,\alpha) = 0\oplus S^{\alpha}$, $\varphi_{\alpha\uparrow A}(S^{\alpha})\subset S^{\alpha}$.
Hence $\varphi_{\alpha\uparrow A} \in S_0\id$ and we get the lemma.
\end{proof}

\begin{lem}
We have $\Hom^{\bullet}_{\mathcal{K}_{\AJS}(S_0)}(S_0\otimes_{S}\mathcal{Q}_{A,\alpha},S_0\otimes_{S}\mathcal{Q}_{\alpha\downarrow A,\alpha}) = S_0i_0^{-}$.
\end{lem}
\begin{proof}
Set $\mathcal{M} = S_0\otimes_{S}\mathcal{Q}_{A,\alpha}$ and $\mathcal{N} = S_0\otimes_{S}\mathcal{Q}_{\alpha\downarrow A,\alpha}$.
Let $\varphi\colon \mathcal{M}\to \mathcal{N}$ be a morphism.
Then $\varphi_{A'} = 0$ for $A'\ne A$.
For $\beta\in \Delta^+\setminus\{\alpha\}$, $\varphi$ sends $\mathcal{M}(A,\beta) = 0\oplus S_0^\beta$ to $\mathcal{N}(A,\beta) = S_0^{\beta}\oplus 0$.
Hence $\varphi_A(S_0^{\beta})\subset S_0^{\beta}$.
The morphism $\varphi$ sends $\mathcal{M}(A,\alpha)$ to $\mathcal{N}(A,\alpha) = \alpha^\vee S^{\alpha}\oplus 0$.
Hence $\varphi_{A}(S_0^{\alpha})\subset \alpha^\vee S_0^{\alpha}$.
Therefore $\varphi_A \in \alpha^\vee S_0 \id$ and we get the lemma.
\end{proof}

\begin{lem}
If $A_1\ne \alpha\downarrow A_2,A_2,\alpha\uparrow A_2$, then $\Hom_{\mathcal{K}_{\AJS}(S_0)}(\mathcal{Q}_{A_1,\alpha},\mathcal{Q}_{A_2,\alpha}) = 0$.
\end{lem}
\begin{proof}
It follows from there is no $A$ such that $\mathcal{Q}_{A_1,\alpha}(A) \ne 0$ and $\mathcal{Q}_{A_2,\alpha}(A)\ne 0$.
\end{proof}

Summarizing the calculations in this subsection, we get the following.

\begin{lem}\label{lem:fully-faithful subgeneric case}
The functor $\mathcal{F}^{\alpha} = \mathcal{F}(S^{\alpha})$ induces an isomorphism $\Hom^{\bullet}_{\mathcal{K}(S_0)}(S_0\otimes_{S}Q_{A_1,\alpha},S_0\otimes_{S}Q_{A_2,\alpha})\xrightarrow{\sim}\Hom^{\bullet}_{\mathcal{K}_{\AJS}(S_0)}(S_0\otimes_{S}\mathcal{F}^{\alpha}(Q_{A_1,\alpha}),S_0\otimes_{S}\mathcal{F}^{\alpha}(Q_{A_2,\alpha}))$.
\end{lem}

\subsection{Equivalence}
\begin{lem}\label{lem:fully-faithful, rank 1}
The functor $\mathcal{F}^{\alpha}\colon \mathcal{K}^{\alpha}_P\to \mathcal{K}^{\alpha}_{\AJS}$ is fully-faithful for $\alpha\in\Delta$.
\end{lem}
\begin{proof}
By Corollary~\ref{cor:base change for K_P} and Proposition~\ref{prop:rank 1 projectives}, we may assume $M = Q^{\alpha}_{A_1,\alpha}$ and $N = Q^{\alpha}_{A_2,\alpha}$ where $A_1,A_2\in \mathcal{A}$.
Hence the lemma follows from Lemma~\ref{lem:fully-faithful subgeneric case}.
\end{proof}

\begin{prop}\label{prop:fully-faithfulness}
The functor $\mathcal{F}\colon \mathcal{K}_P\to \mathcal{K}_{\AJS}$ is fully-faithful.
\end{prop}
\begin{proof}
Let $M,N\in \mathcal{K}_{P}$ and we prove that $\mathcal{F}\colon\Hom^{\bullet}_{\mathcal{K}_P}(M,N)\to \Hom^{\bullet}_{\mathcal{K}_{\AJS}}(\mathcal{F}(M),\mathcal{F}(N))$ is an isomorphism.
By the diagram
\[
\begin{tikzcd}
\Hom^{\bullet}_{\mathcal{K}_P}(M,N) \arrow[rr,"\mathcal{F}"]\arrow[hookrightarrow,rd] && \Hom^{\bullet}_{\mathcal{K}_{\AJS}}(M,N)\arrow[hook',ld]\\
& \prod_{A\in \mathcal{A}}\Hom^{\bullet}_{S^{\emptyset}}(M_A^{\emptyset},N_A^{\emptyset}), &
\end{tikzcd}
\]
$\mathcal{F}$ is injective.
(The injectivity of two morphisms in the above diagram follows from the definitions.)

We prove that $\mathcal{F}$ is surjective.
For $\nu\in X_{\Coeff}$ and let $S_{(\nu)}$ be the localization at the prime ideal $(\nu)\subset S$.
Since $\Hom^{\bullet}_{\mathcal{K}_P}(M,N)$ is graded free, we have $\Ima(\mathcal{F}) = \bigcap_{\nu\in X_{\Coeff}}S_{(\nu)}\otimes_{S}\Ima(\mathcal{F})$.
By Corollary~\ref{cor:base change for K_P}, we have $S_{(\nu)}\otimes_{S}\Ima(\mathcal{F}) = \Ima(\mathcal{F}(S_{(\nu)}))$.
Since any $S_{(\nu)}$ is an $S^{\alpha}$-algebra for some $\alpha\in\Delta$, by Proposition~\ref{prop:fully-faithfulness}, we have $\Ima(\mathcal{F}(S_{(\nu)})) = \Hom^{\bullet}_{\mathcal{K}_{\AJS}(S_{(\nu)})}(\mathcal{F}(S_{(\nu)})(S_{(\nu)}\otimes_{S}M),\mathcal{F}(S_{(\nu)})(S_{(\nu)}\otimes_{S}N))$.
Therefore $\mathcal{F}$ is surjective since $\bigcap_{\nu\in X_{\Coeff}}\Hom^{\bullet}_{\mathcal{K}_{\AJS}(S_{(\nu)})}(\mathcal{F}(S_{(\nu)})(S_{(\nu)}\otimes_{S}M),\mathcal{F}(S_{(\nu)})(S_{(\nu)}\otimes_{S}N))\supset \Hom^{\bullet}_{\mathcal{K}_{\AJS}}(\mathcal{F}(M),\mathcal{F}(N))$.
\end{proof}

Set $\mathcal{Q}_{\lambda} = \mathcal{F}(Q_{\lambda})$.
Let $\mathcal{K}_{\AJS,P}$ be the full-subcategory of $\mathcal{K}_{\AJS}$ consisting of direct summands of direct sums of objects of a form $(\vartheta_{s_1}\circ\dotsm \circ\vartheta_{s_l})(\mathcal{Q}_{\lambda})(n)$ for $s_1,\dots,s_l\in S_{\aff}$, $\lambda\in (\R\Delta)_{\integer}$ and $n\in\Z$.
By Proposition~\ref{prop:compativility of translation functors} and \ref{prop:fully-faithfulness}, we get the following theorem.
\begin{thm}
We have $\mathcal{K}_{P}\simeq \mathcal{K}_{\AJS,P}$.
In particular, the category $\Sbimod$ acts on $\mathcal{K}_{\AJS,P}$.
\end{thm}

\subsection{Representation Theory}\label{subsec:Representation Theory}
In the rest of this paper, we assume that $\Coeff$ is an algebraically closed field of characteristic $p > h$ where $h$ is the Coxeter number.
Let $G$ be a connected reductive group over $\Coeff$ and $T$ a maximal torus of $G$ with the root datum $(X,\Delta,X^\vee,\Delta^\vee)$. 
The Lie algebra $\mathfrak{g}$ of $G$ has a structure of a $p$-Lie algebra.
Let $U^{[p]}(\mathfrak{g})$ be the restricted enveloping algebra.
Let $\widehat{S}$ be the completion of $S$ at the augmentation ideal. 
For $S_0 = \widehat{S}$ or $\Coeff$, let $\mathcal{C}_{S_0}$ be the category defined in \cite{MR1272539}.
The category $\mathcal{C}_{\Coeff}$ is equivalent to the category of $G_1T$-modules where $G_1$ is the kernel of the Frobenius morphism.
Let $Z_{S_0}(\lambda)\in \mathcal{C}_{S_0}$ be the baby Verma module with the highest weight $\lambda$ and $P_{S_0}(\lambda)\in \mathcal{C}_{S_0}$ the indecomposable projective module such that $\Coeff\otimes_{S_0}P_{S_0}(\lambda)$ is the projective cover of the irreducible module with the highest weight $\lambda$.
Such objects exist by \cite[4.19 Theorem]{MR1272539} when $S_0 = \widehat{S}$.

We fix an alcove $A_0\in \mathcal{A}$ and $\lambda_0\in X\cap (pA_0 - \rho)$ where $\rho$ is the half sum of positive roots and $pA_0 = \{pa\mid a\in A_0\}$.
For $S_0 = \widehat{S}$ or $\Coeff$, let $\mathcal{C}_{S_0,0}$ be the full subcategory of $\mathcal{C}_{S_0}$ consisting of quotients of modules of a form $\bigoplus_{w\in W'_{\aff}}P_{S_0}(w\cdot_{p}\lambda_{0})^{n_{w}}$ where $w\cdot_p \lambda_0 = pw((\lambda_0 + \rho)/p) - \rho$ and $n_w\in\Z_{\ge 0}$.
Then the cateogory $\mathcal{C}_{S_0,0}$ is a direct summand of $\mathcal{C}_{S_0}$.
Let $\Proj(\mathcal{C}_{S_0,0}) = \{P\in \mathcal{C}_{S_0,0}\mid \text{$P$ is projective}\}$.

Let $S_0$ be a commutative $S$-algebra which is not necessary graded.
We consider the following object: $\mathcal{M} = ((\mathcal{M}(A))_{A\in \mathcal{A}},(\mathcal{M}(A,\alpha))_{A\in \mathcal{A},\alpha\in\Delta^+})$ where $\mathcal{M}(A)$ is an $(S_0)^{\emptyset}$-module and $\mathcal{M}(A,\alpha)\subset \mathcal{M}(A)\oplus \mathcal{M}(\alpha\uparrow A)$ is a sub-$(S_0)^{\alpha}$-module.
(We consider usual modules, not graded ones.)
We denote the category of such objects by $\mathcal{K}_{\AJS}^{\mathrm{f}}(S_0)$.
Starting from this, we can define the functor $\vartheta_s$ and the category $\mathcal{K}_{\AJS,P}^{\mathrm{f}}(S_0)$ in a similar way.
Andersen-Jantzen-Soergel \cite{MR1272539} proved the following.
We modified the functor using \cite[Theorem~6.1]{MR2726602}.
\begin{thm}
There is an equivalence of the categories $\mathcal{V}\colon \Proj(\mathcal{C}_{\widehat{S},0})\xrightarrow{\sim} \mathcal{K}_{\AJS,P}^{\mathrm{f}}(\widehat{S})$.
\end{thm}
Note that the functor $\mathcal{V}$ is defined explicitly.

Let $\Coeff\otimes_{\widehat{S}}\Proj(\mathcal{C}_{\widehat{S},0})$ be the category defined as follows.
The objects of $\Coeff\otimes_{\widehat{S}}\Proj(\mathcal{C}_{\widehat{S},0})$ are the same as those of $\Proj(\mathcal{C}_{\widehat{S},0})$ and the space of homomorphism is defined by
\[
\Hom_{\Coeff\otimes_{\widehat{S}}\Proj(\mathcal{C}_{\widehat{S},0})}(M,N) = \Coeff\otimes_{\widehat{S}}\Hom_{\Proj(\mathcal{C}_{\widehat{S},0})}(M,N).
\]
\begin{lem}
We have $\Coeff\otimes_{\widehat{S}}\Proj(\mathcal{C}_{\widehat{S},0})\simeq \Proj(\mathcal{C}_{\Coeff,0})$.
\end{lem}
\begin{proof}
We consider the functor $\Coeff\otimes_{\widehat{S}}\Proj(\mathcal{C}_{\widehat{S},0})\to \Proj(\mathcal{C}_{\Coeff,0})$ defined by $P\mapsto \Coeff\otimes_{\widehat{S}}P$.
This is essentially surjective by \cite[4.19 Theorem]{MR1272539} and fully-faithful by \cite[3.3 Proposition]{MR1272539}.
\end{proof}

We also define $\Coeff\otimes_{\widehat{S}}\mathcal{K}_{\AJS,P}^{\mathrm{f}}(\widehat{S})$ and $\Coeff\otimes_{S}\mathcal{K}_{\AJS,P}^{\mathrm{f}}(S)$ by the same way.
\begin{lem}
We have the following.
\begin{enumerate}
\item The category $\mathcal{K}_{\AJS,P}^{\mathrm{f}}(S)$ is equivalent to the category defined as follows: the objects are the same as $\mathcal{K}_{\AJS,P}$ and the space of homomorphisms is defined by $\Hom_{\mathcal{K}_{\AJS,P}^{\mathrm{f}}} = \Hom_{\mathcal{K}_{\AJS,P}}^{\bullet}$.
\item We have $\Coeff\otimes_{\widehat{S}}\mathcal{K}_{\AJS,P}^{\mathrm{f}}(\widehat{S})\simeq \Coeff\otimes_{S}\mathcal{K}_{\AJS,P}^{\mathrm{f}}(S)$.
\end{enumerate}
\end{lem}
\begin{proof}
(1) is obvious.

For (2), define $\widehat{S}\otimes_{S}\mathcal{K}_{\AJS,P}^{\mathrm{f}}$ by the obvious way.
It is sufficient to prove $\mathcal{K}_{\AJS,P}^{\mathrm{f}}(\widehat{S})\simeq\widehat{S}\otimes_{S}\mathcal{K}_{\AJS,P}^{\mathrm{f}}$.
The functor $F\colon \widehat{S}\otimes_{S}\mathcal{K}_{\AJS,P}^{\mathrm{f}}\to \mathcal{K}^{\mathrm{f}}_{\AJS,P}(\widehat{S})$ is defined in a obvious way and it is fully-faithful by \cite[14.8 Lemma]{MR1272539}.
In particular, $F$ sends an indecomposable object to an indecomposable object.
We define the category $\mathcal{K}_{P}^{\mathrm{f}}$ as in (1), namely the objects of $\mathcal{K}_{P}^{\mathrm{f}}$ are the same as those of $\mathcal{K}_{P}^{\mathrm{f}}$ and we define $\Hom_{\mathcal{K}_{P}^{\mathrm{f}}} = \Hom_{\mathcal{K}_{P}}^{\bullet}$.
The indecomposable objects in $\mathcal{K}_{\AJS,P}^{\mathrm{f}}\simeq \mathcal{K}_{P}^{\mathrm{f}}$ and $\mathcal{K}_{\AJS,P}^{\mathrm{f}}(\widehat{S})\simeq \Proj(\mathcal{C}_{\widehat{S},0})$ are both parametrized by $\mathcal{A}$ and it is easy to see that $F$ gives a bijection between the set of indecomposable objects.
Therefore $F$ is essentially surjective.
\end{proof}

Therefore we get
\[
\Proj(\mathcal{C}_{\Coeff,0})\simeq \Coeff\otimes_{\widehat{S}}\Proj(\mathcal{C}_{\widehat{S},0})\simeq \Coeff\otimes_{\widehat{S}}\mathcal{K}_{\AJS,P}^{\mathrm{f}}(\widehat{S})\simeq \Coeff\otimes_{S}\mathcal{K}_{\AJS,P}^{\mathrm{f}}\simeq \Coeff\otimes_{S}\mathcal{K}_{P}^{\mathrm{f}}.
\]
Since the action of $\Sbimod$ on $\mathcal{K}_{P}$ is $S$-linear, it gives an action on $\Coeff\otimes_{S}\mathcal{K}_{P}^{\mathrm{f}}$.
Hence $\Sbimod$ acts on $\Proj(\mathcal{C}_{\Coeff,0})$.
On this action, $B_s$ acts as the wall-crossing functor.
We denote this action by $(M,B)\mapsto M*B$.

Now we prove the following theorem.
\begin{thm}\label{thm:Hecke action on Representations}
There is an action of $\Sbimod$ on $\mathcal{C}_{\Coeff,0}$ such that $B_s$ acts as the wall-crossing functor for $s\in S_\aff$.
\end{thm}

The category $\mathcal{C}_{\Coeff,0}$ has the structure of $\Z\Delta$-category via $M\mapsto M\otimes L(p\lambda)$ for $\lambda\in\Z\Delta$.
Fix a projective $\Z\Delta$-generator $P$ of $\mathcal{C}_{\Coeff,0}$ and set $\mathcal{E} = \bigoplus_{\lambda\in\Z\Delta}\Hom_{\mathcal{C}_{\Coeff,0}}(P,P\otimes L(p\lambda))$.
This is a $\Z\Delta$-graded algebra and $\mathcal{C}_{\Coeff,0}\ni M\mapsto \bigoplus_{\lambda\in\Z\Delta}\Hom(P,M\otimes L(p\lambda))$ gives an equivalence of categories between $\mathcal{C}_{\Coeff,0}$ and the category of finitely generated $\Z\Delta$-graded right $\mathcal{E}$-modules~\cite[E.4 Proposition]{MR1272539}.
Denote the category of finitely-generated $\Z\Delta$-graded right $\mathcal{E}$-modules by $\Mod_{\Z\Delta}(\mathcal{E})$ and the projective objects in $\Mod_{\Z\Delta}(\mathcal{E})$ by $\Proj_{\Z\Delta}(\mathcal{E})$.

\begin{lem}
We have $(Q*B)\otimes L(p\lambda)\simeq (Q\otimes L(p\lambda))*B$ for $Q\in \Proj(\mathcal{C}_{\Coeff,0})$, $B\in\Sbimod$ and $\lambda\in\Z\Delta$.
\end{lem}
\begin{proof}
Let $\lambda\in\Z\Delta$.
Then we have a functor $T_\lambda$ (resp.\ $T_{\AJS,\lambda}$) on $\mathcal{K}_P$ (resp.\ $\mathcal{K}_{\AJS,P}$) defined as follows.
\begin{itemize}
\item For $M\in \mathcal{K}_P$, $T_{\lambda}(M) = M$ and $T_{\lambda}(M)^{\emptyset}_{A} = M^{\emptyset}_{A + \lambda}$.
\item For $\mathcal{M}\in \mathcal{K}_{\AJS}$, $T_{\AJS,\lambda}(\mathcal{M})(A) = \mathcal{M}(A + \lambda)$ and $T_{\AJS,\lambda}(\mathcal{M})(A,\alpha) = \mathcal{M}(A + \lambda,\alpha)$.
\end{itemize}
Since these functors are $S$-linear, they give functors on $\Coeff\otimes_{S}\mathcal{K}_{P}$ and $\Coeff\otimes_{S}\mathcal{K}_{\AJS,P}$, respectively.
These functors give structures of $\Z\Delta$-category on each category.
It is easy to see that equivalences $\Coeff\otimes_{S}\mathcal{K}_{P}\simeq \Coeff\otimes_{S}\mathcal{K}_{\AJS,P}\simeq\Proj(\mathcal{C}_{\Coeff,0})$ are $\Z\Delta$-functor.
Therefore it is sufficient to prove $T_{\lambda}(M*B)\simeq T_{\lambda}(M)*B$ for $M\in \mathcal{K}_{P}$ and $B\in \Sbimod$.
This follows from the definition.
\end{proof}

Therefore the action of $B\in\Sbimod$ on $\Proj(\mathcal{C}_{\Coeff,0})$ is compatible with the $\Z\Delta$-category structure and therefore it gives an action on $\Proj_{\Z\Delta}(\mathcal{E})$.
We denote this action again by $M\mapsto M*B$.
For each $B\in \Sbimod$, we define $\mathcal{E}(B)$ by $\mathcal{E}(B) = \bigoplus_{\lambda\in \Z\Delta}\Hom(P,(P*B)\otimes L(p\lambda))$.
This is a $\Z\Delta$-graded $\mathcal{E}$-bimodule.

\begin{lem}
Let $Q$ be a projective finitely generated $\Z\Delta$-graded $\mathcal{E}$-module.
Then $Q\otimes_{\mathcal{E}}\mathcal{E}(B)\simeq Q*B$.
\end{lem}
\begin{proof}
Denote $\nu$-th graded piece of $Q$ by $Q_\nu$ where $\nu\in\Z\Delta$.
Let $p\in Q_{\nu}$ and denote the corresponding element in $\Hom_{\Mod_{\Z\Delta}(\mathcal{E})}(\mathcal{E},Q(\nu))$ by $\varphi_p$.
Here $(\nu)$ is the shift of the grading.
Then $\varphi_p*B$ gives $\mathcal{E}*B\to Q(\nu)*B$.
By the definition, $\mathcal{E}*B = \mathcal{E}(B)$.
Therefore for $m\in \mathcal{E}(B)$, we have $\varphi_p(m)\in Q(\nu)*B \simeq (Q*B)(\nu)$.
Hence we get $Q\otimes_{\mathcal{E}}\mathcal{E}(B)\to Q*B$ by $p\otimes m\mapsto \varphi_p(m)$.
This is an isomorphism if $Q = \mathcal{E}$, hence it is an isomorphism for any $Q\in\Proj_{\Z\Delta}(\mathcal{E})$.
\end{proof}

Now for $\Z\Delta$-graded right $\mathcal{E}$-module $M$, put $M*B = M\otimes_{\mathcal{E}}\mathcal{E}(B)$.
By the above lemma, $\mathcal{E}(B_1)\otimes_{\mathcal{E}}\mathcal{E}(B_2)\simeq (\mathcal{E}*B_1*B_2) = \mathcal{E}*(B_1\otimes B_2)\simeq \mathcal{E}(B_1\otimes B_2)$.
Hence $(M * B_1)*B_2 = (M\otimes_{\mathcal{E}}\mathcal{E}(B_1))\otimes_{\mathcal{E}}(B_2)\simeq M\otimes_{\mathcal{E}}(\mathcal{E}(B_1)\otimes_{\mathcal{E}}\mathcal{E}(B_2))\simeq M\otimes_{\mathcal{E}}\mathcal{E}(B_1\otimes B_2) = M*(B_1\otimes B_2)$.
It is easy to see that this gives an action of $\Sbimod$ on $\Mod_{\Z\Delta}(\mathcal{E})$, hence on $\mathcal{C}_{\Coeff,0}$.

\subsection{Characters}
Any object $P\in \Proj(\mathcal{C}_{S,0})$ has a Verma flag.
We denote the multiplicity of $Z_{S}(w\cdot_{p}\lambda_0)$ in $P$ by $(P:Z_{S}(w\cdot_p\lambda_0))$.
The following lemma is obvious from the constructions.
\begin{lem}
Let $P\in \Proj(\mathcal{C}_{S,0})$ and $M\in \mathcal{K}_{P}$ such that $\mathcal{V}(P)\simeq \mathcal{F}(M)$.
Then we have $(P:Z_{S}(w\cdot_p\lambda_0)) = \rank(M_{\{wA_0\}})$.
\end{lem}

The projective module $P_S(\lambda)$ is characterized by 
\begin{itemize}
\item $P_S(\lambda)$ is indecomposable.
\item $(P_S(\lambda):Z_S(\lambda)) = 1$.
\item $(P_S(\lambda):Z_S(\mu)) = 0$ unless $\mu - \lambda\in\Z_{\ge 0}\Delta^+$.
\end{itemize}
The module $\mathcal{V}^{-1}(\mathcal{F}(Q(wA_0)))$ satisfies these conditions with $\lambda = w\cdot_p\lambda_0$ by the above lemma.
We get the following.
\begin{prop}
Let $w\in W'_{\aff}$.
Then $\mathcal{V}(P_S(w\cdot_p\lambda_0)) \simeq \mathcal{F}(Q(wA_0))$.
\end{prop}

The following corollary is obvious from the above proposition.
\begin{cor}\label{cor:multiplicity}
We have $[P_{\Coeff}(w\cdot_p\lambda_0):Z_{\Coeff}(v\cdot_p\lambda_0)] = \rank(Q(wA_0)_{\{vA_0\}})$.
\end{cor}

\subsection{Lusztig's conjecture}
For $B\in \Sbimod$ and $w\in W_{\aff}$,  we denote the image of $B\hookrightarrow B\otimes_R R^{\emptyset} = \bigoplus_{x\in W_{\aff}}B_x^{\emptyset}\twoheadrightarrow B_w^{\emptyset}$ by $B^w$.
Put $\ch(B) = \sum_{w\in W_{\aff}}v^{-\ell(w)}\grk(B^w)$.
Then $[B]\mapsto \ch(B)$ induces an isomorphism $[\Sbimod]\simeq \mathcal{H}$.
For each $w\in W_{\aff}$, there exists an indecomposable object $B(w)\in \Sbimod$ unique up to isomorphism such that $\ch(B(w)) \in H_w + \sum_{x < w}\Z[v,v^{-1}]H_x$.
We say that $B(w)$ satisfies the Soergel conjecture if $\ch(B(w))$ is a Kazhdan-Lusztig basis, namely $\ch(B(w))\in H_w + \sum_{x < w}v\Z[v]H_x$.
It is known that the Soergel conjecture is satisfied by any $B(w)$ over a characteristic zero field, therefore, for a fixed $w$, if $p$ is sufficiently large, $B(w)$ satisfies the Soergel conjecture (cf.~\cite{MR3245013}).
We fix $\lambda\in (\R\Delta)_{\integer}$ and $w\in W_{\aff}$ such that $A_{\lambda}^+w\in \Pi_{\lambda}$ here $A_{\lambda}^+$ is the maximal element in $W'_{\lambda}A_{\lambda}^-$.

\begin{lem}
Let $w_{\lambda}\in W_{\aff}$ such that $A_{\lambda}^+w_{\lambda} = A_{\lambda}^-$.
Then we have $S_{A_{\lambda}^+}*B(w_{\lambda})\simeq Q_{\lambda}(\ell(w_0))$.
\end{lem}
\begin{proof}
By the translation as in the proof of Lemma~\ref{lem:successive quotient of Q_lambda}, we may assume $\lambda = 0$.
Then $W_{\lambda}' = W_{\mathrm{f}}$ and it is generated by $S_{\aff}\cap W_{\mathrm{f}}$.
Moreover, the element $w_{\lambda}$ is equal to the longest element $w_0$.

It is sufficient to prove: $B(w_{0}) \simeq \{(z_w)\in R^{W_{\mathrm{f}}}\mid z_{wt}\equiv z_w\pmod{\alpha_t}\}(\ell(w_0))$ where $t$ runs through the set of reflections in $W_{\mathrm{f}}$ and $\alpha_t$ the corresponding element in $\Lambda_{\Coeff}$~\cite[2.1]{arXiv:1901.02336_accepted}.
Let $(G_{\mathbb{C}}^\vee,B_{\mathbb{C}}^\vee,T_{\mathbb{C}}^\vee)$ be the reductive group over $\mathbb{C}$, the Borel subgroup and the maximal torus with the root datum $(X^\vee,\Delta^\vee,X,\Delta)$ and the positive system $\Delta^+\subset \Delta$.
Then the category of $\Coeff$-coefficient parity $B^\vee_{\mathbb{C}}$-equivariant sheaves on $G^\vee_{\mathbb{C}}/B_{\mathbb{C}}^\vee$ is equivalent to the category of Soergel bimodules attached to $(W_{\mathrm{f}},X^\vee_{\Coeff})$~\cite{MR3805034}.
The object $B(w_0)$ corresponds to the indecomposable parity sheaf such that the restriction to the big cell $B_{\mathbb{C}}^\vee w_0B_{\mathbb{C}}^\vee/B_{\mathbb{C}}^\vee$ is $\Coeff_{B_{\mathbb{C}}^\vee w_0B_{\mathbb{C}}^\vee/B_{\mathbb{C}}^\vee}[\ell(w_0)]$.
It is obvious that the constant sheaf $\Coeff_{G^\vee_{\mathbb{C}}/B_{\mathbb{C}}^\vee}[\ell(w_0)]$ satisfies this condition and therefore the constant sheaf corresponds to $B(w_0)$.
As in \cite{MR3330913}, the corresponding Soergel bimodule is given as in the above.
\end{proof}

Recall that $w\in W_{\aff}$ and $\lambda\in(\R\Phi_{\integer})$ such that $A_{\lambda}^+w\in \Pi_{\lambda}$.

\begin{thm}
If $B(w)$ satisfies the Soergel conjecture, then $S_{A_{\lambda}^+}*B(w)\simeq Q(A_{\lambda}^+w)$.
\end{thm}
\begin{proof}
First we prove that $S_{A_{\lambda}^+}*B(w)\in \mathcal{K}_P$.
By the translation as in Lemma~\ref{lem:successive quotient of Q_lambda}, we may assume $\lambda = 0$.
Then $W_{\lambda}' = W_{\mathrm{f}}$ and this is generated by $W_{\mathrm{f}}\cap S_{\aff}$.
We have $sw < w$ for any $s\in W_{\mathrm{f}}\cap S_{\aff}$.
Therefore $H_s\ch(B(w)) = v^{-1}\ch(B(w))$ by \cite[Lemma~4.3]{MR3611719}.
Hence $H_x\ch(B(w)) = v^{-\ell(x)}\ch(B(w))$ for any $x\in W_{\mathrm{f}}$.
Therefore, since the coefficient of $H_{w_0}$ in $\ch(B(w_{0}))$ is $1$, we have $\ch(B(w_{0})) = \sum_{y\in W_{\mathrm{f}}}v^{\ell(w_{0}) - \ell(y)}H_x$.
Therefore we have $\ch(B(w_{0})\otimes B(w)) = \sum_{y\in W_{\mathrm{f}}}v^{\ell(w_{0}) - 2\ell(y)}\ch(B(w))$.
Hence we get $B(w_{0})\otimes B(w)\simeq \bigoplus_{y\in W_{\mathrm{f}}}B(w)(\ell(w_{0}) - 2\ell(y))$.
Therefore, up to shift, $S_{A_{0}^+}*B(w)$ is a direct summand of $S_{A_{0}^+}*(B(w_{0})\otimes B(w))\simeq Q_{0}(\ell(w_{0}))*B(w)\in \mathcal{K}_P$.
Hence $S_{A_{0}^+}*B(w)\in \mathcal{K}_P$.

We return to the proof of the theorem.
By \cite[Theorem~5.2]{MR591724}, $\ch(S_{A_{\lambda}^+}*B(w)) = A_{\lambda}^+\ch(B(w))$ is described by periodic Kazhdan-Lusztig polynomials, namely we have $A_{\lambda}^+\ch(B(w)) = v^{-n}\underline{P}_{A_0}$ for some $A_0\in \mathcal{A}$ and $n\in \Z$, here $\underline{P}_{A'}\in \mathcal{P}^0$ is the element given in \cite[Proposition~4.16]{MR1444322}.
We know $A_{\lambda}^+\ch(B(w))\in A_{\lambda}^+w + \sum_{A' > A_{\lambda}^+w}\Z[v,v^{-1}]A'$.
Comparing with \cite[Lemma~4.21]{MR1444322}, we have $n = \ell(w_0)$ and $\ch(S_{A_{\lambda}^+}*B(w))\in A_{\lambda}^+w + \sum_{A' > A_{\lambda}^+w}v^{- 1}\Z[v^{-1}]A'$.
By the self-duality of $\underline{P}_{A_0}$, we have $\overline{\ch(S_{A_{\lambda}^+}*B(w))} = v^{\ell(w_0)}\underline{P}_{A_0}\in v^{2\ell(w_0)}A_{\lambda}^+w + \sum_{A' > A_{\lambda}^+w}v^{2\ell(w_0)- 1}\Z[v^{-1}]A'$.
Therefore by Theorem~\ref{thm:hom formula}, we have
\[
\grk\Hom^{\bullet}_{\mathcal{K}}(S_{A_{\lambda}^+}*B(w),S_{A_{\lambda}^+}*B(w))
\in 1 + v^{-2}\Z[v^{-1}].
\]
Hence $\End_{\mathcal{K}}(S_{A_{\lambda}^+}*B(w))$ is one-dimensional and has only trivial idempotent.
Therefore $S_{A_{\lambda}^+}*B(w)$ is indecomposable.
Since $Q({A_{\lambda}^+w})$ is a direct summand of $S_{A_{\lambda}^+}*B(w)$, we get the theorem.
\end{proof}

From the above theorem and Corollary~\ref{cor:multiplicity}, the multiplicity of the baby Verma modules in the projective cover of an irreducible module is given by the value at $1$ of the Kazhdan-Lusztig polynomial.
Hence the Lusztig's conjecture holds for sufficiently large $p$.


\begin{thebibliography}{AMRW19}
\bibitem[Abe19]{arXiv:1901.02336_accepted}
Noriyuki Abe, \textit{{A} bimodule description of the {H}ecke category}, to appear in Compos. Math. 

\bibitem[Abe20a]{arXiv:2012.09414}
Noriyuki Abe, \textit{{A} homomorphism between {B}ott-{S}amelson bimodules}, arXiv:2012.09414. 

\bibitem[Abe20b]{arXiv:2004.09014}
Noriyuki Abe, \textit{{O}n singular {S}oergel bimodules}, arXiv:2004.09014. 

\bibitem[AJS94]{MR1272539}
H.~H. Andersen, J.~C. Jantzen, and W.~Soergel, \textit{Representations of quantum groups at a {$p$}-th root of unity and of semisimple groups in characteristic {$p$}: independence of {$p$}}, Ast\'erisque (1994), no.~220, 321. 

\bibitem[AMRW19]{MR3868004}
Pramod~N. Achar, Shotaro Makisumi, Simon Riche, and Geordie Williamson, \textit{Koszul duality for {K}ac--{M}oody groups and characters of tilting modules}, J. Amer. Math. Soc. \textbf{32} (2019), no.~1, 261--310. 

\bibitem[BR20]{arXiv:2009.10587}
Roman Bezrukavnikov and Simon Riche, \textit{{H}ecke action on the principal block}, arXiv:2009.10587. 

\bibitem[EW14]{MR3245013}
Ben Elias and Geordie Williamson, \textit{The {H}odge theory of {S}oergel bimodules}, Ann. of Math. (2) \textbf{180} (2014), no.~3, 1089--1136. 

\bibitem[EW16]{MR3555156}
Ben Elias and Geordie Williamson, \textit{Soergel calculus}, Represent. Theory \textbf{20} (2016), 295--374. 

\bibitem[Fie11]{MR2726602}
Peter Fiebig, \textit{Sheaves on affine {S}chubert varieties, modular representations, and {L}usztig's conjecture}, J. Amer. Math. Soc. \textbf{24} (2011), no.~1, 133--181. 

\bibitem[Fie12]{MR2999126}
Peter Fiebig, \textit{An upper bound on the exceptional characteristics for {L}usztig's character formula}, J. Reine Angew. Math. \textbf{673} (2012), 1--31. 

\bibitem[FL15]{arXiv:1504.01699}
Peter Fiebig and Martina Lanini, \textit{{S}heaves on the alcoves {I}: {P}rojectivity and wall crossing functors}, arXiv:1504.01699. 

\bibitem[FW14]{MR3330913}
Peter Fiebig and Geordie Williamson, \textit{Parity sheaves, moment graphs and the {$p$}-smooth locus of {S}chubert varieties}, Ann. Inst. Fourier (Grenoble) \textbf{64} (2014), no.~2, 489--536. 

\bibitem[JW17]{MR3611719}
Lars~Thorge Jensen and Geordie Williamson, \textit{The {$p$}-canonical basis for {H}ecke algebras}, Categorification and higher representation theory, Contemp. Math., vol. 683, Amer. Math. Soc., Providence, RI, 2017, pp.~333--361. 

\bibitem[KL93]{MR1186962}
D.~Kazhdan and G.~Lusztig, \textit{Tensor structures arising from affine {L}ie algebras. {I}, {II}}, J. Amer. Math. Soc. \textbf{6} (1993), no.~4, 905--947, 949--1011. 

\bibitem[KL94a]{MR1239506}
D.~Kazhdan and G.~Lusztig, \textit{Tensor structures arising from affine {L}ie algebras. {III}}, J. Amer. Math. Soc. \textbf{7} (1994), no.~2, 335--381. 

\bibitem[KL94b]{MR1239507}
D.~Kazhdan and G.~Lusztig, \textit{Tensor structures arising from affine {L}ie algebras. {IV}}, J. Amer. Math. Soc. \textbf{7} (1994), no.~2, 383--453. 

\bibitem[KT95]{MR1317626}
Masaki Kashiwara and Toshiyuki Tanisaki, \textit{Kazhdan-{L}usztig conjecture for affine {L}ie algebras with negative level}, Duke Math. J. \textbf{77} (1995), no.~1, 21--62. 

\bibitem[KT96]{MR1408544}
Masaki Kashiwara and Toshiyuki Tanisaki, \textit{Kazhdan-{L}usztig conjecture for affine {L}ie algebras with negative level. {II}. {N}onintegral case}, Duke Math. J. \textbf{84} (1996), no.~3, 771--813. 

\bibitem[Lib08]{MR2441994}
Nicolas Libedinsky, \textit{Sur la cat\'{e}gorie des bimodules de {S}oergel}, J. Algebra \textbf{320} (2008), no.~7, 2675--2694. 

\bibitem[Lus80]{MR591724}
George Lusztig, \textit{Hecke algebras and {J}antzen's generic decomposition patterns}, Adv. in Math. \textbf{37} (1980), no.~2, 121--164. 

\bibitem[RW18]{MR3805034}
Simon Riche and Geordie Williamson, \textit{Tilting modules and the {$p$}-canonical basis}, Ast\'{e}risque (2018), no.~397, ix+184. 

\bibitem[RW20]{arXiv:2003.08522}
Simon Riche and Geordie Williamson, \textit{{S}mith-{T}reumann theory and the linkage principle}, arXiv:2003.08522. 

\bibitem[Sob20]{MR4092982}
Paul Sobaje, \textit{On character formulas for simple and tilting modules}, Adv. Math. \textbf{369} (2020), 107172, 8. 

\bibitem[Soe97]{MR1444322}
Wolfgang Soergel, \textit{Kazhdan-{L}usztig polynomials and a combinatoric for tilting modules}, Represent. Theory \textbf{1} (1997), 83--114 (electronic). 

\bibitem[Wil17]{MR3671935}
Geordie Williamson, \textit{Schubert calculus and torsion explosion}, J. Amer. Math. Soc. \textbf{30} (2017), no.~4, 1023--1046, With a joint appendix with Alex Kontorovich and Peter J. McNamara. 

\end{thebibliography}
\end{document}